\documentclass[11pt,leqno]{article}

\usepackage[colorlinks=true,urlcolor=blue,citecolor=blue,linkcolor=blue]{hyperref}
\usepackage[left=2.9cm,right=2.9cm,top=2.9cm,bottom=2.9cm]{geometry}

\usepackage{todonotes}
\usepackage{amsthm}
\usepackage{amsmath,amssymb, amsfonts}
\usepackage{txfonts}
\usepackage{tikz}
\usetikzlibrary{matrix,arrows,patterns,calc,through,backgrounds,fadings, decorations}

\usepackage{hyperref}
\newcommand{\footremember}[2]{
   \footnote{#2}
    \newcounter{#1}
    \setcounter{#1}{\value{footnote}}
}
\newcommand{\footrecall}[1]{
    \footnotemark[\value{#1}]
}

\usepackage{fancyhdr}
 \def\addsymbol #1: #2#3{$#1$ \> \parbox{5in}{#2 \dotfill \pageref{#3}}\\}

 \newtheorem{thm}{Theorem}[section]
 \newtheorem{theorem}[thm]{Theorem}
 \newtheorem{cor}[thm]{Corollary}
 \newtheorem{lem}[thm]{Lemma}
 \newtheorem{prop}[thm]{Proposition}
 \newtheorem{conj}[thm]{Conjecture}
 \newtheorem*{thm*}{Theorem}

 \newtheorem{conjecture}[thm]{Conjecture}
 
 \theoremstyle{definition}
 
 \theoremstyle{remark}
 \newtheorem{rem}[thm]{Remark}
 
 \numberwithin{equation}{section}
\newcommand{\kommentar}[1]{}

\let \nc \newcommand
\let \rnc \renewcommand

 \rnc \phi \varphi \rnc \epsilon \varepsilon

\nc {\R}{{\mathbb{R}}} \nc {\C}{{\mathbb{C}}} \nc {\Z}{{\mathbb{Z}}}\nc {\Q}{{\mathbb{Q}}}

\nc {\bd}{\begin{description}} \nc {\ed}{\end{description}} \nc {\bi}{\begin{itemize}} \nc {\ei}{\end{itemize}}
\nc {\be}{\begin{enumerate}} \nc {\ee}{\end{enumerate}} \nc {\bdm}{\begin{displaymath}} \nc
{\edm}{\end{displaymath}} \nc {\bea}{\begin{eqnarray*}} \nc {\eea}{\end{eqnarray*}} \nc {\baa}{\begin{alignat*}}
\nc {\eaa}{\end{alignat*}} \nc {\bsp}{\begin{split}} \nc {\esp}{\end{split}} \nc {\beq}{\begin{equation}} \nc
{\eeq}{\end{equation}} \nc {\btab}{\begin{tabular}} \nc {\etab}{\end{tabular}} \nc {\ba}{\begin{array}} \nc
{\ea}{\end{array}}

\newcommand{\relmiddle}[1]{\mathrel{}\middle#1\mathrel{}}

\newcommand{\fL}{{L}}

\newcommand{\cA}{{\mathcal A}}
\newcommand{\cC}{{\mathcal C}}

\newcommand{\cM}{{\mathcal M}}
\newcommand{\cG}{{\mathcal G}}

\newcommand{\cL}{{\mathcal L}}

\newcommand{\cH}{{\mathcal H}}
\newcommand{\cP}{{\mathcal P}}

\newcommand{\cQ}{{\mathcal Q}}

\newcommand{\fg}{\ensuremath{\mathfrak{g}}}

\newcommand{\calM}{{\mathcal M}}

\def\co{\colon\thinspace}

\nc{\LOT}[1]{L^2(\Omega^{0+1+2}(T;#1))} \nc{\LO}[2]{L^2(\Omega^{0+1+2}(#1;#2))} \nc{\HT}[1]{{H^{0+1+2}(T;#1)}}
\nc{\cHT}[1]{{\cH^{0+1+2}(T;#1)}} \nc{\stor}{D^2\times S^1}

\nc{\pardeg}{\operatorname{\mathrm{pardeg}}}
\nc{\parslope}{\operatorname{\mathrm{par}\mu}}
\nc{\Det}{\mathrm{Det}}
\nc{\U}{\operatorname{\mathrm{U}}}
\nc{\SO}{\operatorname{\mathrm{SO}}}
\nc{\SU}{\operatorname{\mathrm{SU}}} \nc{\Om}{\Omega}
\nc{\dist}{{\mathrm{dist}}}
\renewcommand{\deg}{\operatorname{\mathrm{deg}}}
\nc{\homeo}{\approx} \nc{\im}{\operatorname{\mathrm{im}}}
\nc{\diag}{\operatorname{\mathrm{diag}}}
\nc{\Crit}{\operatorname{\mathrm{Crit}}}
\nc{\grad}{\operatorname{\mathrm{grad}}}
\nc{\Hess}{\operatorname{\mathrm{Hess}}}
\nc{\Ad}{\operatorname{\mathrm{Ad}}}
\rnc{\O}{\operatorname{\mathrm{O}}}
\nc{\PD}{\operatorname{\mathrm{PD}}}
\nc{\ad}{\operatorname{\mathrm{ad}}}
\nc{\vol}{\operatorname{\mathrm{vol}}}
\newcommand{\rk}{\operatorname{\mathrm{rk}}}
\nc{\hol}{\operatorname{\mathrm{hol}}}
\nc{\re}{\operatorname{\mathrm{Re}}}
\nc{\Id}{\operatorname{\mathrm{Id}}}
\nc{\Mas}{\operatorname{\mathrm{Mas}}}
\nc{\SF}{\operatorname{\mathrm{SF}}}
\rnc{\ker}{\operatorname{\mathrm{ker}}}
\nc{\tr}{\operatorname{\mathrm{tr}}}
\nc{\sign}{\operatorname{\mathrm{sign}}}
\nc{\Spec}{\operatorname{\mathrm{Spec}}}\nc{\coker}{\operatorname{\mathrm{coker}}}
\rnc{\hom}{\operatorname{\mathrm{Hom}}}
\nc{\ob}{\operatorname{\mathrm{Ob}}}
\nc{\ch}{\operatorname{\mathrm{ch}}}
\nc{\Lie}{\operatorname{\mathrm{Lie}}}
\nc{\End}{\operatorname{\mathrm{End}}}
\nc{\Aut}{\operatorname{\mathrm{Aut}}} \nc{\lat}{(\frac{1}{2}\Z)^2}
\nc{\hz}{\tfrac{1}{2}\Z} \nc{\la}{\langle} \nc{\ra}{\rangle} \nc
{\Ra}{\Rightarrow} \nc {\La}{\Leftarrow} \nc {\lla}{\longleftarrow}
\nc {\nach}{\rightarrow} \nc {\equ}{\Leftrightarrow} \nc
{\gdw}{\Longleftrightarrow} \nc {\lra}{\longrightarrow} \nc
{\Lra}{\Longrightarrow} \nc {\lmt}{\longmapsto} \nc
{\quer}{\overline} \nc {\tensor}{\otimes} \nc {\Rt}{\widetilde\R}
\nc {\contract}{\lrcorner} \nc{\sgn}{\operatorname{\mathrm{sgn}}}
\newcommand{\kom}[1]{}
\newcommand{\isom}{\cong}

\nc{\Cl}{\operatorname{\rm Cl}}
\nc{\CS}{\operatorname{\mathrm{CS}}}
\nc{\Ch}{\operatorname{\mathrm{Ch}}}
\nc{\Td}{\operatorname{\mathrm{Td}}}
\nc{\Rank}{\operatorname{\mathrm{Rank}}}
\nc{\GL}{\operatorname{\mathrm{GL}}}
\nc{\proj}{\operatorname{\mathrm{proj}}}

\newcommand{\bSigma}{\overline{\Sigma}}

\newcommand{\Diff}{\mathrm{Diff}}

\newcommand{\Sib}{\bSigma}

\newcommand{\frakg}{\mathfrak{g}}
\DeclareMathOperator{\Hom}{Hom}

\AtEndDocument{\bigskip{\footnotesize

\hfill\begin{minipage}{\dimexpr\textwidth-.5cm}

(J.E.~Andersen) \textsc{Centre for Quantum Geometry of Moduli Spaces (QGM),
Aarhus University,
Ny Munkegade 118, bldg. 1530,
DK-8000 Aarhus C,
Denmark}\par
\textit{E-mail address:}\texttt{andersen@qgm.au.dk}\par
  \addvspace{\medskipamount}
  
(B. Himpel)
\textsc{Centre for Quantum Geometry of Moduli Spaces (QGM),
Aarhus University,
Ny Munkegade 118, bldg. 1530,
DK-8000 Aarhus C,
Denmark}\par
\textit{E-mail address}: \texttt{himpel@qgm.au.dk,  himpel@uni-bonn.de }\par 
  \addvspace{\medskipamount}

(S.F.~J\o rgensen) \textsc{Department of Mathematics, Box 480, Uppsala University, SE-75106 Uppsala, Sweden}\par
\textit{E-mail address:} \texttt{s@fuglede.dk
}\par
  \addvspace{\medskipamount}

(J. Martens) \textsc{School of Mathematics and Maxwell Institute for Mathematical Sciences, The University of Edinburgh, Peter Guthrie Tait Road, Edinburgh,
EH9 3FD, United Kingdom}\par
\textit{E-mail address:} \texttt{johan.martens@ed.ac.uk}\par
  \addvspace{\medskipamount}

(B.~McLellan) \textsc{Department of Mathematics, Harvard University, One Oxford Street, Cambridge MA 02138,
 USA}\par
\textit{E-mail address:} \texttt{mclellan@math.harvard.edu}\par 

\xdef\tpd{\the\prevdepth}
\end{minipage}

}}

\begin{document}

\title{The Witten--Reshetikhin--Turaev invariant for links in finite order mapping tori I}

\author{J{\o}rgen Ellegaard Andersen\footremember{qgm}{Supported in part
by the Danish National Research Foundation grant DNRF95 (Centre for
Quantum Geometry of Moduli Spaces -- QGM).}
 \footremember{itgp}{Supported in part by the European Science Foundation Network `Interactions of Low-Dimensional Topology and Geometry with Mathematical Physics' (ITGP).} \and Benjamin Himpel\footrecall{qgm} \and
S{\o}ren Fuglede
J{\o}rgensen\footrecall{qgm} \ 
\footremember{sweden}{Supported in part  by the Swedish Research Council
Grant 621--2011--3629.
} \and Johan Martens\footrecall{qgm} \ 
\footrecall{itgp} \and Brendan McLellan\footrecall{qgm} }

\date{\today}

\maketitle

\begin{abstract}
We state Asymptotic Expansion and Growth Rate conjectures for the Witten--Reshetikhin--Turaev invariants of arbitrary framed links in $3$-manifolds, and we prove these conjectures for the natural links in mapping tori of finite-order automorphisms of marked surfaces.  Our approach is based upon geometric quantisation of the moduli space of parabolic bundles on the surface, which we show coincides with the construction of the Witten--Reshetikhin--Turaev invariants using conformal field theory, as was recently completed by Andersen and Ueno.
\end{abstract}

\section{Introduction}

In this paper we study the asymptotic expansion of the Witten--Reshetikhin--Turaev (WRT) invariants of certain 3-manifolds with links, building on the work  
\cite{andersen95, andersen-himpel2011}, 
which also used the geometric construction of the WRT-TQFT via the geometric quantisation of moduli spaces of flat connections on surfaces as first considered by Axelrod--Della Pietra--Witten \cite{axelrod-dellapietra-witten91}, Hitchin \cite{hitchin90} and further explored by the first named author to prove asymptotic faithfulness \cite{andersen2006_Asymptoticfaithfulness}. For references concerning the study of the large level asymptotics of the WRT quantum invariants of closed  3-manifolds see the references in \cite{andersen95}.  Let us here first 
present a generalisation of the Asymptotic Expansion Conjecture to pairs consisting of a general closed oriented 3-manifold together with an embedded oriented framed link, labelled by level dependent labels.

\subsection*{The Asymptotic Expansion and Growth Rate Conjectures}
The quantum invariants and their associated Topological Quantum Field Theories were proposed in Witten's seminal paper \cite{witten-jones} on quantum Chern--Simons theory with a general compact simple simply-connected Lie group $K$, and subsequently constructed by Reshetikhin and Turaev  \cite{RT1,RT2,Turaev} for $K = {\rm SU}(2)$ and then for $K={\rm SU}(N)$ by Wenzl and Turaev in \cite{WenzlTuraev1,WenzlTuraev2}. These TQFTs were also constructed from skein theory by Blanchet, Habegger, Masbaum and Vogel in \cite{BHMV1, BHMV2} for $K = {\rm SU}(2)$ and for $K={\rm SU}(N)$ in \cite{Blanchet}. We will denote these TQFTs for $K={\rm SU}(N)$ by $Z_N^{(k)}$. The WRT-TQFT associated to a general simple simply-connected Lie group $K$ will be denoted by $Z_K^{(k)}$, 
 e.g. $Z_N^{(k)} = Z_{{\rm SU}(N)}^{(k)}$.

The label set of the WRT-TQFT $Z_K^{(k)}$ theory is given as 
\begin{equation}\label{labelset}
 \Lambda^{(k)}_K = \{\, \lambda \in 
P_+ \, | \,
   0\le\langle\theta, \lambda\rangle \le k \, \},
\end{equation}
where $P_+$ is the set of dominant integral weights of $\mathfrak{k}$, the Lie algebra of $K$. Here $\langle\phantom{a}, \phantom{a}\rangle$ is the normalized Cartan--Killing form defined to be a constant multiple of 
the Cartan--Killing form such that
$  
 \langle\theta ,\, \theta\rangle = 2, 
$
for the longest root $\theta$ of $\mathfrak{k}$.  We will use $\langle\phantom{a}, \phantom{a}\rangle$ at various places throughout the text to identify weights and coweights.

Let $X$ be an oriented closed 3-manifold and let $L$ be a framed link contained in $X$. For notational purposes pick an ordering of the components of $L = L_1 \cup \ldots \cup L_n$. Let $\overline{\lambda}^{(k)} = (\lambda^{(k)}_1,\ldots, \lambda^{(k)}_n)$, be a labelling of the components of $L$ which is $k$-dependent (possibly only for $k$ forming a sub-sequence of $\mathbb{N}$). 
In fact, throughout this paper we will restrict to the simple example $\lambda^{(k)}_i = \lambda_i s$ for $k$-independent $\lambda_i\in \Lambda_N^{(k_0)}$, with $k=sk_0$ for some fixed $k_0$.  After identifying the $\lambda_i$ with elements in the Cartan algebra of $\mathfrak{su}(N)$ using $\langle\phantom{a}, \phantom{a}\rangle$ we denote the conjugacy class in $\operatorname{SU}(N)$ containing $e^{\lambda_i}$ as $c_i$.

We conjecture that the asymptotic expansion of the Witten--Reshetikhin--Turaev invariant of $(X,\fL,\allowbreak \overline{\lambda}^{(k)})$ associated with the quantum group $U_q(\fg)$ at the root of unity 
 $q=e^{2\pi i/\tilde k}$, 
 ${\tilde k} =k+h^{\vee}$, $k$ being the level, $h^{\vee}$ the dual Coxeter number, and $\mathfrak{g}$ the Lie algebra of the complex reductive group $K_{\mathbb{C}}$ has the following form.

\begin{conj}[Asymptotic Expansion Conjecture for triples $(X,\fL, c)$]\label{conj:AEC}
There exist functions of $\overline{c}=(c_1,\dots,c_n)$ (depending on $X,L$), 
$d_{j,r}(\overline{c}) \in \Q$ and
$b_{j,r}(\overline{c}) \in \C$ for $r=1,\ldots u_j(\overline{c})$, $j=1, \ldots, m(\overline{c})$, and $a_{j,r}^p(\overline{c}) \in  \C$ for
$j=1, \ldots, v(\overline{c})$,\, $p=1,2,\ldots$, such that the asymptotic
expansion of $Z^{(k)}_K(X, L, \overline{\lambda}^{(k)})$ in the limit $k \rightarrow \infty$ is
given by
\[Z^{(k)}_K(X, L, \overline{\lambda}^{(k)}) \sim \sum_{j=1}^{v(\overline{c})} e^{2\pi i {\tilde k} q_j} \sum_{r=1}^{u_j(\overline{c})}{\tilde k}^{d_{j,r}(\overline{c})}
b_{j,r}(\overline{c}) \left( 1 + \sum_{p=1}^\infty a_{j,r}^p(\overline{c}) {\tilde k}^{-p/2}\right), \]
where $q_1, \ldots, q_{v(\overline{c})}$ are the finitely many
different values of the Chern--Simons functional on the space
of flat $K$-connections on $X\setminus L$ with meridional holonomy around $L_i$ contained in $c_i$, $i=1, \ldots, n$.
\end{conj}

Here {\bf $\sim$} means {\bf asymptotic expansion} in the
Poincar\'{e} sense, which means the following: let
\[d(\overline{c}) = \max_{j,r} \{d_{j,r}(\overline{c})\}.\]
Then for any non-negative integer $P$, there is a $C_P \in \R$
such that
\[\left\lvert Z^{(k)}_K(X, L, \overline{\lambda}^{(k)}) - \sum_{j=1}^{v(\overline{c})} e^{2\pi i {\tilde k} q_j} \sum_{r=1}^{u_j(\overline{c})}{\tilde k}^{d_{j,r}(\overline{c})}
b_{j,r}(\overline{c}) \left( 1 + \sum_{p=1}^P a_{j,r}^p(\overline{c}) {\tilde k}^{-p/2}\right) \right\rvert \leq C_P {\tilde k}^{d(\overline{c})-P-1}\]
for all levels $k$ that occur. Of course such a condition only puts
limits on the large $k$ behaviour of $Z^{(k)}_K(X, L, \overline{\lambda}^{(k)})$.

Note that a priori, the Chern--Simons functional of a manifold with boundary defines a section over the relevant moduli space of flat connections on the boundary. However, by specifying holonomy conditions on the boundary as in Conjecture~\ref{conj:AEC}, the framing structure of the link allows us to make sense of the Chern--Simons functional as real valued modulo integers. This is discussed in more detail in Appendix~\ref{chernsimonsboundaryappendix}.

Let us introduce 
$$d_j(\overline{c}) = \max_{r} d_{j,r}(\overline{c}).$$

For a flat $K$-connection $A$ on the $3$-manifold $X\setminus L$ with holonomy around $L_i$ given by $c_i$, $i=1, \ldots n$,
denote by $h_{A}^{i}$ the dimension of
the $i$-th $A$-twisted cohomology groups of $X\setminus L$ with Lie algebra coefficients. In analogy with the growth rate conjecture stated in \cite{andersen95} we offer the following conjecture for a topological formula for $d_j(\overline{c})$.

\begin{conjecture}[The Growth Rate Conjecture]\label{conj:GR}
Let $\mathcal{M}_{X,L,\overline{c}}^{q_j}$ be the union of components of the moduli space
of flat $K$-connections on $X\setminus L$, with holonomy around $L_i$ given by $c_i$, $i=1, \ldots n$, and which have Chern--Simons
value $q_{j}$. Then
$$
d_{j}(\overline{c}) = \frac{1}{2} \max_{\nabla \in \mathcal{M}_{X,L,\overline{c}}^{q_j}}
            \left( h_{\nabla,\operatorname{par}}^{1} - h_{\nabla}^{0} \right),
$$
where $\max$ here means the maximum value
$h_{\nabla,\operatorname{par}}^{1} - h_{\nabla}^{0}$ attains on a non-empty Zariski open subset
of $\mathcal{M}_{X,L,\overline{c}}^{q_j}$ on which $h_{\nabla,\operatorname{par}}^{1} - h_{\nabla}^{0}$ is constant.
\end{conjecture}
Here $h^0_{\nabla}$ is the dimension of the $0$-th cohomology with twisted coefficients for the local system induced by the flat connection $\nabla$ on the adjoint bundle, and, following \cite{biswas}, we define $h_{\nabla,\operatorname{par}}^{1}$ to be the dimension  of the image of the $1$-st cohomology with twisted coefficients and compact support of this local system in the usual $1$-st cohomology with twisted coefficients.

\subsection*{Links in mapping tori}

We will in this paper prove these conjectures for $Z^{(k)}_N$ in the situation where  the 3-manifold $X$ admits the structure of a finite order mapping torus over a closed oriented surface $\Sigma$ of genus $\geq 2$, and the oriented framed link $L$ is induced from marked points on the surface in the following way.  Let $f \co \Sigma \to\Sigma$ be a diffeomorphism of $\Sigma$. The mapping torus $X=\Sigma_f$ 
is defined as
\beq \label{finordX}
X = (\Sigma \times I) / [(x,1) \sim (f(x),0)]
\eeq
with the orientation on $X$ given by the product orientation, and with the
standard orientation on the unit interval $I=[0,1]$.  We consider special links $\fL$ that wrap the natural fibre direction in $X$.  Let $\cP\subset \Sigma$ denote a finite $f$-invariant subset of $\Sigma$, i.e. $f(\cP)=\cP$.  Then,
\beq
\fL=(\cP\times I)/[(x,1) \sim (f(x),0)].
\eeq
Given a labelling $\overline{\lambda}$ of $\fL$ we get induced a labelling of the points $\cP$ on $\Sigma$.

The two dimensional part of the  WRT-TQFT we are considering is a modular functor. For the axioms of modular functor see e.g. \cite{Turaev}, \cite{Walker}, \cite{AU2, AU4, AU5}. A modular functor is a functor from the category of labelled marked surfaces to the category of finite dimensional vector spaces.

A marked surface is the following datum: $\bSigma= (\Sigma, \cP, V, W)$, where $\Sigma$ is a closed oriented surface, $\cP=\{p_1,\dots,p_n\}$ is a set of points on $\Sigma$, $V$ is a set of `projective' tangent vectors at the marked points (i.e. non-zero elements of $T_{p_i}\Sigma/\mathbb{R}_+$) and $W$ is a Lagrangian subspace of the first real cohomology of $\Sigma$. A labelling of $\bSigma$ is a map $\overline{\lambda}: \cP \rightarrow \Lambda: p_i\mapsto \lambda_i$, where $\Lambda$ is a finite label set specific to the modular functor in question.  
From now on we will assume that $f$ is an automorphism of the labelled marked surface $(\bSigma, \overline{\lambda})$.

Note that the link inside the mapping torus of an automorphism of a marked surface naturally inherits a framing.  
From the general axioms for a TQFT we have 
\beq\label{general}Z^{(k)}_K(\Sigma_f, L, \overline{\lambda})=\tr\left(Z^{(k)}_K(f):Z^{k}_K(\bSigma,\overline{\lambda})\rightarrow Z^{(k)}_K(\bSigma,\overline{\lambda})\right).\eeq

We shall in this paper use the gauge theory construction of the vector space $Z_{K}^{(k)}(\bSigma, \overline{\lambda})$ that the WRT-TQFT $Z_{K}^{(k)}$ associates to a labelled marked surface $(\bSigma, \overline{\lambda})$. Let us from now on in this paper specialise to the case $K={\rm SU}(N)$.  This allows us to use the work of Andersen and Ueno \cite{AU2,AU4,AU5,andersen-ueno2007} as follows: if ${\mathcal V}^\dagger_{N,k}$ is the vacua modular functor constructed in \cite{AU2}, 
then the main result of \cite{AU5} states

\begin{theorem}[Andersen \& Ueno]
For any $\Lambda^{(k)}_N=\Lambda^{(k)}_{\operatorname{SU}(N)}$-labelled marked surface $(\Sib, \overline{\lambda})$, there is a natural isomorphism
$$ I_{N,k} : Z_N^{(k)}(\Sib,\overline{\lambda}) \rightarrow {\mathcal V}^\dagger_{N,k}(\Sib, \overline{\lambda})$$
which is an isomorphism of modular functors.
\end{theorem}

By Definition 11.3 in \cite{AU2}, ${\mathcal V}^\dagger_{N,k}(\Sib, \overline{\lambda})$ is the space of covariantly constant sections of a bundle equipped with a flat connection over the Teichm\"uller space ${\mathcal T}_{\Sib}$ of the marked surface $\Sib$ (see \cite[\S 3]{AU2} for a discussion of ${\mathcal T}_{\Sib}$).   Further, 
by Remark 11.4 of \cite{AU2}, we have for any point $\sigma$ in ${\mathcal T}_{\Sib}$ (giving rise to the Riemann surface $\Sigma_{\sigma}$) that
$$ {\mathcal V}^\dagger_{N,k}(\Sib, \overline{\lambda}) \cong({\mathcal V}^\dagger_{\mathrm{ab}})^{-\frac12 \zeta}(\Sigma_\sigma) \otimes {\mathcal V}^\dagger_{N,k,\overline{\lambda}}(\Sigma_\sigma,\cP) .$$  
Here $\zeta$ is the \emph{central charge} of the Wess--Zumino--Novikov--Witten (WZNW) conformal field theory, i.e. $$\zeta=\frac{k\dim(K)}{k+h^{\vee}},$$  $({\mathcal V}^\dagger_{\mathrm{ab}})^{-\frac12 \zeta}(\Sigma_\sigma)$ is the fibre over $\sigma$ of a certain line bundle over ${\mathcal T}_{\Sib}$ (depending on $W$), defined in Theorem 11.3 of \cite{AU2}, and ${\mathcal V}^\dagger_{N,k,\overline{\lambda}}(\Sigma_\sigma,\cP)$ is the space of vacua  or conformal blocks for the WZNW model for the curve $\Sigma_\sigma$ (see Section \ref{sectionconfblocks}).

From now on we will assume that $f$ is of finite order $m$. 
Then there exists $\sigma \in {\mathcal T}_{\Sib}$ which is a fixed point for $f$. Let us also denote by $f$ the element $(f,0)$ in the extended mapping class group of $\Sib$. Then by Remark 11.4 in \cite{AU2} we have that 

\begin{align}\label{combitrace}
\tr(Z_N^{(k)}(f))  &=  \tr\Big(f^* :  ({\mathcal V}^\dagger_{\mathrm{ab}})^{-\frac12 \zeta}(\Sigma_\sigma)  \rightarrow ({\mathcal V}^\dagger_{\mathrm{ab}})^{-\frac12 \zeta}(\Sigma_\sigma)\Big)\\
&\phantom{={}} \cdot \tr \Big({\mathcal V}^\dagger_{N,k,\overline{\lambda}}(f) : {\mathcal V}^\dagger_{N,k,\overline{\lambda}}(\Sigma_\sigma,\cP) \rightarrow {\mathcal V}^\dagger_{N,k,\overline{\lambda}}(\Sigma_\sigma,\cP)\Big). \nonumber
\end{align}
Let use the notation
\beq\label{abcontrib}\Det(f)^{-\frac12 \zeta} = \tr(f^* : ({\mathcal V}^\dagger_{\mathrm{ab}})^{-\frac12 \zeta}(\Sigma_\sigma) \rightarrow ({\mathcal V}^\dagger_{\mathrm{ab}})^{-\frac12 \zeta}(\Sigma_\sigma)).\eeq
The factor $\Det(f)^{-\frac12 \zeta}$ was computed explicitly in \cite[Theorem 5.3]{andersen95} in terms of the Seifert invariants of $X$.  
We shall denote for short
$$\tr\left({\mathcal V}^\dagger_{N,k,\overline{\lambda}}(f)\right) = \tr \left({\mathcal V}^\dagger_{N,k,\overline{\lambda}}(f) : {\mathcal V}^\dagger_{N,k,\overline{\lambda}}(\Sigma_\sigma,\cP) \rightarrow {\mathcal V}^\dagger_{N,k,\overline{\lambda}}(\Sigma_\sigma,\cP)\right).$$
In order to compute this trace, we will need an alternate description of the vector space ${\mathcal V}^\dagger_{N,k,\overline{\lambda}}(\Sigma_\sigma,\cP)$ and the action ${\mathcal V}^{\dagger}_{N,k,\overline{\lambda}}(f)$ of $f$ on it. Indeed, one can consider the algebraic stack $\mathfrak{M}_{\Sigma_{\sigma},\cP,P_i}$ of (quasi-) parabolic bundles on $\Sigma_{\sigma}$ (i.e. algebraic $K_{\mathbb{C}}=G$-bundles on $\Sigma_{\sigma}$ with a reduction of structure group to the parabolic subgroups $P_i$ at the marked points), and through a presentation of this stack involving the loop group of $G$, one can identify the spaces of conformal blocks with spaces of sections of line bundles $\mathcal{L}_{(k,\overline{\lambda})}$ on this stack, as was shown in \cite{pauly, laszlo-sorger}.  Moreover, each of these line bundles determines a stability condition on the stack, and the substack it selects has a so-called \emph{good moduli space}, which is a variety, a coarse moduli space for the moduli problem, polarised by that line bundle.  In turn these moduli spaces of (semi-)stable parabolic bundles can be identified with gauge-theoretic moduli spaces of flat connections on the punctured surface $\Sigma\setminus \cP$ with prescribed holonomy around the punctures.  

We shall denote these moduli spaces as $\mathcal{M}_{\Sigma_{\sigma},\cP,\overline{\alpha}}$, or $\mathcal{M}_{\overline{\alpha}}$ for short; here $\overline{\alpha}$ is $\frac{\overline{\lambda}}{k_0}$. From the gauge theory side they come equipped with a natural symplectic structure, and one can construct a \emph{pre-quantum line bundle} $\mathcal{L}_{\operatorname{CS}}^k$ for them, from (classical) Chern--Simons theory. In a suitable sense this can be identified with the polarising line bundle from the algebro-geometric perspective, which we shall denote by $\mathcal{L}^k_{\operatorname{pd}}$:
\begin{thm*}
We have
$$\mathcal{L}^k_{\operatorname{CS}}\cong \mathcal{L}^k_{\operatorname{pd}}.$$
\end{thm*}
Finally all of this combines under some minor conditions (see Sections \ref{csbundle} and \ref{sectionconfblocks}) to give
 
\begin{thm*}There is a natural isomorphism, canonical up to scalars,
$$ {\mathcal V}^\dagger_{N,k,\overline{\lambda}}(\Sigma_\sigma) \cong H^0(\cM_{\overline{\alpha}}, \mathcal{L}_{\operatorname{CS}}^k).$$
\end{thm*}

In Sections \ref{csbundle} and \ref{sectionconfblocks} we will construct an explicit action of $f$ on both of the line bundles occurring, covering its action on $\cM_{\overline{\alpha}}$ and establish that  this isomorphism is equivariant. From this we get that
\beq\label{maintrace}\tr\left({\mathcal V}^\dagger_{N,k,\overline{\lambda}}(f)\right) = \tr \Big(f^* : H^0(\cM_{\overline{\alpha}}, \mathcal{L}_{\operatorname{CS}}^k) \rightarrow H^0(\cM_{\overline{\alpha}}, \mathcal{L}_{\operatorname{CS}}^k)\Big).\eeq

Let $\mathcal{M}_{\Sigma_f,L,\overline{\alpha}}$ be the moduli space of flat connections on $\Sigma_f\setminus L$ whose holonomy around $L_i$ lies in the conjugacy class containing $e^{\alpha_i}$. 
We have the following main theorem of this paper.

\begin{theorem}\label{MainTheorem}
For $\overline{\lambda}=(\lambda_1,\ldots,\lambda_n)$ with all $\lambda_i\in\Lambda_{N}^{(k)}$, there exist unique polynomials $P_{\gamma}$ of degree $d_{\gamma}$,  such that
$$Z_N^{(k)}(\Sigma_f,L,\overline{\lambda}) = \Det(f)^{-\frac12 \zeta} \sum_{\gamma} e^{2 \pi i k q_{\gamma}} k^{d_{\gamma}} P_{\gamma}(1/k).$$ 
Here the sum is over all components $\gamma$ of $\mathcal{M}^f_{\overline{\alpha}}$, the $f$-fixed point locus of $\mathcal{M}_{\overline{\alpha}}$.   The number $q_{\gamma}$ is the value the Chern--Simons functional takes on any element of $\mathcal{M}_{\Sigma_f, L, \overline{\alpha}}$ restricting to the corresponding component of $\mathcal{M}_{\overline{\alpha}}$. If the component $\gamma$ is contained in the smooth locus of the moduli space $\mathcal{M}_{\overline{\alpha}}$ then 
$$P_\gamma (k) =  \exp{ (k \Omega\mid_{\mathcal{M}^\gamma_{\overline{\alpha}}})} \cup\operatorname{ch}(\lambda_{-1}^\gamma \mathcal{M}_{\overline{\alpha}})^{-1} \cup \operatorname{Td}(T_{\mathcal{M}^\gamma_{\overline{\alpha}}}) \cap [\mathcal{M}^\gamma_{\overline{\alpha}}],$$  where $\Omega$ is the K\"ahler form on $\mathcal{M}_{\overline{\alpha}}$, and $\lambda_{-1}^\gamma \mathcal{M}_{\overline{\alpha}}$ is a certain element in the $K$-theory of $\mathcal{M}_{\overline{\alpha}}$. 
In particular in the cases where $\mathcal{M}_{\overline{\alpha}}$ 
 is smooth we get a complete formula for the asymptotic expansion of $Z_N^{(k)}(\Sigma_f,L,\overline{\lambda})$, where each coefficient of a power of $k$ is expressed as a cohomology pairing on the moduli space $\mathcal{M}_{\overline{\alpha}}$.
\end{theorem}

We will see in Section \ref{discussion} below that indeed the value the Chern--Simons functional takes on a connection in  $\mathcal{M}_{\Sigma_f,L,\overline{\alpha}}$ depends only on $\gamma$.
The theorem is proved in a way similar to the proof of the main theorem of \cite{andersen95}, namely by applying the
Baum--Fulton--MacPherson--Quart localisation theorem to compute (\ref{maintrace})
as a sum of contributions from each component of the fixed variety. This is then combined with an identification of the traces on the fibres of the line bundle $\mathcal{L}_{\operatorname{CS}}^k$ over each of these components.  It is exactly here that the use of the bundle $\mathcal{L}_{\operatorname{CS}}^k$ coming from Chern--Simons theory comes in, as this part of the contribution can be expressed as the integral of the classical Chern--Simons form over the mapping torus.
See Section \ref{endgame} for the details.
Our main theorem has the immediate following corollary:

\begin{cor}
\label{aecthm}
The Asymptotic Expansion Conjecture holds for the pairs $(X=\Sigma_f,L)$ obtained from finite order mapping tori and for any $k=sk_0$-dependent labelling ($s\in\mathbb{N}$) $$\overline{\lambda}^{(k)}=k\overline{\lambda},$$ where $\overline{\lambda}=(\lambda_1,\ldots,\lambda_n)$ and all $\lambda_i \in \Lambda_{N}^{(k_0)}$ for some fixed $k_0\in\mathbb{N}$. \end{cor}

By further analysing dimensions of parabolic twisted cohomology groups (see Section \ref{fromsoeren}), we get the following theorem. Let 
$ \mathcal{M}^{\gamma}_{\Sigma_f,L,\overline{\alpha}}$ be the union of components of $ \mathcal{M}_{\Sigma_f,L,\overline{\alpha}}$ whose connections restrict to lie in the $\gamma$-component of $\mathcal{M}^f_{\overline{\alpha}}$.

\begin{theorem}\label{GRT}
If a given connected component $\gamma$ contains smooth points from $\mathcal{M}_{\bar \alpha}$, then
$$
d_{\gamma} = \frac{1}{2} \max_{\nabla \in  \mathcal{M}^{\gamma}_{\Sigma_f,L,\overline{\alpha}} }
            \left( h_{\nabla,\operatorname{par}}^{1} - h_{\nabla}^{0} \right),
$$
where $\max$ here means the maximum value
$h_{\nabla.\operatorname{par}}^{1} - h_{\nabla}^{0}$ attains on a non-empty Zariski open subset
of $\mathcal{M}_{\Sigma_f,L,\overline{c}}^{\gamma}$ on which $h_{\nabla,\operatorname{par}}^{1} - h_{\nabla}^{0}$ is constant.
In particular, the Growth Rate Conjecture holds for pairs of manifolds and links $(X=\Sigma_f,L)$ which are obtained as finite order mapping tori, when all connected components $\gamma$ contain smooth points.
\end{theorem}

\subsection*{Outline and further comments}
The rest of this paper is organised as follows: in Section \ref{two} we give a quick introduction to the moduli spaces $\mathcal{M}_{\overline{\alpha}}$ of parabolic bundles or flat connections on punctured surfaces, to set up notations.  We also give a proof of the simply-connectedness of $\mathcal{M}_{\overline{\alpha}}$.  Section \ref{csbundle} is entirely devoted to the \emph{Chern--Simons line bundle} on $\mathcal{M}_{\overline{\alpha}}$ -- this line bundle arises from classical Chern--Simons theory and is gauge-theoretic~/ symplectic in origin, giving a pre-quantum line bundle $\mathcal{L}_{\operatorname{CS}}^k$ for the canonical symplectic form on $\mathcal{M}_{\overline{\alpha}}$.  We study how diffeomorphisms act on this line bundle, and in fact we do a little more: we construct a lift of the action of the relevant mapping class group to the total space of this bundle. In Section \ref{sectionconfblocks} we switch to an algebro-geometric or representation-theoretic picture, and discuss the \emph{parabolic determinant bundle} on $\mathcal{M}_{\overline{\alpha}}$, whose sections give rise to the spaces of conformal blocks.  Again we exhibit the lift for the action of a (complex) automorphism of the Riemann surface to the total space of this bundle, and using known results on Quillen metrics we show that the Chern--Simons and parabolic determinant line bundles are indeed equivariantly isomorphic.  In the final Section \ref{endgame} we put the strategy of \cite{andersen95} to work to establish the Asymptotic Expansion Conjecture for links in finite-order mapping tori by using the Lefschetz--Riemann--Roch theorem of Baum, Fulton and Quart on the moduli spaces $\mathcal{M}_{\overline{\alpha}}$.  We conclude with a discussion of parabolic group cohomology that establishes the Growth Rate Conjecture in our situation.

We should mention at this point a certain restriction we have to impose: as we are relying heavily on the paper \cite{daska-went1} for the construction of $\mathcal{L}_{\operatorname{CS}}^k$, we are forced to restrict ourselves to the situation where the $\alpha_i$, or equivalently the $\lambda_i$, are \emph{regular}, that is to say, they are contained in the interior of the Weyl chamber (this is equivalent to working with \emph{full flags} in the picture of parabolic bundles).  Though the statements of their results are no doubt true in greater generality, Daskalopoulos and Wentworth need to impose this restriction for technical reasons on a number of occasions.  We will therefore also impose this restriction from Section \ref{constr} through Section \ref{isotopysection}, as well as anywhere later on where $\mathcal{L}_{\operatorname{CS}}^k$ occurs.

We would like to remark that we completely link the gauge-theoretic definition of the quantum invariants with the construction of the modular functor using conformal blocks.  Through the work of Andersen and Ueno \cite{AU4, AU2, AU5} the latter is known to be equivalent to the original constructions of Reshetikhin--Turaev.  

As a quick look at the bibliography will betray we are drawing upon a rather large body of literature to establish our results, and necessarily this paper involves various perspectives and different technical tools (from Sobolev spaces to algebraic stacks and Kac--Moody Lie algebras).  In particular we are crucially using the papers \cite{pauly}, \cite{daska-went1}, \cite{laszlo-sorger}, \cite{biswas} and \cite{biswas-ragha} together with \cite{AU2,AU4,AU5,andersen-ueno2007}.  To the extent possible we have tried to use notations in line with those authors, and in general we have tried to strike a balance between giving complete references to the literature, and avoiding an overload of  translations between notational conventions in our exposition.

An obvious further question following on our results is to determine the coefficients in the Asymptotic Expansion Conjecture , and to give a topological interpretation of them, as was done in the case of mapping tori without links in \cite{andersen-himpel2011}. We intend to take this up in future work.

\subsection*{Acknowledgements} The authors wish to thank Indranil Biswas, Christoph Sorger, Michael Thaddeus, and in particular Richard Wentworth for helpful conversations.

\section{Moduli spaces of parabolic bundles and flat connections}\label{two}
Let $\Sigma_{\sigma}$ be a compact Riemann surface of genus $g\geq 2$, and $\cP=\{p_1, \ldots, p_n\}$ a collection of distinct marked points on $\Sigma_{\sigma}$.  Below we shall denote with $\Sigma$ the smooth surface underlying $\Sigma_{\sigma}$, and with $\Sigma^o$ the punctured surface $\Sigma\setminus \cP$.  A \emph{quasi-parabolic} structure on a holomorphic vector bundle $E\rightarrow \Sigma$ of rank $N$ is a choice of a filtration of its fibres over each of the points in $\cP$:
$$E|_{p_i}=E_{i,1}\varsupsetneq E_{i,2}\varsupsetneq \ldots \varsupsetneq E_{i,r_i}\varsupsetneq E_{i,r_i+1}= \{0\}.$$  Its multiplicities are $m_{i,j}=\dim (E_{i,j}/E_{i,j+1})$.  If all multiplicities are $1$ for a given $i$, or equivalently $r_i=N$, then the flag at $p_i$ is said to be \emph{full}; in general the tuple $(m_{i,1},\ldots,m_{i,r_i})$ is said to be the \emph{flag type} at $p_i$.  Alternatively, this data determines a reduction of structure group for the frame bundle of $E$ to the corresponding parabolic subgroups of $\operatorname{GL}(N,\mathbb{C})$ at the marked points (below we shall freely switch between the equivalent vector bundle and principal bundle pictures).

A parabolic bundle is further equipped with \emph{parabolic weights} $\overline{\alpha}=(\alpha_1,\ldots, \alpha_n)$ for all flags, i.e. a choice of real numbers, \begin{equation}\label{parweights}\alpha_i=(\alpha_{i,1},\ldots, \alpha_{i,r_i})\hspace{.5cm}\text{with}\hspace{.5cm}0\leq \alpha_{i,1}< \alpha_{i,2}<\ldots< \alpha_{i,r_i} <1.\end{equation}  Note that often in the literature the inequalities between the $\alpha_{i,j}$ for various $j$ are asked to be strict (as above), but later on it will be convenient for us to relax this condition.  One can think of each of the parabolic weights $\alpha_i$ (with each of the $\alpha_{i,j}$ occurring with multiplicity $m_{i,j}$)  as living in the Weyl alcove in the Cartan algebra of $\operatorname{SU}(N)$.

The \emph{parabolic degree} of $E$ is $\text{pdeg}(E)=\deg(E)+\sum_{i,j}\alpha_{i,j}m_{i,j}$, and its \emph{slope} is $$\mu(E)=\frac{\text{pdeg}(E)}{\rk(E)}.$$  Any sub-bundle $F$ of $E$ inherits a canonical structure of parabolic bundle itself (the same is true for quotient bundles).  We can therefore define $E$ to be (semi-)stable if, for every sub-bundle $F$, we have that $$\mu(F)\underset{(=)}{<}\mu(E).$$  

For purely numeric reasons, the set of parabolic weights for a given flag-type for which strictly semi-stable bundles can exist consists of a union of hyperplanes in the space of all weights (and indeed, for those weights strictly semi-stable bundles do exist).  We shall refer to weights in the complement of these hyperplanes as \emph{generic}. 
Given any rank, degree, flag type and choice of weights, there exists a coarse moduli space $\mathcal{N}_{\Sigma_{\sigma},\cP,\overline{\alpha}}$, or $\mathcal{N}_{\overline{\alpha}}$ for short, of ($\mathcal{S}$-equivalence classes of) semi-stable parabolic bundles \cite{mehtaseshadri}.  

By taking the determinant, one obtains a morphism from $\mathcal{N}_{\overline{\alpha}}$ to the Picard group of $\Sigma$.  The moduli space of semi-stable parabolic bundles with trivial determinant $\mathcal{M}_{\Sigma_{\sigma},\cP,\overline{\alpha}}$, or $\mathcal{M}_{\overline{\alpha}}$ for short, is the fibre of the trivial line bundle under this morphism (we shall focus on $\mathcal{M}_{\overline{\alpha}}$ in the sequel).  Both $\mathcal{N}_{\overline{\alpha}}$ and $\mathcal{M}_{\overline{\alpha}}$ are normal projective varieties.  Their singular locus consists exactly of the semi-stable bundles (with the exception of the fixed-determinant non-parabolic rank $2$ case in genus $2$).  

We note that when considering bundles with trivial determinant (or, equivalently, $\operatorname{SL}(N,\mathbb{C})$ principal bundles), conventions in the literature vary about the weights.  One can ask for the $\alpha_{i,j}$ to  satisfy the same inequality above plus the condition that $\sum_{i}\alpha_{i,j}m_{i,j}\in \mathbb{Z}$ or, corresponding to a standard representation for the Weyl alcove of $\operatorname{SL}(N,\mathbb{C})$, for $\alpha_{i,1}< \alpha_{i,2}<\ldots< \alpha_{i,r_i}$ with $\alpha_{i,r_i}-\alpha_{i,1}< 1$ and $\sum_j \alpha_{i,j}m_{i,j}=0$.  We shall use the latter convention, in particular this implies that the parabolic degree of our bundles vanishes.

There exists a homeomorphism (which is a diffeomorphism on the smooth locus) between $\mathcal{M}_{\overline{\alpha}}$ and the moduli space of those representations of $\pi_1(\Sigma^o)$ into $\operatorname{SU}(N)$ where the loop around each of the $p_i$ gets mapped to the conjugacy class of the exponential of the parabolic weights \cite{mehtaseshadri}.   Indeed, $\pi_1(\Sigma^o)$ admits a presentation
\beq\label{fundpunctures}
  \pi_1(\Sigma^o) = \left\langle A_1,\dots,A_g,B_1,\dots,B_g,a_1,\dots,a_n \relmiddle{|} \prod_{i=1}^g [A_i,B_i]\prod_{j=1}^n a_j =1 \right\rangle.
\eeq
If we fix the conjugacy classes $c_1, \dots, c_n \in K = \SU(N)$, each containing $e^{\alpha_i}$ respectively (where we abuse notation and let $\alpha_i$ be the diagonal matrix with entries $i\alpha_{i,j}$, each occurring with multiplicity $m_{i,j}$), 
then we have topologically
\beq\label{charvar}
  \calM_{\overline{\alpha}} \cong \left\{ \rho \in \Hom(\pi_1(\Sigma^o),K)\relmiddle| \rho(a_i) \in c_i, \, i =1, \dots, n \right\} \ \Big/ \ K,
\eeq
where $K$ acts by simultaneous conjugation, using the presentation of $\pi_1(\Sigma^o)$ given above.
 For our purposes it is most useful to consider the differential geometric version of this, due to Biquard \cite{biquard}, Poritz \cite{poritz}  and Daskalopoulos--Wentworth \cite{daska-went1}, 
 generalising the work of Donaldson \cite{donaldson} in the non-parabolic case.  From this point of view there is a diffeomorphism between $\mathcal{M}_{\overline{\alpha}}$ and the moduli space of flat $\operatorname{SU}(N)$-connections whose holonomy around the marked points lies in the conjugacy class of the exponential of the relevant weights (remark that the flatness is a consequence of the vanishing of the parabolic degree; more generally one would have central curvature determined by the parabolic degree).  As we shall need the construction in our discussion of the Chern--Simons bundle, we shall review it in Section \ref{csbundle}, following \cite{daska-went1}.  To minimise notation we shall denote both the moduli space of parabolic bundles and the moduli space of flat connections by $\mathcal{M}_{\overline{\alpha}}$, as it will always be clear from the context which perspective we take.

The moduli space $\mathcal{N}_{\overline{\alpha}}$ (and hence by restriction also $\mathcal{M}_{\overline{\alpha}}$) admits a natural symplectic form on its smooth locus (independent of the complex structure of $\Sigma$), that combines with the complex structure to give a K\"ahler structure.  In the closed case this was first described by Atiyah--Bott \cite{ab} and Goldman \cite{gold}.  In the non-closed case we  are considering here, it was discussed by Biswas and Guruprasad in \cite{biswas}.  It is perhaps easiest described in terms of moduli of connections, from the principal point of view.  Let $K$ be a compact Lie group (this shall be $\operatorname{SU}(N)$ for us) with $\mathfrak{k}$ its Lie algebra.  Then the (real) tangent space to a smooth point $[\nabla]$ of $\mathcal{M}_{\overline{\alpha}}$ can be described as the image of $$H^1_{\text{c}}(\Sigma^o, \mathfrak{k}_{\text{ad}})\rightarrow H^1(\Sigma^o, \mathfrak{k}_{\text{ad}}),$$ where $H^1_{\text{c}}$ stands for first cohomology with compact support, and we consider the adjoint bundle $\mathfrak{k}_{\text{ad}}$ with the induced flat connection given by $\nabla$.  Using the Killing form  on $\mathfrak{k}$, we put 
\beq\label{formcompsup}
\Omega(A,B)=\int_{\Sigma^o} \operatorname{tr}(A\wedge B).
\eeq  
In Section \ref{fromsoeren} below we shall also need another incarnation of the tangent space, in line with the view on $\mathcal{M}_{\overline{\alpha}}$ as a character variety (\ref{charvar}), given by \cite{biswas}.  Indeed, if $[\rho]$ is an equivalence class of (irreducible) representations, we have $$T_{[\rho]}\mathcal{M}_{\overline{\alpha}}\cong H^1_{\text{par}}(\pi_1(\Sigma^o),\mathfrak{k}_{\Ad_{\rho}}),$$ where the right hand side is the first parabolic group cohomology  (see Section \ref{fromsoeren} below for further details and references). 
  
Finally, we shall need to know the fundamental group of $\mathcal{M}_{\overline{\alpha}}$.

\begin{thm}\label{simply}
For $\Sigma$ and $\cP$ as above, both $\mathcal{M}_{\overline{\alpha}}$ and its smooth locus are simply-connected for any choice of weights $\overline{\alpha}$.
\end{thm}

This line of proof was essentially already suggested in \cite[page 173]{scheinost}, see also \cite[\S 4]{daska-went1}.  Note that this property also follows from the rationality of these moduli spaces \cite{bodyok}, since smooth projective rationally connected varieties are simply-connected \cite[Cor. 4.18]{debar}, and rational varieties are rationally connected.  In order not to impose the (minor) conditions of \cite{bodyok}, or genericity of weights, we have provided the direct proof below.

\begin{proof}  We will find it useful here to allow the inequalities in (\ref{parweights}) to be weak.  If any of the $\alpha_{i,j}$ coincide for consecutive $j$, it is clear that the stability of any parabolic bundle is equal to that of the underlying coarser parabolic bundle, where the relevant part of the flag is forgotten.  In particular, for such $\overline{\alpha}$, the moduli space is a bundle of flag varieties over the corresponding moduli space of coarser flag type.  In particular, when one sets all weights equal to one, one just obtains a flag variety bundle over the moduli space $\mathcal{M}$ of non-parabolic $\operatorname{SL}(N, \mathbb{C})$-bundles.  It was shown by Daskalopoulos and Uhlenbeck \cite[Theorem 3.2]{daska-uhl} by analytic methods that the smooth part of the non-parabolic moduli space is simply-connected.  Since flag manifolds are also simply-connected, we therefore have from the homotopy long exact sequence for fibre bundles that this smooth part of the moduli space of weight zero is simply-connected.  

It is a well-known fact, essentially a consequence of Zariski's main theorem (see e.g. \cite[Thm. 12.1.5]{toric} or \cite[page 33]{fult-laz}), that the fundamental group of an open subvariety of a normal variety surjects onto the fundamental group of the whole variety.  We therefore have that both the moduli space for weights zero, as well as all moduli spaces for weights in the interior of a nearby chamber, are all simply-connected (since the bundles that are stable for zero weights remain so for weights in an adjacent chamber).  Moving the weights around in general leads to the well-known variation of GIT pictures (see e.g. \cite{dolg-hu}, \cite[\S 7]{thaddeus} or \cite{boden-hu}).  In particular, all moduli spaces are birational, and since the fundamental group is a birational invariant for smooth projective varieties, we obtain the result for all generic weights.  
By using the fact that the singular locus for weights on a wall is of high enough codimension, and by applying the above fact of fundamental groups for normal varieties when hitting a wall, we finally obtain it for both the stable (i.e. non-singular) locus as well as the whole moduli space, for any choice of weights.
\end{proof}
Remark that this proof does not hold for genus $1$, since there in the moduli space of non-parabolic bundles all points are strictly semi-stable (see \cite{tu} for a rendition of Atiyah's classical results \cite{atiyah} in this language), and as a consequence one cannot use the same start of the birational argument.

\section{The Chern--Simons line bundle}\label{csbundle}

As mentioned in the introduction we need to consider two line bundles on $\mathcal{M}_{\overline{\alpha}}$, one of a symplectic~/ gauge-theoretic nature (the Chern--Simons line bundle), one of an algebro-geometric nature (the parabolic determinant bundle).  In this section we discuss the former.  Our prime focus is on describing the lift of the action of the mapping class group to this bundle.

\subsection{Review for closed surfaces}\label{no punctures}
To set the tone we begin by reviewing the construction of the lift of the mapping class group action on the Chern--Simons line bundle in the case of a closed surface $\Sigma$.  In Section \ref{with punctures} we will then construct an analogous lift for the case of a punctured Riemann surface $\Sigma^o=\Sigma\setminus \mathcal{P}$.  Let $\Sigma$ denote a closed Riemann surface of genus $g>0$.  Let $P$ be a smooth principal $K$-bundle over $\Sigma$, for $K$ a compact, semi-simple and simply-connected Lie group -- without a loss of generality we can and will assume this to be trivial $P\cong \Sigma \times K$.  Let $f \co \Sigma \to\Sigma$  be an orientation-preserving diffeomorphism of $\Sigma$.  

We are now interested in $\mathcal{L}_{\operatorname{CS}}^k$, i.e. the Chern--Simons line bundle at level $k\in\mathbb{Z}$ over the moduli space of flat connections $\mathcal{M}$ described in \cite{rsw} for $K=\operatorname{SU}(2)$ or \cite{freed95} in the general case, and used in \cite{andersen95}, for example.  It can be constructed as follows: let $\cA_P$ denote the space of connections on $P$ -- using our trivialisation of $P$ we can identify $\cA_P$ with the space of sections of the adjoint bundle of $P$, by expressing any $\nabla=\nabla_A$ as $d+A$ (in Section \ref{with punctures} we will use Sobolev completions, which strictly speaking ought to be done here as well).  The moduli space $\mathcal{M}$ can be constructed as an (infinite dimensional) symplectic reduction of $\cA_P$ by the gauge group $\cG\simeq C^{\infty}(\Sigma,K)$; here the moment map is given by the curvature of a connection, hence the level set that one takes the quotient of consists of the flat connections.  One can now lift the action of $\cG$ to the trivial bundle $\cA_P\times \mathbb{C}$ as follows: define the cocycle $\Theta^{k}:\cA\times\cG\rightarrow \mathbb{C}$ by 
\beq\label{firstco}
\Theta^{k}(\nabla_A,g):=\operatorname{exp}\left(2 \pi ik(\CS(\widetilde{A}^{\tilde{g}})-\CS(\widetilde{A}))\right),
\eeq 
where $\widetilde{A}$ and $\tilde{g}$ are any extensions of $A$ and $g$ (which always exist in our setup) to an arbitrary compact 3-manifold $Y$ with boundary $\Sigma$ and $\CS$ is the Chern--Simons action as usual, i.e. $$\CS (A)=\frac{1}{8\pi^2}\int_Y \operatorname{tr}\left( A\wedge dA +\frac{2}{3} A\wedge A \wedge A\right) .$$  The action of $\cG$ on $\cA\times\mathbb{C}$ is given by
$$(\nabla_A,z)\cdot g:=(\nabla_A^g,\Theta^{k}(\nabla_A,g)\cdot z),
$$
where $\nabla_A^g:=d+\operatorname{Ad}_{g^{-1}}A+g^{*}\omega$ denotes the usual gauge group action with $\omega\in\Omega^{1}(G,\fg)$ the Maurer--Cartan form.  Since $\Theta^{k}$ satisfies the cocycle condition
$$
\Theta^{k}(\nabla_A,g)\Theta^{k}(\nabla_A^g,h)=\Theta^{k}(\nabla_A,gh),
$$
and $\cG$ preserves flat connections, we obtain the induced Chern--Simons line bundle $\mathcal{L}_{\CS}^k$ over $\cM$. 

For future purposes it is useful to observe that $\Theta^{k}$ can equivalently be constructed without requiring the existence of a bounding 3-manifold for $\Sigma$.  Since every gauge transformation is homotopic to the identity, we may extend $g$ on $\Sigma$ to $\tilde{g}$ on the cylinder $[0,1]\times\Sigma$ using such a homotopy, so that $\tilde{g}_{0}=g$ and $\tilde{g}_{1}=e$.  For $\pi:[0,1]\times\Sigma\rightarrow \Sigma$ the natural projection map, extend $\nabla_A$ on $\Sigma$ to $\widetilde{\nabla_A}=\pi^{*}\nabla_A=d+\pi^*A$ on $[0,1]\times\Sigma$.  Then $\widetilde{\nabla_A}^{\tilde{g}}$ is an extension of $\nabla_A^g$ to $[0,1]\times \Sigma$.  Choosing the standard orientation on $[0,1]$ we define 
\beq\label{secondco}
\Theta^{k}(\nabla_A,g)=\operatorname{exp}\left(-2 \pi ik\CS_{[0,1]\times\Sigma}(\widetilde{A}^{\tilde{g}})\right),
\eeq 
and one can easily show that equations \eqref{firstco} and \eqref{secondco} agree.  The expression for $\Theta$ given in \eqref{secondco} generalises more readily to the case of a surface with punctures since it does not require the existence of a bounding 3-manifold to be well defined.

A crucial aspect of the line bundle $\mathcal{L}^k_{\operatorname{CS}}$ is that it is a pre-quantum line bundle on $\mathcal{M}$, i.e. it naturally comes with a connection whose curvature is the Atiyah--Bott--Goldman symplectic form.

Let $\operatorname{Diff}_{+}(\Sigma)$ denote the group of orientation preserving diffeomorphisms of $\Sigma$. Then $\operatorname{Diff}_{+}(\Sigma)$ naturally acts on $\cA_P$ and $\cG$ by pullback and we would like to show this action may be lifted to an action on $\cL^{k}_{\CS}$.  First, define an action of $\operatorname{Diff}_{+}(\Sigma)$ on $\cA_P\times\mathbb{C}$ by
\beq\label{liftaction}
f^{*}(\nabla_A,z):=(f^{*}\nabla_A,z).
\eeq
This trivially defines a lifted action on $\cA_P\times\mathbb{C}$.  Furthermore, one can show that this lift is compatible with the gauge group action.  Indeed, we have

\begin{lem}\label{combination}The two lifts described above combine to a lift of the action of $\operatorname{Aut}(P)= \mathcal{G}\rtimes_{\Psi} \operatorname{Diff}_{+}(\Sigma)$ on $\mathcal{A}_P$ to $\mathcal{A}_P\times \mathbb{C}$.  
\end{lem}  Here the semi-direct product is made with respect to the morphism $\Psi:\operatorname{Diff}_+(\Sigma) \rightarrow \operatorname{Aut}(\mathcal{G}): \Psi(f)(g)=g\circ f$.  Strictly speaking we have defined a left action of $\operatorname{Diff}_+(\Sigma)$, and we use a right action of $\mathcal{G}$, so we switch to a right action of $\operatorname{Diff}_+(\Sigma)$ to obtain a right action of the semi-direct product.
The proof of this lemma is identical to the proof given in Lemma \ref{liftlem} for the case with punctures.
Since this action preserves flatness, we get an induced action of $\operatorname{Diff}_{+}(\Sigma) = \operatorname{Aut}(P)/\mathcal{G}$ on $\mathcal{L}^k_{\CS}$ over the moduli space $\mathcal{M}$.

Let $\operatorname{Diff}_{0}(\Sigma)<\operatorname{Diff}_{+}(\Sigma)$ denote the subgroup of diffeomorphisms isotopic to the identity.  It is straightforward to see that $\operatorname{Diff}_{0}(\Sigma)$ acts trivially on $\cM$ -- perhaps the easiest way to see this is through the identification of the moduli space as the representation space $\Hom(\pi_1(\Sigma), K)/K$, and to then observe that since maps homotopic to the identity induce the identity map on the fundamental group, the $\operatorname{Diff}_0(\Sigma)$-action is trivial. 

To see that the action of the mapping class group lifts to an action on
the Chern--Simons line bundle, suppose that $f \in \Diff_0(\Sigma)$, let
$[A] \in \mathcal{M}$ be given, and let $g$ be a gauge transformation
with $f^*A = A^g$. We then claim that $\Theta^k(A,g) = 1$.

Any isotopy from $f$ to the identity diffeomorphism defines a suitable
extension of $\nabla_A$ and $\nabla_A^g$ to $[0,1] \times \Sigma$. Moreover, using
the isotopy, now considered as a diffeomorphism of $[0,1] \times \Sigma
$, this extension is a pullback of a product connection on $[0,1] \times
\Sigma$, and since the Chern--Simons functional of a product connection
vanishes, the claim follows by diffeomorphism invariance of the
Chern--Simons functional.

As in \cite[\S 7]{andersen95}, given a fixed point $[\nabla_A]\in\cM$, so that $[f^{*}\nabla_A]=[\nabla_A]$, then since $P$ is isomorphic to $f^{*}P$, there exists an isomorphism $\psi$ from $P$ to $f^{*}P$ such that $\psi f^{*}\nabla_A=\nabla_A$.  Composing $\psi$ with the natural bundle map from $f^{*}P$ to $P$ covering $f$, we get a lift $\phi:P\rightarrow P$ covering $f$.  Let $\nabla_{A_{\phi}}$ denote the connection induced by $\nabla_A$ on the mapping torus $P_{\phi}=[0,1]\times_{\phi} P$.
\begin{lem}{\cite[Lemma 7.2]{andersen95}, \cite[Thm. 2.19]{freed95}}\label{analem}
We have
\begin{equation*}
\operatorname{tr}\left(f^{*}:\cL^{k}_{\CS}\Big|_{[\nabla_A]}\rightarrow \cL^{k}_{\CS}\Big|_{[\nabla_A]}\right)=\operatorname{exp}\left(2\pi ik\CS(P_\phi,A_{\phi})\right).
\end{equation*}
\end{lem}
The proof of this is identical to the proof of Lemma \ref{tracelem} given in the sequel.

\subsection{Punctured surfaces}\label{with punctures}
\subsubsection{Introduction}
In this section we give a construction of the \emph{Chern--Simons line bundle} over the moduli space of flat connections with prescribed holonomy around the punctures.  This line bundle has been discussed in many places in the literature (e.g. \cite{rsw, konno1, konno2, daska-went1, fujita, meinwood, charles}), but apparently never quite in the generality or the setting we need. 
All the basic ideas are well-known however.   Remark that the line bundle we construct differs somewhat from the one considered by Freed in \cite{freed95}.  In general Freed lays out the classical Chern--Simons field theory in great detail, hence motivating the appearance of the Chern--Simons line bundle.  However, for surfaces with boundary he chooses to use a slight modification of the Chern--Simons action, so as to obtain a line bundle that descends to all moduli spaces of flat connections with prescribed holonomy, without an integrality condition on the latter.   As mentioned by Freed, by just using the Chern--Simons action instead one would obtain the line bundle as in \cite{daska-went1}; this is the approach we shall follow.  Unlike the description in \cite{meinwood, charles}, we work with punctured surfaces, rather than surfaces with boundary, though the end results should be equivalent.

The analytic construction of the moduli spaces was done by Biquard \cite{biquard}, Poritz \cite{poritz} and Das\-ka\-lo\-poul\-os--Went\-worth \cite{daska-went1, daska-went2}.  
All of them use \emph{weighted} Sobolev spaces, the use of which in a gauge-theoretic context was pioneered by Taubes \cite{taubes}.  As only Daskalopoulos--Wentworth discuss the line bundle we are interested in, we shall follow their approach.  Strictly speaking, the exposition in \cite{daska-went1} was only for the case of once-punctured Riemann surfaces (but arbitrary rank), and in \cite{daska-went2} these authors discuss the case for arbitrary finite punctures, but only for the case of $\operatorname{SU}(2)$.  The general case is a superposition of these two though, and we summarise it here just for the sake of completeness (see also \cite[\S 3]{matic} for some of the analytic background).  Remark that unlike in Section~\ref{no punctures} we now specialise to the case $K=\operatorname{SU}(N)$.

The construction of the \emph{Chern--Simons line bundle} $\cL^k_{\CS}$ can be found in \cite{daska-went1} in the case where $\cP$ consists of a single point.  Note that the construction of $\mathcal{M}_{\overline{\alpha}}$ and $\cL_{\CS}^k$ by \cite{daska-went1} happens in two steps: in the first step one takes the quotient of the space of all connections by the group of gauge transformations that vanish at the marked points, and in the second step one further quotients out by a finite-dimensional compact Lie group.  The construction of the line bundle also follows this two step procedure.  The first step follows closely the case for a surface without punctures, as outlined above, using a certain cocycle to define the line bundle.  This part of the construction always goes through.  In the second step one needs however a certain integrality condition on the parabolic weights times the level $k$ to hold for the line bundle to fully descend to $\mathcal{M}_{\overline{\alpha}}$.

\subsubsection{Construction of the moduli space}\label{constr}

Let $(D_{i},z_i)$ be (disjoint) local coordinates around each $p_i\in\cP$, so that $z_i$ are local isomorphisms onto the open unit disk in $\mathbb{C}$ with $z_i(p_i)=0$.  Setting $w_i=-\log z_i$, then $w_i$ maps $D_i\setminus\{p_i\}$ to the semi-infinite cylinder
$$\cC=\left\{(\tau,\theta)\,\,|\,\,\tau\geq 0, \theta\in[0,2\pi]\right\}/\left[(\tau,\theta)\sim (\tau,\theta+2\pi)\right].$$
Let $(\tau_{i},\theta_i)$ denote the corresponding coordinates on $D_i\setminus\{p_i\}$.  Also fix a metric $h$ on $\Sigma^o$ compatible with the complex structure on $\Sigma$ such that it restricts to the standard flat metric on the semi-infinite cylinder, $h\big|_{D_i\setminus\{p_i\}}=d\tau_{i}^{2}+d\theta_{i}^{2}$.  Note that a priori the $z_i$ are just smooth functions, but if a complex structure is chosen to obtain a Riemann surface $\Sigma_{\sigma}$, we will assume the $z_i$ to be holomorphic.

We suppose we have chosen weights for each of the $p_i$ as in (\ref{parweights}) -- recall from the introduction that we will now assume these to be regular.  Recall that we think of these as living in the Weyl alcove of $\operatorname{SU}(N)$, and we denote the corresponding diagonal matrix in $\mathfrak{su}(N)$ also as $\alpha_i$. 
 
 We shall need the centraliser of $e^{\alpha_i}$ in $\operatorname{SU}(N)$, denoted by $L_i^c$, as well as the Lie algebra of its normaliser in $\operatorname{GL}(N,\mathbb{C})$, which  we shall denote as $\mathfrak{l}^{\mathbb{C}}_i$ (note that in \cite{daska-went1} the former are denoted as ${\operatorname{P}_a}$).

 Let $P$ be the trivial principal $K=\operatorname{SU}(N)$-bundle over $\Sigma^o$; further let $E$ be the vector bundle associated to $P$ and the defining representation of $\operatorname{SU}(N)$, $\mathfrak{g}_P$ the adjoint bundle of $P$, and $\mathfrak{gl}_E=E\otimes E^*$ -- all of these are smooth complex hermitian bundles, and we can think of $\mathfrak{g}_P$ as a sub-bundle of $\mathfrak{gl}_E$.
We fix a base connection $\nabla_0$ on $P$ that takes the form $d+\alpha_id\theta_i$ on $D_i\setminus \{p_i\}$ (such a $\nabla_0$ always exists), and by abuse of notation denote the induced linear connections on $E$, $\mathfrak{g}_P$ and $\mathfrak{gl}_E$ by $\nabla_0$ as well.

For a given Hermitian vector bundle $F$ over $\Sigma^o$ equipped with a hermitian connection $\nabla_0$ (e.g. $E\otimes T^*\Sigma^o$, where we combine the base connection on $E$ with the Chern connection on $T^*\Sigma^o$), we will need to consider the \emph{weighted} Sobolev spaces of sections of $F$.  For any $\delta \in \mathbb{R}$, recall that this is the completion of the space of (compactly supported) smooth sections of $T^{*}\Sigma^o\otimes E$ in the norm 
$$
|| \sigma ||_{l,\delta}^p = \left( \int_{\Sigma} e^{\tau \delta} \left( |\nabla_0^{(l)} \sigma|^p +\dots + |\nabla_0 \sigma |^p+ |\sigma |^p\right) \right)^{\frac{1}{p}}.
$$Here $\tau:\Sigma^o \rightarrow \mathbb{R}$ is the smooth function coinciding with $\tau_i$ on each $D_i\setminus\{p_i\}$, and zero outside of the $D_i$.  
We remark that the weights $\delta$ used for the Sobolev spaces and the parabolic weights $\alpha_i$ are different notions.

We define a space of connections modelled on these Banach spaces by
\begin{equation}\label{Adelta}\mathcal{A}_{\delta} = \nabla_0 + L^2_{1,\delta}(T^*\Sigma^o \otimes \mathfrak{g}_{P}),\end{equation}  and we will denote by $\mathcal{A}_{\delta,F}$ the subspace of flat irreducible connections in $\mathcal{A}_{\delta}$.  Note that these spaces do not change if we replace $\nabla_0$ by another connection which coincides with $\nabla_0$ on the $D_i$.  From the complex structure of $\Sigma$ the vector space $ L^2_{1,\delta}(T^*\Sigma^o \otimes \mathfrak{g}_{P})$ moreover inherits a canonical complex structure (cfr. \cite[\S 5]{ab} and for further discussion see also \cite{andersen1998_NewPolarizations}). 

Next we turn our attention to the group of gauge transformations.  Define $$\mathcal{D}= \left\{ \phi\in L^2_{2,\text{loc.}}(\mathfrak{gl}(\mathfrak{g}_{P}))\ \big| \ ||\nabla_0 \phi || _{1,\delta}^{2} < \infty \right\}.$$ Here, as usual, $L^2_{2,\text{loc.}}$ refers to those sections whose product with the characteristic function $1_K$, for any compact $K\subset \Sigma^o$, is in $L^2_2$.

We furthermore have a natural map, $$\sigma: \mathcal{D}\rightarrow \prod_i \mathfrak{l}_i^{\mathbb{C}}: \phi\mapsto (\sigma_1(\phi),\ldots, \sigma_n(\phi)).$$  Here we identify $\mathfrak{l}_i^{\mathbb{C}}$ with the space of parallel sections (with respect to $\nabla_0$) of $\mathfrak{gl}(\mathfrak{g}_P)$ restricted to a circle around $p_i$ in $D_i$, and we put $\sigma_i(\phi)(\theta)=\lim_{\tau\rightarrow \infty} \phi(w_i^{-1}(\tau, \theta))$ (see \cite[\S 3.1]{daska-went1} or \cite[\S 2]{matic} for more details).

We can now define the Banach Lie groups 
\begin{align*}
\mathcal{G}_{\delta} &=\left\{ \phi \in \mathcal{D}\ \big| \ \phi\phi^*=\phi^*\phi=I, \det \phi = 1\right\},\\
\mathcal{G}_{0,\delta}&=\left\{ \phi \in \mathcal{G}_{\delta}\ \big|\ \sigma (\phi)=I\right\},
\end{align*}
and we have a short exact sequence 
\beq \label{longexactsplits}1\longrightarrow \mathcal{G}_{0,\delta}\longrightarrow \mathcal{G}_{\delta}\longrightarrow \prod_i L_i^c\longrightarrow 1.\eeq
We are thinking here of each $L_i^c$ as sitting inside $\mathfrak{l}_{i}^{\mathbb{C}}$.  The sequence (\ref{longexactsplits}) in fact splits (at least when we are using regular $\alpha_i$ so that all $L_i^c$ are equal to the maximal torus in $\operatorname{SU}(N)$, see \cite[p. 26]{daska-went2}), so we have $\mathcal{G}_{\delta}=\mathcal{G}_{0,\delta} \rtimes\left(\prod_i L_i^c\right)$.
The spaces we are interested in are 
$$ \mathcal{F}_{\delta} = \mathcal{A}_{\delta,F}/\mathcal{G}_{0,\delta},\hspace{1cm} \mathcal{M}_{\delta}=\mathcal{A}_{\delta,F}/\mathcal{G}_{\delta},$$
and we have of course
 $$\mathcal{M}_{\delta}=\mathcal{F}_{\delta}/ \prod_i L_i^c.$$

One of the reasons for setting up the weighted Sobolev spaces (as opposed to just working with the Fr\'echet spaces  of smooth sections) is access to index theorems.  In particular we have the following, an application of an Atiyah--Patodi--Singer index theorem \cite{APS}.  Let 
$$\delta_{\nabla}=(\nabla, e^{-\tau\delta}\nabla^*e^{\tau\delta}): L^2_{1,\delta}(T^*(\Sigma)\otimes \mathfrak{g}_{P}) \rightarrow L^2_{1,\delta}\left(\Lambda^2 T^*\Sigma^o\otimes \mathfrak{g}_{P}\right)\oplus L^2_{1,\delta}(\mathfrak{g}_{P}).$$  
Here $\nabla^*$ is the $L^2$ adjoint of $\nabla$.  Then, for a small positive range of $\delta$, this operator $\delta_{\nabla}$ is bounded Fredholm, of index $2(g-1)(N^2-1) +\sum_i \dim \left(\operatorname{SU}(N)/L_i^c\right)$ \cite[Prop. 3.5]{daska-went1}.  Henceforth, we shall always assume $\delta$ to be in this range.

As a consequence, Daskalopoulos and Wentworth prove the following:

\begin{thm}[{\cite[Theorems 3.7 and 3.13]{daska-went1}}]\label{twomoduli}
The spaces $\mathcal{F}_{\delta}$ and $\mathcal{M}_{\delta}$ are smooth manifolds of dimensions  $(2(g-1)+n)(N^2-1)$ and $2(g-1)(N^2-1)+ \sum_i \dim(\operatorname{SU}(N)/L_i^c)$ respectively.  Moreover there is a diffeomorphism between $\mathcal{M}_{\delta}$ and the stable locus of $\mathcal{M}_{\overline{\alpha}}$.
Finally a complex structure on $\Sigma$ naturally puts the structure of an almost complex manifold on $\mathcal{M}_{\delta}$ that makes the diffeomorphism biholomorphic.
\end{thm}
 In \cite{daska-went1} only irreducible connections are discussed, but it is well-known that if one also includes reducible flat connections, one obtains a homeomorphism with all of $\mathcal{M}_{\overline{\alpha}}$, see e.g. \cite[Theorem 6.4]{poritz}.  The image of $\mathcal{G}_{\delta}$-orbits of reducible flat connections is exactly the semi-stable locus in $\mathcal{M}_{\overline{\alpha}}$.
 
 \subsubsection{Construction of $\mathcal{L}_{\CS}^k$}\label{constrLCS}
Now we define a cocycle that will be the analogue of \eqref{secondco}.  It will be convenient to define this cocycle in terms of the Chern--Simons action on the cylinder $[0,1]\times\Sigma^o$.  
Let 
$$\widetilde{\cG}_{0,\delta}:=\{\tilde{g}:[0,1]\times\Sigma^o\rightarrow G\,|\,\tilde{g}(t,\cdot)\in\cG_{0,\delta},\,\forall\,t\in[0,1],\,\text{and is continuous and piecewise smooth in}\,t\}.$$  Since every gauge transformation in $\cG_{0,\delta}$ is smoothly homotopic to the identity  (\cite[Prop. 3.3]{daska-went1}), we may extend $g\in\cG_{0,\delta}$ on $\Sigma^o$ to $\tilde{g}\in\widetilde{\cG}_{0,\delta}$ so that $\tilde{g}_{0}:=\tilde{g}(0,\cdot)=g$ and $\tilde{g}_{1}:=\tilde{g}(1,\cdot)=e$.  
Similarly, using  the natural projection map $\pi:[0,1]\times\Sigma^o\rightarrow \Sigma^o$, $\nabla_A=\nabla_0+A$ on $\Sigma^o$ extends to $\widetilde{\nabla_A}=\pi^{*}\nabla_A=d+\widetilde{A\!+\!A_0}$, where $\widetilde{A\!+\!A_0}=\pi^*(A\!+\!A_0)$.  Then $\widetilde{\nabla_A}^{\tilde{g}}\in\widetilde{\cA}_{\delta}$ is an extension of $\nabla_A^g$ to $[0,1]\times \Sigma^o$, and we define 
\beq\label{punctco}
\Theta^{k}(\nabla_A,g):=\operatorname{exp}\left(-2\pi ik\CS_{[0,1]\times\Sigma^o}(\widetilde{A\!+\!A_0}^{\tilde{g}})\right).
\eeq 
Recall that  Daskalopoulos and Wentworth in \cite[Eq. 5.1]{daska-went1} define a cocycle
$$
\widetilde{\Theta}^{k}:L^{2}_{1,\delta}(T^{*}\Sigma^o\otimes\mathfrak{g}_{P})\times \cG_{0,\delta}\rightarrow\operatorname{U}(1),
$$
$$
\widetilde{\Theta}^{k}(\nabla_A,g):=\operatorname{exp}\left(\frac{ik}{4\pi}\int_{\Sigma^o}\operatorname{tr}(\operatorname{Ad}_{g^{-1}}(A\!+\!A_0)\wedge g^{-1}dg)-\frac{ik}{12\pi}\int_{[0,1]\times \Sigma^o}\operatorname{tr}\left(\tilde{g}^{-1}\tilde{d}\tilde{g}\right)^{3}\right),
$$
and use this to define the Chern--Simons line bundle $\cL^{k}_{\CS}$ over $\cM_{\delta}$ (here $\tilde{d}=d+\frac{d}{dt}$).  Note that we have used $\operatorname{Ad}$-invariance of $\operatorname{tr}$ to write \cite[Eq. 5.1]{daska-went1} in a slightly different form than it originally appeared.  
The two cocycles are equal: 
\begin{lem}\label{equivlem}
$\widetilde{\Theta}^{k}(\nabla_A,g)=\Theta^{k}(\nabla_A,g).$
\end{lem}
\begin{proof}
We have the formula
\begin{eqnarray*}
8\pi^2 \CS_{[0,1]\times\Sigma^o}\left(\widetilde{A\!+\!A_0}^{\tilde{g}}\right)&=&\int_{[0,1]\times\Sigma^o}\operatorname{tr}\left(\widetilde{A\!+\!A_0}^{\tilde{g}}\wedge F_{\widetilde{\nabla_A}^{\tilde{g}}}-\frac{1}{6}\widetilde{A\!+\!A_0}^{\tilde{g}}\wedge[\widetilde{A\!+\!A_0}^{\tilde{g}}\wedge\widetilde{A\!+\!A_0}^{\tilde{g}}]\right).
\end{eqnarray*}
Now, the usual gauge change formula for the curvature is
$$
F_{\widetilde{\nabla_A}^{\tilde{g}}}=\operatorname{Ad}_{\tilde{g}^{-1}}F_{\widetilde{\nabla_A}}.
$$
and by definition
$$
\widetilde{A\!+\!A_0}^{\tilde{g}}=\operatorname{Ad}_{\tilde{g}^{-1}}\widetilde{A\!+\!A_0}+\tilde{g}^{*}\omega,
$$
where $\omega$ is the Maurer--Cartan form on $G$.  Since $\widetilde{A\!+\!A_0}=\pi^{*}(A\!+\!A_0)$, some of the forms involving only $\widetilde{A\!+\!A_0}$ on $[0,1]\times\Sigma^o$ vanish and we have
$$
\operatorname{tr}\left(\widetilde{A\!+\!A_0}^{\tilde{g}}\wedge F_{\widetilde{\nabla_A}^{\tilde{g}}}\right)=\operatorname{tr}\left(\tilde{g}^{*}\omega\wedge
\operatorname{Ad}_{\tilde{g}^{-1}}d\widetilde{A\!+\!A_0}+\frac{1}{2}\tilde{g}^{*}\omega\wedge\operatorname{Ad}_{\tilde{g}^{-1}}[\widetilde{A\!+\!A_0}\wedge\widetilde{A\!+\!A_0}]\right). 
$$
Also, one computes
\begin{multline*} \frac{1}{6}\operatorname{tr}\left(\widetilde{A\!+\!A_0}^{\tilde{g}}\wedge[\widetilde{A\!+\!A_0}^{\tilde{g}}\wedge\widetilde{A\!+\!A_0}^{\tilde{g}}]\right)=\\
\frac{1}{2}\operatorname{tr}\left(\tilde{g}^{*}\omega\wedge\operatorname{Ad}_{\tilde{g}^{-1}}[\widetilde{A\!+\!A_0}\wedge\widetilde{A\!+\!A_0}]\right)+\frac{1}{2}\operatorname{tr}\left(\operatorname{Ad}_{\tilde{g}^{-1}}\widetilde{A\!+\!A_0}\wedge[\tilde{g}^{*}\omega\wedge\tilde{g}^{*}\omega]\right)+\frac{1}{6}\operatorname{tr}\left(\tilde{g}^{*}\omega\wedge[\tilde{g}^{*}\omega\wedge\tilde{g}^{*}\omega]\right). 
\end{multline*}
We have
\begin{multline*}
8\pi^2 \CS_{[0,1]\times\Sigma^o}(\widetilde{A\!+\!A_0}^{\tilde{g}})=\\ \int_{[0,1]\times\Sigma^o}\operatorname{tr}\left(\tilde{g}^{*}
\omega\wedge\operatorname{Ad}_{\tilde{g}^{-1}}d\left(\widetilde{A\!+\!A_0}\right)-\frac{1}{2}\operatorname{Ad}_{\tilde{g}^{-1}}\widetilde{A\!+\!A_0}\wedge[\tilde{g}^{*}\omega\wedge\tilde{g}^{*}\omega]-\frac{1}{6}\tilde{g}^{*}\omega\wedge[\tilde{g}^{*}\omega\wedge\tilde{g}^{*}\omega]\right),
\end{multline*}
and one can show that
\beq\label{useq}
\operatorname{Ad}_{\tilde{g}^{-1}}d\left(\widetilde{A\!+\!A_0}\right)=d\left(\operatorname{Ad}_{\tilde{g}^{-1}}
\widetilde{A\!+\!A_0}\right)+[\tilde{g}^{*}\omega\wedge \operatorname{Ad}_{\tilde{g}^{-1}}
\widetilde{A\!+\!A_0}].
\eeq
Using \eqref{useq},
\begin{align*}
\int_{[0,1]\times\Sigma^o} \operatorname{tr}&\left(\tilde{g}^{*}
\omega \wedge \operatorname{Ad}_{\tilde{g}^{-1}}d\left(\widetilde{A\!+\!A_0}\right)\right)\\ =&\int_{[0,1]\times\Sigma^o}\operatorname{tr}\left(\tilde{g}^{*}
\omega\wedge\left[d\left(\operatorname{Ad}_{\tilde{g}^{-1}}
\widetilde{A\!+\!A_0}\right)+[\tilde{g}^{*}\omega\wedge \operatorname{Ad}_{\tilde{g}^{-1}}
\widetilde{A\!+\!A_0}]\right]\right)\\
 =&\int_{[0,1]\times\Sigma^o}\operatorname{tr}\left(\tilde{g}^{*}
\omega\wedge d\left(\operatorname{Ad}_{\tilde{g}^{-1}}
\widetilde{A\!+\!A_0}\right)+\operatorname{Ad}_{\tilde{g}^{-1}}
\widetilde{A\!+\!A_0}\wedge[\tilde{g}^{*}\omega\wedge \tilde{g}^{*}\omega]\right)\\
 =&\int_{[0,1]\times\Sigma^o}\operatorname{tr}\left(\operatorname{Ad}_{\tilde{g}^{-1}}
\widetilde{A\!+\!A_0}\wedge d\tilde{g}^{*}\omega+d\left(\operatorname{Ad}_{\tilde{g}^{-1}}
\widetilde{A\!+\!A_0}\wedge\tilde{g}^{*}\omega\right)+\operatorname{Ad}_{\tilde{g}^{-1}}
\widetilde{A\!+\!A_0}\wedge[\tilde{g}^{*}\omega\wedge \tilde{g}^{*}\omega]\right).\\
\end{align*}
Thus
$$
8\pi^2\CS_{[0,1]\times\Sigma^o}(\widetilde{A\!+\!A_0}^{\tilde{g}})=\int_{[0,1]\times\Sigma^o}\operatorname{tr}\left(d\left(\operatorname{Ad}_{\tilde{g}^{-1}}
\widetilde{A\!+\!A_0}\wedge\tilde{g}^{*}\omega\right)-\frac{1}{6}\tilde{g}^{*}\omega\wedge[\tilde{g}^{*}\omega\wedge\tilde{g}^{*}\omega]\right),
$$
since the Maurer-Cartan equation says
$$d\tilde{g}^{*}\omega+\frac{1}{2}[\tilde{g}^{*}\omega\wedge\tilde{g}^{*}
\omega]=0.$$
Since Stokes' theorem holds when $A\in L^{2}_{1,\delta}(T^{*}\Sigma^o\otimes\mathfrak{g}_{P})$, the Lemma is proven after writing $\tilde{g}^{*}\omega=\tilde{g}^{-1}\tilde{d}\tilde{g}$.
\end{proof}
Given Lemma \ref{equivlem}, one has again that $\Theta^{k}$ is independent of the choice of the path in $\cG_{0,\delta}$ (\cite[Lemma 5.2]{daska-went1}). 
The action of $\cG_{0,\delta}$ on $\cA_{\delta}\times\mathbb{C}$ is given by
\beq\label{liftgauge}(\nabla_A,z)\cdot g:=(\nabla_A^g,\Theta^{k}(\nabla_A,g)\cdot z),
\eeq
and we have the following
\begin{lem}\label{cocyclelem}
$\Theta^{k}$ satisfies the cocycle condition
$$
\Theta^{k}(\nabla_A,g)\Theta^{k}(\nabla_A^g,h)=\Theta^{k}(\nabla_A,gh).
$$\end{lem}
This corresponds to \cite[Lemma 5.3]{daska-went1}; as no proof is given there we include one here for completeness, using our construction of $\Theta^k$.
\begin{proof}
Let
$$
\tilde{h}_{1}:[0,1]\times\Sigma^o\rightarrow G,
$$
be an extension of $h$ from $\Sigma^o$ to $[0,1]\times\Sigma^o$ such that
$$\tilde{h}_{1}(0,\cdot)=h(\cdot),\,\,\,\,\,\text{and},\,\,\,\,\tilde{h}_{1}(1,\cdot)=e(\cdot),$$
where $e:\Sigma^o\rightarrow G$ is the identity gauge transformation.
Let
$$
\tilde{g}_{1}:[0,1]\times\Sigma^o\rightarrow G,
$$
be an extension of $g$ from $\Sigma^o$ to $[0,1]\times\Sigma^o$ such that
$$\tilde{g}_{1}(0,\cdot)=g(\cdot),\,\,\,\,\,\text{and},\,\,\,\,\tilde{g}_{1}(1,\cdot)=e(\cdot).$$
Define an extension $\tilde{h}_{0}:[0,1]\times\Sigma^o\rightarrow G$ of $h$ by:
$$
\tilde{h}_{0}(t,\cdot):=\begin{cases}
\tilde{h}_{1}(2t,\cdot),\,\,t\leq 1/2,\\
\pi^{*}e(2t-1,\cdot),\,\,t\geq 1/2,
\end{cases}
$$
so that,
$$\tilde{h}_{0}(0,\cdot)=h(\cdot),\,\,\,\,\,\text{and},\,\,\,\,\tilde{h}_{0}(1,\cdot)=e(\cdot).$$
Also, define an extension $\tilde{g}_{0}:[0,1]\times\Sigma^o\rightarrow G$ of $g$ by:
$$
\tilde{g}_{0}(t,\cdot):=\begin{cases}
\pi^{*}g(2t,\cdot),\,\,t\leq 1/2,\\
\tilde{g}_{1}(2t-1,\cdot),\,\,t\geq 1/2,
\end{cases}
$$
so that
$$\tilde{g}_{0}(0,\cdot)=g(\cdot),\,\,\,\,\,\text{and},\,\,\,\,\tilde{g}_{0}(1,\cdot)=e(\cdot).$$
By construction we have
\begin{align*}
\Theta^{k}(\nabla_A,gh)&=\operatorname{exp}\left[-2\pi ik\CS_{[0,1]\times\Sigma^o}\left(\widetilde{A\!+\!A_0}^{\widetilde{gh}}\right)\right]\\
                &=\operatorname{exp}\left[-2\pi ik\CS_{[0,1]\times\Sigma^o}\left(\widetilde{A\!+\!A_0}^{\tilde{g}_{0}\tilde{h}_{0}}\right)\right],\,\,\text{since $\Theta^{k}$ is independent of extension},\\
                &=\operatorname{exp}\left[-2\pi ik\left(\CS_{[0,\frac{1}{2}]\times\Sigma^o}\left(\widetilde{A\!+\!A_0}^{\tilde{g}_{0}\tilde{h}_{0}}\right)+
                \CS_{[\frac{1}{2},1]\times\Sigma^o}\left(\widetilde{A\!+\!A_0}^{\tilde{g}_{0}\tilde{h}_{0}}\right)
                \right)\right]
                \\
                &=\operatorname{exp}\left[-2\pi ik\left(\CS_{[0,1]\times\Sigma^o}
                \left(\left(\pi^{*}\left(A\!+\!A_0\right)^g\right)^{\tilde{h}_1}\right)+
                \CS_{[0,1]\times\Sigma^o}\left(\widetilde{A\!+\!A_0}^{\tilde{g}_{1}}\right)
                \right)\right],\,\,\text{by definition of $\tilde{h}_{0}, \tilde{g}_{0}$},\\
                &=\operatorname{exp}\left[-2\pi ik\left(\CS_{[0,1]\times\Sigma^o}
                \left(\left(\pi^{*}\left(A\!+\!A_0\right)^g\right)^{\tilde{h}_{1}}\right)+
                \CS_{[0,1]\times\Sigma^o}\left(\widetilde{A\!+\!A_0}^{\tilde{g}_{1}}\right)
                \right)\right]
                \\
                &=\Theta^{k}\left(\nabla_A^g,h\right)\cdot\Theta^{k}\left(\nabla_A,g\right),\,\,\text{by 
                definition of $\Theta^{k}$}.
\end{align*}
\end{proof}
Since $\Theta$ satisfies the cocycle condition and $\cG_{0,\delta}$ preserves flat connections, we obtain the induced Chern--Simons line bundle over $\mathcal{F}_{\delta}$.  This line bundle will not always descend to $\mathcal{M}_{\delta}$ however -- the weights from (\ref{parweights}) need to satisfy an integrality condition for that.  The full result is:

\begin{theorem}[{Cfr. \cite[Theorems 5.8 and 6.1]{daska-went1}}] \label{liftunderconditions}
Suppose that, for $k\in\mathbb{N}$, $k$ times the parabolic weights $\alpha_i$ from (\ref{parweights}) are in the co-character lattice of $L_i^c$ for each marked point $p_i\in \cP$, and that $k\sum_i \alpha_i$ is in the co-root lattice of $\operatorname{SU}(N)$.  Then the $k$-th power of the line bundle on $\mathcal{F}_{\delta}$ constructed above descends to $\mathcal{M}_{\delta}$.  It comes naturally equipped with a connection, whose curvature is $\frac{k}{2\pi i}$ times the symplectic form $\Omega$ from (\ref{formcompsup}).
\end{theorem}
Once again, strictly speaking this is only discussed in \cite{daska-went1} only for the locus of $\mathcal{M}_{\overline{\alpha}}$ consisting of irreducible connections.  It however carries over to the whole of $\mathcal{M}_{\overline{\alpha}}$: as discussed in \cite[page 268]{freed95}, one needs to be concerned only about the connected components of the stabilisers of flat connections acting trivially, and by \cite[Proposition 6.8]{andersen95}, in the case of $K=\SU(N)$, this reduces to checking that the centre of $\SU(N)$ acts trivially, which is indeed covered by \cite{daska-went1}.

By abuse of notation we shall refer to these line bundles as the \emph{Chern--Simons line bundles}, denoted by $\mathcal{L}^k_{\CS}$.  Recall that the (complex) codimension of the strictly semi-stable locus is at least two (except if $g=2, r=2$), hence by Hartogs' theorem this line bundle extends canonically to all of $\mathcal{M}_{\overline{\alpha}}$.  In fact, there is no obstruction to carrying the construction of \cite{daska-went1} of the line bundle $\mathcal{L}_{\operatorname{CS}}^k$ through also for reducible connections, which would construct $\mathcal{L}_{\operatorname{CS}}^k$ directly for all of $\mathcal{M}_{\overline{\alpha}}$.

\subsection{Lift of the mapping class group action}\label{liftCS}

In this section we will discuss how any suitable diffeomorphism of $\Sigma^o$ gives an action on $\mathcal{M}_{\overline{\alpha}}$ that lifts to the line bundle $\mathcal{L}^k_{\operatorname{CS}}$.  In fact we will do a little more, and show that this action factors through the mapping class group.  The diffeomorphisms, and isotopies, in question are supposed to preserve some first order information at the marked points, and there are a number of ways to encode this.  One could allow diffeomorphisms that only permute points in $\cP$ that carry the same label, and as before preserves some projective tangent vector there.  We will however choose a different description, dictated by our construction of $\mathcal{M}_{\overline{\alpha}}$ using weighted Sobolev spaces as outlined above.  In particular, we will only allow those diffeomorphisms that preserve the chosen local coordinates around marked points (only permuting those with equal weights), and all isotopies have to do the same.  It is a straightforward exercise that the mapping class group so obtained is isomorphic to the one where only projective tangent vectors are asked to be preserved.

\subsubsection{Diffeomorphisms}\label{diffeos}
Let 
$\operatorname{Diff}_{+}(\Sigma, \overline{z},\overline{\alpha})$ denote the orientation-preserving diffeomorphisms of $\Sigma$ preserving each subset of $\cP$ whose points carry the same weights, as well as their neighbourhoods $D_i$ and local coordinates $z_i$ we have chosen, i.e. $z_j\circ f=z_i$ if $f(p_i)=p_j$.  Our goal is to show that an analogue of Lemma \eqref{analem} holds for the punctured surface $\Sigma^o$ when $f\in \operatorname{Diff}_{+}(\Sigma, \overline{z},\overline{\alpha})$.   First remark that $\operatorname{Diff}_{+}(\Sigma, \overline{z},\overline{\alpha})$ acts by pullback on $\cA_{\delta}$, as by construction the weights used in the Sobolev norms are preserved.   
We lift this action to the trivial line bundle 
$\cA_{\delta}\times\mathbb{C}$ 
by
\beq\label{liftactionbis}
f^{*}(\nabla_A,z):=(f^{*}\nabla_A,z),
\eeq for $f\in \operatorname{Diff}_{+}(\Sigma, \overline{z},\overline{\alpha})$.  
As in the case without punctures, we can define a morphism $\Psi: \operatorname{Diff}_{+}(\Sigma, \overline{z},\overline{\alpha})\rightarrow\operatorname{Aut}(\mathcal{G}_{0,\delta})$ by $\Psi(f)(g):=g\circ f$, and we have again
\begin{lem}\label{liftlem}
The lifts (\ref{liftgauge}) and (\ref{liftactionbis}) combine to an action of $\mathcal{G}_{0,\delta}\rtimes_{\Psi}\operatorname{Diff}_{+}(\Sigma, \overline{z},\overline{\alpha})$ on $\mathcal{A}_{\delta}\times\mathbb{C}$.
\end{lem} 
As with Lemma \ref{combination} we have switched again to a right action of $\operatorname{Diff}_{+}(\Sigma, \overline{z},\overline{\alpha})$ to obtain a right action of the semi-direct product.
\begin{proof}
This statement reduces to showing that \begin{equation}\label{boil down}
f^{*}\left((\nabla_A,z)\cdot g\right)=\left(f^{*}(\nabla_A,z)\right)\cdot (g\circ f),
\end{equation}
for $f\in\operatorname{Diff}_{+}(\Sigma, \overline{z},\overline{\alpha}), g\in\cG_{0,\delta}$.  To see this, first observe that
\beq\label{firsteq}
(f^{*}\nabla_A)^{g\circ f}=f^{*}(\nabla_A^g),
\eeq
which is easy to show directly.  Then compute
\begin{align*}
\left(f^{*}(\nabla_A,z)\right)\cdot (g\circ f)&=\left((f^{*}\nabla_A)^{g\circ f},\Theta^{k}(f^{*}\nabla_A,g\circ f)\cdot z\right)\\
                            &=\left(f^{*}(\nabla_A^g),\Theta^{k}(f^{*}\nabla_A,g\circ f)\cdot z\right),\,\,\text{by \eqref{firsteq}},\\
                            &=f^{*}\left(\nabla_A^g,\Theta^{k}(f^{*}\nabla_A,g\circ f)\cdot z\right).
\end{align*} 
Hence to establish (\ref{boil down}) it suffices to show that
\beq\label{seceq}
\Theta^{k}(f^{*}\nabla_A,g\circ f)=\Theta^{k}(\nabla_A,g).
\eeq
The verification of \eqref{seceq} boils down to basic diffeomorphism invariance of integration on manifolds.  Indeed, by definition
$$
\Theta^{k}(f^{*}\nabla_A,g\circ f)=\operatorname{exp}\left(-2\pi ik\CS_{[0,1]\times\Sigma^o}\left(\pi^*\left({f^{*}(A\!+\!A_0)}\right)\right)^{\widetilde{g\circ f}}\right),
$$
and if we put $F=\operatorname{id}\times f$ we can write this as 
\begin{eqnarray*}
\Theta^{k}(f^{*}\nabla_A,g\circ f)&=\operatorname{exp}\left(-2\pi ik\CS_{[0,1]\times\Sigma^o}\left(F^{*}\left(\widetilde{A\!+\!A_0}\right)\right)^{\tilde{g}\circ F}\right)\\
  &=\operatorname{exp}\left(-2\pi ik\CS_{[0,1]\times\Sigma^o}\left(F^{*}\left(\widetilde{A\!+\!A_0}^{\tilde{g}}\right)\right)\right).
\end{eqnarray*}
By diffeomorphism invariance of integration on manifolds,
$$\CS_{[0,1]\times\Sigma^o}\left(F^{*}\left(\widetilde{A\!+\!A_0}^{\tilde{g}}\right)\right)=\CS_{[0,1]\times\Sigma^o}\left(\widetilde{A\!+\!A_0}^{\tilde{g}}\right),$$
and therefore (\ref{seceq}) follows.
\end{proof} 
Lemma \ref{liftlem} implies that the action of $\operatorname{Diff}_{+}(\Sigma, \overline{z},\overline{\alpha})$ on $\cA_{\delta}\times\mathbb{C}$ descends to an action on the Chern--Simons line bundle over $\mathcal{F}_{\delta}$.   As explained in \cite[\S 4.2 and \S 5.2]{daska-went1}, there are some rationality conditions on the weights to further descend the line bundle to $\mathcal{M}_{\delta}$, but when it comes to lifting the action of $\operatorname{Diff}_{+}(\Sigma, \overline{z},\overline{\alpha})$ these pose no further problem,  as the action of $\operatorname{Diff}_{+}(\Sigma, \overline{z},\overline{\alpha})$ commutes with the action of $\prod_i L_i^c$.  Hence we obtain as desired a lift of the action of $\operatorname{Diff}_{+}(\Sigma, \overline{z},\overline{\alpha})$ to $\mathcal{L}^k_{\CS}$, whenever the latter exists.

Finally, once again as in \cite[\S 7]{andersen95} we have the following basic observation:
\begin{lem}\label{repetition}Let $\nabla_A$ be a connection with prescribed holonomy such that $[\nabla_A]\in \mathcal{M}_{\overline{\alpha}}^f$.  We can then find a lift $\tilde{f}$ of $f$ to $P$ such that $\nabla_A$ is invariant under $\tilde{f}$.
\end{lem}
Remark however that while $f^m=\text{id}$, the same need not be the case for $\tilde{f}$.
\begin{proof}
Let $f':f^*(P)\rightarrow P$ be the natural bundle isomorphism covering $f$.  Then, since $[A]\in \mathcal{M}_{\overline{\alpha}}^f$, there exists a bundle isomorphism $\psi$ from $P$ to $f^*P$ (covering the identity on $\Sigma^o$) such that $\psi(\nabla_A)=f^*(\nabla_A)$.  The composition of $f'$ with $\psi$ gives us the desired $\tilde{f}$.
\end{proof}

Given such a $\nabla_A$ and $\tilde{f}$, we can create a connection $\nabla_{A_{\phi}}$ on the mapping torus $Q_{\phi}=[0,1]\times_{\phi} Q$ thought of as a bundle over $\Sigma^o_f=$.  We have the following:

\begin{lem}\label{tracelem}
\begin{equation}
\operatorname{tr}\left(f^{*}:\cL^{k}_{\CS}\Big|_{[\nabla_A]}\rightarrow \cL^{k}_{\CS}\Big|_{[\nabla_A]}\right)=\operatorname{exp}\left(2\pi ik\CS_{\Sigma^o_f}(Q_\phi,\nabla_{A_{\phi}})\right).
\end{equation}
\end{lem}
\begin{proof}  Without a loss of generality, we can assume that $\nabla_A$ takes the standard form $d + \alpha_i d\theta_i$ on $D_i$ \cite[Lemma 2.7]{daska-went1}.  
Given that $[f^{*}\nabla_A]=[\nabla_A]$ with $f\in \operatorname{Diff}_{+}(\Sigma, \overline{z},\overline{\alpha})$, there exists $g\in\cG_{0,\delta}$ such that
$$f^{*}\nabla_A=\nabla_A^g.$$
Then
\begin{align*}
f^{*}[(\nabla_A,z)]&=[(f^{*}\nabla_A,z)]\\
              &=[(\nabla_A^g,z)]\\
              &=[(\nabla_A,\Theta^{k}(\nabla_A,g)^{-1}\cdot z)],
\end{align*}
where by definition
$$\Theta^{k}(\nabla_A,g)=\operatorname{exp}\left(-2\pi ik\CS_{[0,1]\times\Sigma^o}\left(\widetilde{A\!+\!A_0}^{\tilde{g}}\right)\right),$$
and the connection $\widetilde{\nabla_A}^{\tilde{g}}$ descends to the connection $A_{\phi}$ on the bundle $Q_{\phi}=[0,1]\times_{\phi}Q$ over the open mapping torus $\Sigma^o_{f}:=[0,1]\times_{f}\Sigma^o$.  Thus
$$\Theta^{k}(\nabla_A,g)=\operatorname{exp}\left(-2\pi ik\CS_{\Sigma_f}(Q_\phi,\nabla_{A_{\phi}})\right),$$
and indeed
$$\operatorname{tr}\left(f^{*}:\cL^{k}_{\CS}\Big|_{[\nabla_A]}\rightarrow \cL^{k}_{\CS}\Big|_{[\nabla_A]}\right)=\operatorname{exp}\left(2\pi ik\CS(Q_\phi,\nabla_{A_{\phi}})\right).$$
\end{proof}

\subsubsection{Isotopy}
\label{isotopysection}
Suppose now that $f$ is an element of $\operatorname{Diff}_{0}(\Sigma, \overline{z},\overline{\alpha})$, the subgroup of $\operatorname{Diff}_{+}(\Sigma, \overline{z},\overline{\alpha})$ consisting of diffeomorphisms isotopic to the identity within $\operatorname{Diff}_{+}(\Sigma, \overline{z},\overline{\alpha})$.  Suppose we have a smooth isotopy $f_t$ given, with $f_0=\operatorname{id}$, and $f_1=f$. It is standard that the action of any $f_t$ on $\nabla_A\in\mathcal{A}_{\delta}$ can be understood as a gauge transformation.  Indeed, we can just let $g_t(p)$ be the holonomy of $\nabla_A$ along the path $s\mapsto f_{(1-t)(1-s)}(p)$.  Given that $f\in \operatorname{Diff}_{0}(\Sigma, \overline{z},\overline{\alpha})$ immediately implies that $g_t\in \mathcal{G}_{0,\delta}$ for all $t$.   This shows that $\operatorname{Diff}_{0}(\Sigma, \overline{z},\overline{\alpha})$ acts trivially on $\mathcal{M}_{\overline{\alpha}}$.  Moreover we have that

\begin{prop}
For $\nabla_A$ and $g_t$ as above, we have that $\Theta^k(\nabla_A,g_0)=1$.
\end{prop}
\begin{proof} 
Let $\widetilde{\nabla_A }= \pi^* \nabla_A$  be as in (\ref{punctco}).  We consider the gauge transformation $\tilde g$ over the surface cylinder induced by $g_t$ from above. We observe that $g_t$ is constructed exactly such that ${\nabla_A}^{g_t}$ has trivial holonomy along the curves $t\mapsto f_{t}(p)$, for $t\in [0,1]$ and any $p\in \Sigma^o$, which also carries over to $\widetilde{\nabla_A}^{\tilde{g}}$ having trivial holonomy along the curves $t\mapsto (t, f_t(p))$ in $[0,1]\times \Sigma^o$. Now we define the diffeomorphism
$$ F : [0,1]\times \Sigma^o  \rightarrow [0,1]\times \Sigma^o $$
by the formula
$$F(t,p) = (t,f_t(p)).$$
We observe that $F^*\widetilde{\nabla_A}^{\tilde g}$ is trivial along the lines $t\mapsto (t,p)$ for all $p\in \Sigma^o$, since $F$ maps these lines to the curves $t\mapsto (t, f_t(p))$. This implies that
$$\operatorname{CS}_{[0,1]\times \Sigma^0}\left(F^*\left(\widetilde{A\!+\!A_0}^{\tilde g}\right)\right) = 0.$$
We conclude that
$$ \operatorname{CS}_{[0,1]\times \Sigma^0}\left(\widetilde{A\!+\!A_0}^{\tilde g}\right) = \operatorname{CS}_{[0,1]\times \Sigma^0}\left(F^*\left(\widetilde{A\!+\!A_0}^{\tilde g}\right)\right) = 0.$$
\end{proof}

\begin{cor}
We have an induced action of the mapping class group  
$$
 \operatorname{Diff}_{+}(\Sigma, \overline{z},\overline{\alpha})\ \Big/\ \operatorname{Diff}_{0}(\Sigma, \overline{z},\overline{\alpha})$$on $\mathcal{M}_{\overline{\alpha}}$ with a lift to $\mathcal{L}^k_{\operatorname{CS}}$.
\end{cor}

\begin{rem}
\label{dehntwistboundaryremark}
Note that for the action on $\mathcal{M}_{\overline{\alpha}}$ and $\mathcal{L}^k_{\operatorname{CS}}$, we could actually also allow the diffeomorphisms  to `rotate' the local coordinates around the marked points (i.e. such that $z_j\circ f= e^{\vartheta_i}z_i$ if $f(p_i)=p_j$, for some $\vartheta_i\in\mathbb{R}$, or more generally simply such that the function $\tau$ is preserved).  The crucial thing to observe is that, though isotopies that may rotate the $z_i$ act trivially on $\mathcal{M}_{\overline{\alpha}}$, they do not on $\mathcal{L}^k_{\operatorname{CS}}$.  

In the set up that we have used, where the $z_i$ are preserved by the diffeomorphisms $f$, this corresponds to the fact that the Dehn twists around the marked points act trivially on $\mathcal{M}_{\overline{\alpha}}$.  This was also observed in \cite{charles} (in the context of surfaces with boundary), where the character with which the Dehn twists act on the fibres of $\mathcal{L}^k_{\operatorname{CS}}$ was also determined. Charles does not allow for boundary components to be permuted but the result is otherwise the same. We have included in Appendix~\ref{dehntwistappendix} the explicit evaluation of this character in our setup.
\end{rem}
\section{Conformal blocks}\label{sectionconfblocks}
In this Section we relate the space of holomorphic sections of the Chern--Simons line bundle $\mathcal{L}^k_{\CS}$ (for a given choice of complex structure $\sigma$ on $\Sigma$) constructed in Section \ref{csbundle} to the space of conformal blocks as defined in \cite{TUY}.  The idea that these spaces are linked goes back to Witten's first paper \cite{witten-jones}; our aim is mainly to make this explicit and mathematically rigorous for the various precise definitions of Chern--Simons bundle and space of conformal blocks we use.  This allows us to link the sections of the Chern--Simons bundle with the  Reshetikhin--Turaev invariants \cite{RT1,RT2,Turaev}, as the modular functor that determines these can be described using conformal blocks \cite{AU2}.  Once again, essentially all of the ingredients for this are in the literature, but we are unaware of any place where they are linked in the way we need them.  In Section \ref{non-equi} below we outline the isomorphism, drawing on various known descriptions and correspondences.  In Section \ref{equi} we then discuss how the isomorphism can be shown to be $f$-equivariant, which is crucial for our application.

\subsection{Conformal blocks and the stack of quasi-parabolic bundles}\label{non-equi}
Given a Riemann surface $\Sigma_{\sigma}$ and a divisor of marked points $\cP$ as above, a parabolic subgroup $P_i$ of $\operatorname{SL}(N,\mathbb{C})$ corresponding to a flag type for every point $p_i$ of $\cP$ (the $L_i^c$ used in Section \ref{csbundle} are compact forms of the Levi factors of these), Tsuchiya, Ueno and Yamada in \cite{TUY} construct a corresponding space of conformal blocks $\mathcal{V}^{\dagger}_{N,k,\overline{\lambda}}(\Sigma_{\sigma},\cP)$  with it for every level $k\in \mathbb{N}$ and $\overline{\lambda}=(\lambda_1,\ldots,\lambda_n)$, where the (integral) weights $\lambda_i$ lie in the Weyl alcove at level $k$, and moreover in  the wall of the Weyl chamber corresponding to $P_i$,  for every point in $\cP$ (we will review the construction below in Section \ref{actionsconfblocks}, see also \cite{looij,coordfree} for coordinate-free constructions).  
In turn, this is linked to moduli of parabolic bundles by Pauly \cite{pauly} and Laszlo--Sorger \cite{laszlo-sorger}.  In particular, these authors consider the stack $\mathfrak{M}_{\Sigma_{\sigma},\cP,P_i}$ of quasi-parabolic bundles.  Its Picard group is given by 
\beq\label{picstack}\text{Pic}(\mathfrak{M}_{\Sigma_{\sigma},\cP,P_i})=\mathbb{Z} \oplus \bigoplus_i \mathfrak{X}(P_i),
\eeq where the latter terms are the character lattices of the parabolic groups, i.e. $\text{Hom}(P_i,\mathbb{G}_m)$, which can be identified with the Picard group of the flag varieties $\operatorname{SL}(N,\mathbb{C})/P_i$.    For the structure group $\operatorname{SL}(N,\mathbb{C})$, the $\mathbb{Z}$-summand is generated by the \emph{determinant of cohomology line bundle}, or determinant line bundle for short.  It assigns to a family $\mathcal{F}$ of quasi-parabolic bundles parametrized by $S$ the line bundle whose fibre over $s\in S$ is given by $\Lambda^{\text{top}}(H^0(\Sigma_{\sigma},\mathcal{F}(s))^*)\otimes \Lambda^{\text{top}}(H^1(\Sigma_{\sigma},\mathcal{F}(s)))$ -- when thinking in terms of principal bundles, note that we use the standard representation of $\operatorname{SL}(N,\mathbb{C})$ to define this.  The determinant line bundle only depends on the underlying (non-parabolic) bundle, and not on the parabolic structure of $\mathcal{F}$.

Following \cite{teleman} we say a line bundle $\mathcal{L}_{(k,\overline{\lambda})}$ on $\mathfrak{M}_{\Sigma_{\sigma},\cP,P_i}$ is \emph{semipositive} if in the above presentation of the Picard group it is given by $(k, \overline{\lambda})$, where $k\geq 0$ and the $\lambda_i$ are dominant weights, necessarily in the face of the Weyl chamber corresponding to $P_i$, with $\langle\lambda_i,\theta\rangle\leq k$.  The line bundle is \emph{positive} if all these inequalities are strict, and $\lambda$ is moreover regular.  Given a positive $\mathcal{L}_{(k,\overline{\lambda})}$, the complement of the base locus of all powers of $\mathcal{L}_{(k,\overline{\lambda})}$ is the semi-stable locus of the stack, denoted by $\mathfrak{M}_{\Sigma_{\sigma},\cP,P_i}^{\mathcal{L}_{(k,\overline{\lambda})}-\operatorname{ss}}$.  It consists of those quasi-parabolic vector bundles that are semi-stable for the weights $\alpha_i=\frac{\lambda_i}{k}$.  

We now have the following result by Pauli and Sorger: 
\begin{thm}[{\cite[Prop 6.5 and 6.6]{pauly}, \cite[Thm 1.2]{laszlo-sorger}}]\label{paulassor}
Given a line bundle $\mathcal{L}_{(k,\overline{\lambda})}$ on $\mathfrak{M}_{\Sigma_{\sigma},\cP,P_i}$ as above, there exists an isomorphism
$$H^0(\mathfrak{M}_{\Sigma_{\sigma},\cP,P_i},\mathcal{L}_{(k,\overline{\lambda})})\cong \mathcal{V}^{\dagger}_{N,k,\overline{\lambda}}(\Sigma_{\sigma},\cP),$$ which is canonical once the local coordinates are chosen.  
\end{thm}
This theorem generalises earlier results of Beauville--Laszlo \cite{beaulas} and Kumar, Narasimhan and Ramanathan \cite{KNR} to the parabolic case.

Moreover, for structure group $\operatorname{SL}(N,\mathbb{C})$, most of these line bundles descend to the moduli spaces $\mathcal{M}_{\overline{\alpha}}$ (essentially due to a descent lemma of Kempf \cite[Thm. 2.3]{DZ}, see also \cite[Theorem 10.3]{Alper} -- in the terminology of the latter $\mathcal{M}_{\overline{\alpha}}$ is a \emph{good moduli space} for the stack $\mathfrak{M}_{\Sigma_{\sigma},\cP,P_i}^{\mathcal{L}_{(k,\overline{\lambda})}-\operatorname{ss}}$).  One just needs to verify that the stabilisers of (semi-)stable bundles act trivially on the fibre of the line bundle over the corresponding closed point of $\mathfrak{M}_{\Sigma_{\sigma},\cP,P_i}$, which for $\operatorname{SL}(N,\mathbb{C})$ reduces to verifying that the exponential $e^{\sum_i \lambda_i}$ is trivial on the centre of $\operatorname{SL}(N,\mathbb{C})$.  
In this case we shall refer to  the descent of $\mathcal{L}_{(k,\overline{\lambda})}$ to $\mathcal{M}_{\overline{\alpha}}$ with $\alpha_i=\frac{\lambda_i}{k}$ as the parabolic determinant bundle, following \cite{biswas-ragha}, and denote it as $\mathcal{L}^k_{\operatorname{pd}}$.   We have 
\begin{thm}[{\cite[Thm 9.6]{teleman}, \cite[5.2]{pauly}}]\label{cohvanishing}
Whenever $\sum_{i} \lambda_i$ lies in the root lattice of $\operatorname{SL}(N, \mathbb{C})$,  
there exists a canonical isomorphism (up to scalars)
$$H^0(\mathfrak{M}_{\Sigma_{\sigma},\cP,P_i},\mathcal{L}_{(k,\overline{\lambda})})\cong H^0(\mathcal{M}_{\overline{\alpha}},\mathcal{L}^k_{\operatorname{pd}}),$$  where $\alpha_i=\frac{\lambda_i}{k}$.  Moreover all higher cohomology of these line bundles vanishes.
\end{thm}

We want to link this with the Chern--Simons line bundles $\mathcal{L}^k_{\CS}$ constructed earlier.  It suffices to show that the line bundle $\mathcal{L}^k_{\operatorname{pd}}$ considered above can be given a connection whose curvature is the $k$-th multiple of the K\"ahler class.  Indeed, the same is true for the $k$-th power of Chern--Simons bundle, and since we know that the moduli spaces are simply connected (by Theorem \ref{simply}), all line bundles are determined by their curvature \cite[ Cor. II.9.2]{kob-nom}.  Hence the Chern--Simons line bundle $\mathcal{L}^k_{{\CS}}$ and $\mathcal{L}^k_{\operatorname{pd}}$ are isomorphic as bundles with connections, and therefore also as holomorphic line bundles for the natural holomorphic structure on $\mathcal{L}^k_{{\CS}}$ (see e.g. \cite[Thm 5.1]{ahs}).  

By the seminal work of Quillen \cite{quillen}, there exists a natural `regularised' Hermitian metric on determinant line bundles over  moduli spaces such as $\mathcal{M}_{\overline{\alpha}}$.  From the discussion above however, it follows that we are not interested in the determinant line bundle (which is $\mathcal{L}_{(1,\overline{0})}$ in the notation above) itself, but rather by a twist of the determinant line bundle (or a power thereof) by a line bundle coming from the parabolic structures, chosen to correspond to the weights\footnote{Though less relevant for our approach as we work with punctured Riemann surfaces, the matter of the Quillen metric in the case of Riemann surfaces with boundary was discussed in \cite{chang}.}.  There is, however, another viewpoint on parabolic bundles (for rational weights), as they correspond to orbifold bundles, or alternatively, equivariant bundles on a suitable ramified cover of $\Sigma_{\sigma}$ (also referred to as \emph{$\pi$-bundles}).  In particular, under this correspondence the moduli spaces of semi-stable bundles are isomorphic (but their natural K\"ahler structures differ by a factor).

Biswas and Raghavendra take the latter approach and show that, on the stable locus of the moduli space of $\pi$-bundles, the $\pi$-determinant bundle equipped with the Quillen metric has as curvature $N$ times the natural K\"ahler form \cite[Theorem 3.27]{biswas-ragha} (see also \cite{bisw3,bisw2}).  Moreover, using this $\pi$-determinant they show \cite[\S 5]{biswas-ragha} that on the moduli space of parabolic bundles $\mathcal{M}_{\overline{\alpha}}$, there exists a metrized line bundle (which they dub the \emph{parabolic determinant}; it corresponds to the descent of the line bundle $\mathcal{L}_{(Np,Np\overline{\alpha})}$ we considered before on the stack, where $p$ is the least common multiple of all denominators in the $\alpha_{i,j}$), whose curvature is $\frac{Np}{2\pi i}$ times the K\"ahler form $\Omega$ from (\ref{formcompsup})
on $\mathcal{M}_{\overline{\alpha}}$ \cite[Theorem 5.3]{biswas-ragha}. 
In particular this implies \beq\label{frombiswas}c_1\left(\mathcal{L}^k_{\operatorname{pd}}\right)=k[\Omega ].\eeq
 To be precise, they do this without fixing the determinant of $E$ and then need a correction factor in the line bundle, which is however trivial in the fixed-determinant case.
 
We can conclude (recall again our assumption that all $\lambda_i$ are regular)
\begin{cor}\label{bundleisom}
The line bundles $\mathcal{L}^k_{\CS}$ and $\mathcal{L}^k_{\operatorname{pd}}$ are isomorphic as holomorphic line bundles on $\mathcal{M}_{\overline{\alpha}}$.
\end{cor}

Note that an alternative approach, 
 avoiding the use of $\pi$-bundles, would be to use the recent work of Zograf and Takhtajan \cite{tak-zog}, who calculate the curvature of the Quillen metric, not on the determinant line bundle but on the canonical bundle of the moduli space.  They show in particular that its curvature is equal to a multiple of the K\"ahler form, minus a `cuspidal defect', which they express in terms of natural curvature forms on line bundles coming from the parabolic bundles.  We can on the other hand identify the canonical bundle for the moduli stack of quasi-parabolic bundles.  Indeed, it follows from e.g. \cite[Thm 8.5]{laszlo-sorger} that $\mathfrak{M}_{\Sigma, \mathcal{P},P_i}$ 
 is a flag bundle (locally trivial in the \'etale topology) over the stack of (non-parabolic) bundles $\mathfrak{M}_{\Sigma_{\sigma}}$.  The canonical bundle of $\mathfrak{M}_{\Sigma_{\sigma}, \mathcal{P},P_i}$ is 
 therefore the tensor product of the canonical bundle of $\mathfrak{M}_{\Sigma_{\sigma}}$ with the canonical bundle of the various flag varieties $G/P_i$ (as before $G=\operatorname{SL}(N,\mathbb{C})$.  The former is determined by the weight $-2 h^{\vee}$ \cite[Cor. 10.6.4]{sorger-lectures}, the latter is well-known to be $-2\rho_{P_i}$ \cite[Page 202]{jantzen}, where $h^{\vee}$ is the dual Coxeter number of $G$ and $\rho_{P_i}$ is half the sum of those positive roots of $G$ that determine $P_i$.  Combining these two expressions should lead to the same result, giving (a multiple of) the K\"ahler form on $\mathcal{M}_{\overline{\alpha}}$ as curvature for a connection on the line bundle $\mathcal{L}^k_{\operatorname{pd}}$, 
 
 but we were unable to resolve some ambiguities with respect to normalisation conventions.

We can finally summarise all the results quoted in this Section as 
\begin{thm}\label{confquant}
There exists an isomorphism, canonical up to scalars, between the space of conformal blocks and the K\"ahler quantisation of the moduli space of flat connections using the Chern--Simons line bundle:
$$\mathcal{V}^{\dagger}_{N,k,\overline{\lambda}}(\Sigma_{\sigma},\cP)\cong H^0( \mathcal{M}_{\overline{\alpha}}, \mathcal{L}_{\CS}^k),$$
if $k$ is such that all $\lambda_i=k\alpha_i$ are elements of the lattices $\mathfrak{X}(P_i)$, and $\sum_i\lambda_i$ is in the root lattice of $\operatorname{SL}(N,\mathbb{C})$.

\end{thm}

\subsection{Equivariance}\label{equi}
For our purposes it is very important to establish the isomorphism given in Theorem \ref{confquant} as an $f$-equivariant isomorphism.  In Section \ref{liftCS} a lift of any diffeomorphism $f\in \operatorname{Diff}_{+}(\Sigma, \overline{z},\overline{\alpha})$  to the Chern--Simons line bundle was constructed.  This will induce an action on the geometric (K\"ahler) quantisation only if the chosen complex structure is also preserved by $f$.  In such a case $f$ is finite order, and vice versa, for every finite order diffeomorphism one can choose a complex structure $\sigma$ on $\Sigma$ preserved by it.  We shall now assume such a complex structure to be chosen, and we shall discuss the corresponding equivariance of the spaces of conformal blocks and non-abelian theta functions.  All throughout we shall consider $f$ to be a finite order automorphism of the Riemann surface which preserves the set of labelled marked points (i.e. subsets of $\mathcal{P}$ are only allowed to be permuted if the corresponding $\lambda_i$ are the same) that gives rise to an automorphism of a marked surface $\overline{\Sigma}$ (as in the introduction) whose surface is $\Sigma$.  

We begin by making an elementary observation.
\begin{lem}
If $f$ is a finite order diffeomorphism of $\Sigma$ that preserves the set of marked points $\cP$, and some choice of non-zero tangent vectors at the $p_i$ up to real positive scalars, then necessarily all $f$-orbits in $\cP$ are generic, i.e. their lengths are equal to $m$, the order of $f$.
\end{lem}
\begin{proof}As mentioned above, we can pick a complex structure preserved by $f$, so that $f$ is an automorphism of $\Sigma_{\sigma}$, e.g. by choosing an $f$-invariant metric.  It now suffices to show that $f$ cannot fix any of the $\cP$ (indeed, if there were no fixed points in $\cP$, but some orbit was not generic, then a suitable power of $f$ would fix that orbit and we could replace $f$ by that power).  Suppose that $p$ were an $f$-fixed point, then we can choose a holomorphic disk around $p$ preserved by $f$.  It is well known that automorphisms of holomorphic disks preserving the centre have to be rotations, and since we also know that $f$ preserves a tangent vector at $p$ up to real positive scalars, $f$ has to be the identity on this disk, hence everywhere on $\Sigma$.
\end{proof}
This of course implies that $n$ has to be a multiple of $m$.  It also implies that we can find open disks $D_i$ around the $p_i$ with local holomorphic coordinates $z_i$ giving isomorphisms onto the unit disk in $\mathbb{C}$, and that are preserved by $f$, i.e. $z_j \circ f = z_i$, if $f(p_i)=p_j$.  Vice versa, given an automorphism $f$ that preserves such a choice of holomorphic coordinates around the points in $\cP$, we can choose tangent vectors such that $f$ is an automorphism of the marked surface (in any case, for finite order automorphisms of the marked surface we can always find a suitable normalisations of the tangent vectors, so that we need not be concerned about the $\mathbb{R}_+$-ambiguity).  We shall therefore from now on assume that such a choice of coordinates has been made, and that the constructions in Sections \ref{constr} and \ref{constrLCS} are done with respect to these neighbourhoods.

\subsubsection{Action on spaces of conformal blocks}\label{actionsconfblocks}

From the work of \cite{TUY} it follows that the bundle of conformal blocks is a vector bundle over the stack of smooth marked curves $\mathcal{M}_{g,n}$ (in fact even over its Deligne--Mumford compactification $\overline{\mathcal{M}}_{g,n}$, though that does not concern us here).  In particular this implies that for every Riemann surface with automorphisms, there exists a canonical action of the automorphism group on the corresponding space of conformal blocks.  Moreover this story goes through even if one allows automorphisms that can interchange marked points with identical labels. This is explained in detail in \cite{AU2}.

In our case, as $f$ comes from an automorphism of a marked surface, a concrete description of this action is provided by \cite[Proposition 4.3]{AU2}, in terms of the construction of the spaces of conformal blocks.  As we will need a minor variation on this description to link it with the non-abelian theta-functions, we outline it here.

Given a semi-simple (complex) Lie algebra $\mathfrak{g}$, we will denote by $\widehat{\mathfrak{g}}_n$ the central extension of $\bigoplus_{j=1}^n \mathfrak{g}\otimes \mathbb{C}(\!(\xi_j)\!)$.  Here $\xi_j$ is a local coordinate at the $j$-th marked point, and the Lie bracket of the central extension is determined by thinking of this Lie algebra $\widehat{\mathfrak{g}}_n$  as a Lie subalgebra of the $n$-fold sum of the affine Lie algebra of $\mathfrak{g}$: $$\widehat{\mathfrak{g}}_n\subset \bigoplus_{j=1}^n \widehat{\mathfrak{g}}.$$  We fix a level $k$ and $n$ weights $\lambda_i\in \mathfrak{X}(P_k)$, and look at the representation of $\bigoplus_{j=1}^n \widehat{\mathfrak{g}}$ given term-wise by the corresponding representations of level $k$ and weight $\lambda_i$ of $\widehat{\mathfrak{g}}$: 
$$\mathcal{H}_{k,\overline{\lambda}}=\mathcal{H}_{k,\lambda_1}\otimes \ldots\otimes \mathcal{H}_{k,\lambda_n}.$$
Now we begin by defining the dual of the space of conformal blocks, also known as the space of co-vacua:
$$\mathcal{V}_{\mathfrak{g},k,\overline{\lambda}}(\Sigma_{\sigma},\cP)=
\mathcal{V}_{\mathfrak{g},k,\overline{\lambda}}= \mathcal{H}_{k,\overline{\lambda}}/ \widehat{\mathfrak{g}}(U_{\mathcal{P}})\mathcal{H}_{k,\overline{\lambda}}.$$
Here $U_{\mathcal{P}}=\Sigma_{\sigma} \setminus \mathcal{P}$ and $\widehat{\mathfrak{g}}(U_{\mathcal{P}})$ is the Lie algebra $\mathfrak{g}\otimes \mathcal{O}(U_{\cP})$.  This is a Lie subalgebra of $\bigoplus_{j=1}^n \mathfrak{g}\otimes \mathbb{C}(\!(\xi_j)\!)$, and since by the residue theorem the restriction of the central extension of the latter to this subalgebra splits, we can think of it as a subalgebra of $\widehat{\mathfrak{g}}_n$.  The space of conformal blocks $\mathcal{V}^{\dagger}_{\mathfrak{g},k,\overline{\lambda}}(\Sigma_{\sigma},\cP)=\mathcal{V}^{\dagger}_{\mathfrak{g},k,\overline{\lambda}},$ also known as the space of vacua, is the dual to the space of co-vacua.  Alternatively, we can put 
$$\mathcal{V}^{\dagger}_{\mathfrak{g},k,\overline{\lambda}}=\left(\mathcal{H}^*_{k,\overline{\lambda}}\right)^{\widehat{\mathfrak{g}}(U_{\cP})},$$ i.e. the   $\widehat{\mathfrak{g}}(U_{\cP})$-invariant subspace of $\mathcal{H}^{*}_{k,\overline{\lambda}}$.  When $\mathfrak{g}=\mathfrak{sl}(N,\mathbb{C})$  we use the notation $\mathcal{V}^{\dagger}_{N,k,\overline{\lambda}}$ and $\mathcal{V}_{N,k,\overline{\lambda}}$.

We now want to describe the induced action of $f$ on $\mathcal{V}^{\dagger}_{\mathfrak{g},k,\overline{\lambda}}$.  In order to do so we will assume that, besides the marked points on the Riemann surface, we have chosen a divisor $\mathcal{Q}$ consisting of one generic orbit of $f$, i.e. one point $p\in \Sigma_{\sigma}$ that is not part of the marked points in $\cP$, as well as all of its $f$-translates $$\cQ=p+ f(p)+ f^2(p)+ \ldots+ f^{m-1}(p),$$  where $m$ is the order of $f$ -- the genericity of the orbit implies that these $m$ points are all distinct.  To each of the points in this orbit we assign the weight $0$ -- hence by the propagation of vacua we have an explicit isomorphism between the spaces of conformal blocks (Theorem 4.5 in \cite{AU2}), and in the description above we are using $U_{\mathcal{P}\cup\mathcal{Q}}$.  We assume that we have chosen $f$-compatible formal neighbourhoods around these new marked points, which we can always do.

The action of $f$ on $\mathcal{H}_{k,\overline{\lambda}}$ and $\mathcal{H}_{k,\overline{\lambda}}^*$ is now simply given by the permutation on the tensor factors, determined by the permutation by $f$ of the marked points.  As the latter action preserves $\widehat{\mathfrak{g}}(U_{\mathcal{P}\cup\mathcal{Q}})$ this induces an automorphism of $\mathcal{V}^{\dagger}_{\mathfrak{g},k,\overline{\lambda}}$.  

\begin{rem} This description appears slightly different from the one in  \cite[Proposition 4.3]{AU2}, because there one changes the labelling of the marked points (and local coordinates) by $f$, and therefore the isomorphism between the spaces of conformal blocks is induced by the identity on $\mathcal{H}_{\overline{\lambda}}$.
\end{rem}

\subsubsection{Alternative description}
The link between the spaces of conformal blocks and spaces of non-abelian theta functions given in \cite{laszlo-sorger} and \cite{pauly} depends on a variation of the construction of the covacua: 

\begin{prop}[{\cite[Proposition 2.3]{beauville}}] Let $\cP$ and $\cQ$ be two non-empty finite disjoint subsets of $\Sigma_{\sigma}$, with tuples of weights $\overline{\lambda}$ and $\overline{\mu}$ assigned to them respectively (all weights are assumed to be in the Weyl alcove at level $k$).  For any weight $\lambda$ we let $V_{\lambda}$ be the corresponding irreducible highest weight module of $\mathfrak{g}$, and put $V_{\overline{\lambda}}=V_{\lambda_1}\otimes V_{\lambda_2} \otimes \dots $.  If $\lambda\in\Lambda_N^{(k)}$ then we  consider $V_{\lambda}$ as a subspace of $\mathcal{H}_{k,\lambda}$.  We then have that the induced natural maps 
$$\Big(V_{\overline{\lambda}} \otimes  \mathcal{H}_{k,\overline{\mu}}\Big)\ \Big/\  \widehat{\mathfrak{g}}(U_{\mathcal{Q}})\Big(  V_{\overline{\lambda}} \otimes  \mathcal{H}_{k,\overline{\mu}}  \Big) \longrightarrow \Big(\mathcal{H}_{k,\overline{\lambda}}\otimes \mathcal{H}_{k,\overline{\mu}}\Big)\ \Big/\ \widehat{\mathfrak{g}}(U_{\mathcal{P}\cup\mathcal{Q}})\Big(\mathcal{H}_{k,\overline{\lambda}} \otimes  \mathcal{H}_{k,\overline{\mu}}\Big) \ = \ \mathcal{V}_{\mathfrak{g},k,\overline{\lambda},\overline{\mu}} $$ 
and 
$$\Big( V^*_{\overline{\lambda}} \otimes   \mathcal{H}^*_{k,\overline{\mu}}\Big)^{\widehat{\mathfrak{g}}(U_{\mathcal{Q}})} \rightarrow \Big(\mathcal{H}_{k,\overline{\lambda}}\otimes \mathcal{H}_{k,\overline{\mu}}\Big)^{\widehat{\mathfrak{g}}(U_{\cP\cup \cQ})}=\mathcal{V}^{\dagger}_{\mathfrak{g},k,\overline{\lambda},\overline{\mu}} $$  
are isomorphisms to the spaces of covacua and vacua respectively.  Here $\mathfrak{g}(U_{\cQ})$ acts on each of the $V_{\lambda_i}$ and $V^*_{\lambda_i}$ by evaluation at the corresponding point of $\cP$.
\end{prop}
We will use the alternative description of the spaces of conformal blocks offered by this proposition in the particular case where $\cP$ is the set of chosen marked points as before, and $\cQ$ is the disjoint extra generic $f$-orbit we have chosen (any will do).  It is clear how the $f$-action carries over to this description: again one simply takes the induced action by permuting the factors in the tensor products $V^*_{\overline{\lambda}}$ and $\mathcal{H}^*_{k,\overline{\mu}}$.

\subsubsection{Equivariance of line bundles over stack}
The proof of Theorem \ref{paulassor} stated above of Pauly and Laszlo--Sorger relies on a presentation of the stack $\mathfrak{M}_{\Sigma_{\sigma},\cP,P_i}$ as a quotient of a product of an affine Grassmannian and some flag varieties.  We need a minor variation:

\begin{prop}\label{uniform} Let $\Sigma_{\sigma}, \cP$ and $\cQ$ be as above.  
There is an isomorphism of stacks
$$\mathfrak{M}_{\Sigma_{\sigma},\cP,P_i}\isom G(U_\cQ)\ \Big\backslash\ \mathfrak{Q}^{\operatorname{par}}_G,$$ where \begin{equation}\label{product}\mathfrak{Q}^{\operatorname{par}}_G=\prod_{\cQ}\mathfrak{Q}_G \times \prod_{\cP} G/P_i,\end{equation}  which is canonical once local coordinates are chosen.  Here $\mathfrak{Q}_G$ is the affine Grassmannian, and $G(U_{\cQ})$ is the group of algebraic morphisms from $U_{\cQ}$ to $G$.
\end{prop}
The case where $\# \cQ=1$ is the one proven by Pauly and Laszlo--Sorger.  
\begin{proof} This is a direct and straightforward generalisation of the proofs of \cite[Proposition 4.2]{pauly} and \cite[Theorem 8.5]{laszlo-sorger}.
\end{proof}
We want to study the induced action of $f$ on $\mathfrak{M}_{\Sigma_{\sigma},\cP,P_i}$, given by pulling back bundles (recall that we assume that $f$ satisfies the conditions described at the beginning of Section \ref{equi}). In the above presentation this induced action is again straightforward:

\begin{lem}\label{liftalg}
Under the isomorphism of Proposition \ref{uniform} the action of $f$ on $\mathfrak{M}_{\Sigma_{\sigma},\cP,P_i}$ is induced by the permutation of the factors of (\ref{product}) by $f$.
\end{lem}

The proof of Theorem \ref{paulassor} in \cite{pauly} and \cite{laszlo-sorger} uses the presentation of Proposition \ref{uniform} for the stack $\mathfrak{M}_{\Sigma_{\sigma},\cP,P_i}$, as well as the Borel--Weil--Bott theorem, both the standard one for flag varieties $G/P$ and the affine version due to Kumar \cite{kumar} and Mathieu \cite{mathieu1,mathieu2} that realises representations of the affine Lie algebra $\widehat{\mathfrak{g}}$ as sections of line bundles over $\mathfrak{Q}_G$. In our case, as we are replacing a single `dummy' point by an entire $f$-orbit $\cQ$, the morphism $\mathfrak{Q}^{\operatorname{par}}_G\rightarrow \mathfrak{M}_{\Sigma_{\sigma},\cP,P_i}$ is no longer locally trivial in the \'etale topology, hence it no longer induces an isomorphism of Picard groups.  

We no longer have in general that the Picard group of $\mathfrak{M}_{\Sigma_{\sigma},\cP,P_i}$ is equal to the Picard group of $\mathfrak{Q}^{\operatorname{par}}_G$, but we can easily describe the pullbacks of the line $\mathcal{L}_{(k,\overline{\lambda})}$ to $\mathfrak{Q}^{\operatorname{par}}_G$.  Indeed, the Picard group of $\mathfrak{Q}^{\operatorname{par}}_G$ is isomorphic to $\mathbb{Z}^m\oplus \bigoplus_i \mathfrak{X}(P_i)$, and we have
\begin{lem}The pullback of the line bundle $\mathcal{L}_{(k,\overline{\lambda})}$ to $\mathfrak{Q}^{\operatorname{par}}_G$ is given by $(k,\ldots, k, \overline{\lambda})$.
\end{lem}
\begin{proof}
It suffices to remark that for each $j=1,\ldots, m$, we can consider the morphism $\mathfrak{Q}_G\times \prod_i G/P_i \rightarrow \mathfrak{Q}^{\operatorname{par}}_G$ that sends the factor $\mathfrak{Q}_G$ as the identity to the $j$-th factor, and enters the trivial element of $\mathfrak{Q}_G$ for all other factors.  Pulling back line bundles under this morphism gives the morphism of Picard groups that in the above presentation can be described as $(k_1,\ldots, k_n, \overline{\lambda})\mapsto (k_j, \overline{\lambda})$.  Finally, we can factor the morphisms to $\mathfrak{M}_{\Sigma_{\sigma},\cP,P_i}$ through this, and remark that for each $j$ we have the commutative diagram 
\begin{center}
\begin{tikzpicture}
\matrix (n)[matrix of math nodes,  row sep=3em, column sep=2.5em, text height=1.5ex, text depth=0.25ex]
{\mathfrak{Q}_G\times \prod_i G/P_i & \\
\mathfrak{Q}^{\operatorname{par}}_G & \mathfrak{M}_{\Sigma_{\sigma},\cP,P_i}\\};
\path[->] 
(n-1-1) edge (n-2-1) edge  node[auto] {$\Psi$} (n-2-2);
\path[->]
(n-2-1) edge (n-2-2);

\end{tikzpicture}
\end{center}
where $\Psi$ induces an isomorphism of Picard groups.
\end{proof}
Hence if we lift $f$ to act on the line bundle corresponding to $(k,\ldots,k,\overline{\lambda})$ on $\mathfrak{Q}^{\operatorname{par}}_G$ by simply permuting the factors, this induces a lift of the action of $f$ on $\mathfrak{M}_{\Sigma_{\sigma},\cP,P_i}$ to $\mathcal{L}_{(k,\overline{\lambda})}$, and we can conclude:
\begin{prop}
The action of $f$ on $\mathfrak{M}_{\Sigma_{\sigma},\cP,P_i}$ lifts to the line bundle $\mathcal{L}_{(k,\overline{\lambda})}$ so that the isomorphism of Theorem \ref{paulassor} is $f$-equivariant (with the action of $f$ on $\mathcal{V}^{\dagger}_{N,k,\overline{\lambda}}$ as described in Section \ref{equi}).
\end{prop}
\subsubsection{Equivariance of isomorphism with Chern--Simons line bundle}

We have now two isomorphic line bundles $\mathcal{L}^k_{\CS}$ and $\mathcal{L}^k_{\operatorname{pd}}$ on $\mathcal{M}_{\overline{\alpha}}$, and two lifts of the action of $f$ on $\mathcal{M}$ to them, which we want to show is the same.  

\begin{prop}\label{nogequi}
The isomorphism of Corollary \ref{bundleisom} is $f$-equivariant.
\end{prop}

As the moduli space is projective, both lifts of $f$ can differ at most by a character of the finite cyclic group generated by $f$, and it would suffice to verify this character is trivial at a single point in $\mathcal{M}_{\overline{\alpha}}$ to obtain the result.  In the non-parabolic case, i.e. where no link is present (cfr. \cite{andersen95}), one can do this at the trivial bundle / trivial connection.  In the parabolic case this is not so simple, as, for given parabolic weights, the trivial bundle may not be stable for any choice of flags.  Nevertheless one can make this reasoning work.

The main thing to note is that the lift of the action constructed in Section \ref{liftCS} comes from a lift of the action to the (trivial) line bundle over all of $\mathcal{A}_{\delta}$ (as in (\ref{Adelta})).  This means that the lift of the action also descends to the line bundle over the quotient $$\left[\mathcal{A}_{\delta} / \mathcal{G}_{\delta}^{\mathbb{C}}\right],$$ interpreted as a stack (here $\mathcal{G}_{\delta}^{\mathbb{C}}$ is the complexification of $\mathcal{G}_{\delta}$), which at least morally speaking is the same as the stack $\mathfrak{M}_{\Sigma_{\sigma},\cP,P_i}$.  Hence it suffices to show that the lift agrees for any bundle/connection, not necessarily a semistable or flat one.

To make this reasoning precise, we can argue as follows:
\begin{proof}[Proof of Proposition \ref{nogequi}]
Without a loss of generality we can assume that the base connection $\nabla_0$ chosen in Section \ref{constr} is $f$-invariant.  By (\ref{liftactionbis}) we have of course that $f$ acts trivially on the line over $\nabla_0$.  Moreover, with this assumption the action of $f$, induced by pulling back connections, on $L^2_{1,\delta}(T^*\Sigma^o \otimes \mathfrak{g}_{P})$ (as in (\ref{Adelta})) is linear.  Take now a flat connection, represented by a smooth element of $L^2_{1,\delta}(T^*\Sigma^o \otimes \mathfrak{g}_{P})$, as well as all of its $f$-orbit.  These generate a complex finite-dimensional sub-vector space  $V$ of $L^2_{1,\delta}(T^*\Sigma^o \otimes \mathfrak{g}_{P})$, consisting of smooth elements.  We think of this subspace as parametrising a family of connections 
on  $\Sigma^o$, and using the complex structure on $\Sigma_{\sigma}$, we take all of the associated complex structures given by the $(0,1)$-part of the connections.  These all extend to the whole of $\Sigma_{\sigma}$ as (quasi-)parabolic bundles, as in e.g. \cite[\S 4]{poritz}.  This gives us now an algebraic family of parabolic bundles parametrised by $V$, hence a morphism $\phi$ from $V$ into $\mathfrak{M}_{\Sigma_{\sigma},\cP,P_i}$.  As we had chosen the connection $\nabla_0$ to be $f$-invariant, the same will be true for the corresponding complex structure and parabolic structures, and by using the presentation of $\mathfrak{M}_{\Sigma_{\sigma},\cP,P_i}$ as in Proposition \ref{uniform}, we can represent this by an element of $\mathfrak{Q}^{\operatorname{par}}_G$ that is invariant under the permutations induced by $f$.  Finally, by Lemma \ref{liftalg} and the lift of $f$ to the line bundle $(k,...,k,\overline{\lambda})$ we have chosen, we see that also the action of $f$ on the line over the corresponding point in $\mathfrak{M}_{\Sigma_{\sigma},\cP,P_i}$ is trivial.  This implies that the two line bundles we can consider on $V$ are $f$-equivariantly isomorphic (since also linearizations of actions on affine spaces are unique up to characters).  If we now restrict to the subset of $V$ whose connections are flat, or equivalently to the subvariety where the corresponding complex structures are semi-stable, we can descend the line bundles to the moduli space $\mathcal{M}_{\overline{\alpha}}$.  This shows that the lifts over the $f$-orbit we have chosen are identical, and as the moduli spaces are projective we can conclude that indeed the line bundles $\mathcal{L}^k_{\CS}$ and $\mathcal{L}_{\operatorname{pd}}^k$ are $f$-equivariantly isomorphic over the whole of $\mathcal{M}_{\overline{\alpha}}$.
\end{proof}

We can therefore conclude this Section with

\begin{cor}
The isomorphism of Theorem \ref{confquant} between the K\"ahler quantisation of the moduli space of flat connections using the Chern--Simons line bundle and the spaces of conformal blocks is $f$-equivariant.
\end{cor}

\subsection{Further comments}
Though the Chern--Simons and parabolic determinant bundles are isomorphic, it should be noted that they have a different nature.  In \cite{freed95} Freed stresses that Chern--Simons theory takes values in the CS bundle, not in the determinant bundle.  This is also reflected in \cite{fujita}, where it is shown that when one considers both of these bundles over the moduli space of Riemann surface, they differ by tensoring by the Hodge bundle.

\section{Main results}\label{endgame}

\subsection{Fixed points and the CS functional}\label{discussion}

We now want to consider the mapping torus $\Sigma_f$ of $\Sigma$, which contains the link $L$ that is the mapping torus of $\cP$.  The complement of $L$ in $\Sigma_f$ is $\Sigma^o_f$, the mapping torus of $\Sigma^o$.

Recall that the fundamental group of $\Sigma^o_f$ can be written as 
\beq\label{fundmaptor}
  \pi_f := \pi_1(\Sigma^o_f) = \pi_1 (\Sigma^o) \rtimes_f \mathbb{Z} = \left\langle\pi_1(\Sigma^o),\eta \relmiddle| \eta^{-1}\gamma\eta = f_*\gamma \text{ for all } \gamma \in \pi_1(\Sigma^o) \right\rangle
\eeq

We shall denote by $\mathcal{M}_{\Sigma_f,L,\overline{\alpha}}$  the moduli space of flat connections on $\Sigma^o_f$ whose holonomy around the $i$-th component of the link $L$ lies in the conjugacy class of $e^{\alpha_i}$ (we orient the links compatibly with the orientation of $\Sigma_{\sigma}$, such that an oriented frame in $\Sigma$ together with a vector in the positive `time' direction gives an oriented frame for $\Sigma_f$).  Alternatively, one can think of $\mathcal{M}_{\Sigma_f,L,\overline{\alpha}}$ again as a moduli space of representations of $\pi_f$ in $K=\operatorname{SU}(N)$.

We can restrict connections on $\Sigma^o_f$ to $\Sigma^o$, giving rise to a map $r:\mathcal{M}_{\Sigma_f,L,\overline{\alpha}}\rightarrow \mathcal{M}_{\overline{\alpha}}$, with the image of this map in fact being contained in the fixed point locus $\mathcal{M}_{\overline{\alpha}}^f$.  As in \cite[\S 7]{andersen95}, one sees that over the part of the fixed point locus consisting of irreducible connections, this map is a $\lvert Z(G)\rvert$-fold cover, though we will not use this directly.  The main points of relevance for us are the following: firstly, recall from Lemma \ref{repetition} that
given a $\nabla_A$ with $[\nabla_A]\in\mathcal{M}^f_{\overline{\alpha}}$, we can create a connection $\widetilde{\nabla_A}$ on the mapping torus $P_{\tilde{f}}$ thought of as a bundle over $\Sigma^o_f$.  
Of particular relevance is that this implies that $\CS_{\Sigma^o_f}(\widetilde{A\!+\!A_0})$ only depends on the restriction of a connection to $\Sigma^o$ (using Lemma \ref{tracelem}).  In particular the Chern--Simons functional takes the same value on all components of $\mathcal{M}_{\Sigma_f,L,\overline{\alpha}}$ that restrict to the same component in $\mathcal{M}_{\overline{\alpha}}^f$, the $f$-fixed point locus in $\mathcal{M}_{\overline{\alpha}}$.

\subsection{Localization}
We will  use  the following special form of the  Lefschetz--Riemann--Roch theorem of Baum, Fulton and Quart.  

\begin{thm}[\cite{baum-fulton-macpherson1979_Riemann-Roch},\cite{baum-fulton-quart_LefschetzRiemann-Roch}]\label{LRR}
 Let $M$ be a projective variety, $\mathcal{L}$ a line bundle over $M$, and $f$ a finite order automorphism of $M$ that is lifted to $\mathcal{L}$.  Then we have
\beq\label{LRRformula}
\sum_i (-1)^i\operatorname{tr}\left( f: H^i(M, \mathcal{L}^k)\rightarrow H^i(M, \mathcal{L}^k)\right)= \sum_{\gamma} a^k_{\gamma} \operatorname{ch}(\mathcal{L}^k\big|_{M^{\gamma}}) \cap \tau_{\bullet}\circ L^{\gamma}_{\bullet}(\mathcal{O}_{M}),
\eeq 
where $\operatorname{ch}$ is the Chern character, the sum is over all the fixed point components $M^{\gamma}$ of the action of $f$ on $M$, $a_{\gamma}^k$ is the number by which $f$ acts on $\mathcal{L}^k\big|_{M_{\gamma}}$, and $$L_{\bullet}^{\gamma}: K_0^{eq}(M)\rightarrow K_0 (M^{\gamma})\otimes \mathbb{C}$$ and $$\tau_{\bullet}:K_0(M^{\gamma})\rightarrow H_{\bullet}(M^{\gamma})$$ are as defined in \cite[\S 2]{ baum-fulton-macpherson1979_Riemann-Roch} and \cite[page 180]{baum-fulton-quart_LefschetzRiemann-Roch} respectively.  
If a fixed point component $M^\gamma$ is contained in the smooth locus of $M$ then 
\beq\label{smoothcontribution}
\operatorname{ch}(\mathcal{L}^k\big|_{M^{\gamma}}) \cap \tau_{\bullet}\circ L^{\gamma}_{\bullet}(\mathcal{O}_{M})
= \exp{ (k c_1(\mathcal{L}\mid_{M^\gamma}))} \cup\operatorname{ch}(\lambda_{-1}^\gamma M)^{-1} \cup \operatorname{Td}(T_{M^\gamma}) \cap [M^\gamma],
\eeq
where $\lambda_{-1}^\gamma M$ is as defined in \cite[p. 31]{andersen95}.
\end{thm}
A general overview of the theorem and its ingredients was given in \cite[Appendix B]{andersen-himpel2011}; we refer to this for further details.

We can now move on to 

\begin{proof}[Proof of Theorem \ref{MainTheorem}]
We apply Theorem \ref{LRR} to the $f$-equivariant line bundle $\mathcal{L}^k_{\operatorname{pd}}$ over $\mathcal{M}_{\overline{\alpha}}$.  By Theorem \ref{cohvanishing} we have that the higher cohomology vanishes, hence the LHS of (\ref{LRRformula}) is exactly (\ref{maintrace}).   We now combine (\ref{general}), (\ref{combitrace}) and (\ref{abcontrib}). In the RHS of (\ref{LRRformula}) the $a_{\gamma}^k$ are by Lemma \ref{tracelem} and the discussion in Section \ref{discussion} exactly equal to $e^{2\pi i k q_{\gamma}}$, where $q_{\gamma}$ is the corresponding element in $\CS(\mathcal{M}_{\Sigma_f,L,\overline{\alpha}})$. Here we use the interpretation of the Chern--Simons functional given in Appendix~\ref{chernsimonsboundaryappendix}; the correspondence with Lemma~\ref{tracelem} is given in Lemma~\ref{correspondencewitholdCS}. 
If $\mathcal{M}^{\gamma}_{\overline{\alpha}}$ is smooth  we have by (\ref{frombiswas}) that $\operatorname{ch}\left(\mathcal{L}^{k}_{\CS}\big|_{\mathcal{M}^{\gamma}_{\overline{\alpha}}}\right)=\exp\left(k\Omega\big|_{\mathcal{M}^{\gamma}_{\overline{\alpha}}}\right)$, and we shall below abuse notation and denote by $\left[\Omega\right]$ the cohomology class $\frac{1}{k}c_1\left(\mathcal{L}^k_{\operatorname{CS}}\right)$  on all of $\mathcal{M}_{\overline{\alpha}}$.
Hence we arrive at 
\beq\label{endresult}\begin{split}Z_{N}^{(k)}(\Sigma_f,L,\overline{\lambda}) & = \operatorname{Det}(f)^{-\frac{1}{2}\zeta} \sum_{\gamma} e^{2\pi i k q_{\gamma}} 
\exp\left(k\left[\Omega\right]\big|_{\mathcal{M}^{\gamma}_{\overline{\alpha}}}\right) \cap \tau_{\bullet} (L_\bullet^c(\mathcal{O}_{\mathcal{M}_{\overline{\alpha}}}))
\\ & = \operatorname{Det}(f)^{-\frac{1}{2}\zeta} \sum_{\gamma} \left( e^{2\pi i k q_{\gamma}} \sum_{h=0}^{d_c} \left(\frac{1}{h!} \left(\left[\Omega\right]\big|_{\mathcal{M}^{\gamma}_{\overline{\alpha}}}\right)^h\cap \tau_\bullet(L_\bullet^c(\mathcal{O}_{\mathcal{M}_{\overline{\alpha}}}))\right)k^h\right), \end{split}
\eeq
establishing Theorem \ref{MainTheorem}, when combined with formula (\ref{smoothcontribution}) above. 
\end{proof}

\subsection{Growth rate conjecture}\label{fromsoeren}
Finally, we discuss the Growth Rate Conjecture for the link $L$ in $\Sigma_f$, i.e. Theorem \ref{GRT} from the introduction.
In the closed case, this is a statement about dimensions of fixed point components in relation to the dimensions of certain twisted de Rham cohomology groups (or, alternatively, group cohomology groups) that come from the relevant mapping torus. The cohomology groups enter the picture as they correspond to tangent spaces of the moduli spaces under consideration. In the parabolic case, the tangent space at a given conjugacy class of a $\pi_1$-representation is no longer the full cohomology group but rather \emph{parabolic group cohomology} (cfr. \cite{biswas}), and so these will be the groups of interest to us.

\subsubsection{(Parabolic) group cohomology}
Let us just briefly recall the construction of  group cohomology in low rank.  Let $\pi$ be any group. A \emph{$\pi$-module} is an abelian group $N$ with a left action of $\pi$. The elements of $N$ invariant under the action will be denoted $N^{\pi}$. A \emph{cocycle on $\pi$ with values in $N$} is a map $u : \pi \to N$ satisfying the cocycle condition
\begin{align*}
  u(gh) = u(g) + gu(h).
\end{align*}
A \emph{coboundary} is a cocycle of the form $g \mapsto \delta m(g) := m - gm$ for some $m \in N$. The set of cocycles is denoted $Z^1(\pi,N)$, and the set of coboundaries is denoted $B^1(\pi,N)$. We define the first cohomology group of $\pi$ with coefficients of $N$ as the quotient
\begin{align*}
  H^1(\pi,N) = Z^1(\pi,N) / B^1(\pi,N).
\end{align*}
Notice that an element of $N$ satisfies $\delta m \equiv 0$ exactly when $m \in N^{\pi}$. We are thus led to define
\begin{align*}
  H^0(\pi,N) = N^{\pi}.
\end{align*}

We will now be interested in the case $\pi=\pi_1(\Sigma^o)$ -- recall that we had given a particular presentation of this in (\ref{fundpunctures}).
Let $\pi$ denote the fundamental group of a genus $g$ surface $\Sigma$ with $n$ punctures.

Now, for every $\rho \in \Hom(\pi,K)$, $\pi$ acts on $\mathfrak{k} = \Lie(K)$ by $\gamma . v = \Ad(\rho(\gamma))v$, and we will denote by $Z^i(\pi,\Ad \rho)$ and $B^i(\pi, \Ad \rho)$ the corresponding spaces of cocycles and coboundaries as described in the previous section.

We say that a cocycle $u \in Z^1(\pi, \Ad \rho)$ is \emph{parabolic} if for every $i=1,\dots, n$ there is a $\mu_i \in \mathfrak{k}$ such that $u(a_i) = \mu_i - \Ad \rho(a_i) \mu_i$. The space of parabolic cocycles will be denoted $Z^1_{\mathrm{par}}(\pi , \Ad \rho)$. Then \cite[(1.1)]{biswas} says that for an equivalence class $\rho \in \calM_{\overline{\alpha}}$ (using the incarnation of (\ref{charvar})), the Zariski tangent space at $\rho$ is
\[
  T_{[\rho]}\calM \cong Z^1_{\mathrm{par}}(\pi,\Ad \rho)/B^1(\pi,\Ad \rho) =: H^1_{\mathrm{par}}(\pi,\Ad \rho),
\]
which we refer to as the \emph{first parabolic (group) cohomology}.

\subsubsection{Mapping tori and dimensions of fixed point components}

Just like above, we can talk about the (parabolic) group cohomology of $\pi_f$ (note that in this section we do not need to assume that $f$ is of finite order). That is, suppose we have an element of $\operatorname{Hom}(\pi_f,K)$ that restricts to $\rho$ on $\pi=\pi_1(\Sigma^o)$, and that maps the newly introduced generator $\eta$ of $\pi_f$ -- see (\ref{fundmaptor}) -- to $g\in K$.  We shall denote this homomorphism by $(\rho,g)$, we let $H^0(\Sigma_f,\Ad(\rho,g))$ be the $\pi_f$-invariant part of $\mathfrak{k}$, and $H^1_{\mathrm{par}}(\Sigma_f,\Ad(\rho,g))$ as above; note that it still makes sense to talk about parabolic cocycles by regarding the boundary homotopy classes $a_i$ from $\pi$ as elements in $\pi_f$. 

\begin{lem}
We have
$$H^1_{\operatorname{par}}(\pi_f, \Ad(\rho,g))\cong \mathrm{image}\left(H^1_c(\Sigma_f\setminus L, \mathfrak{k}_{\ad})\rightarrow H^1(\Sigma_f\setminus L, \mathfrak{k}_{\ad}) \right).$$
\end{lem}
\begin{proof}We will mirror the discussion in \cite[p. 537]{biswas}.  The representation $(\rho,g): \pi_f\mapsto K$ induces a linear representation of $\pi_f$ on $\mathfrak{k}$, and hence a flat connection on the adjoint bundle, and it is standard that in this way we have \beq\label{identification}H^1\left(\pi_f, \Ad(\rho,g)\right)\cong H^1\left(\Sigma_f\setminus L, \mathfrak{k}_{\ad}\right),\eeq since $\Sigma^o_f$ is aspherical.  By slicing $\Sigma_f$ at $t=0$ and by choosing a closed neighbourhood $D_i$ around each marked point $p_i$, whose (oriented) boundary $S^1_i=\partial D_i$ is homotopic to $a_i\in \pi\subset\pi_f$.  If we let $\widetilde{D_i}$ be the corresponding neighbourhood of the $i$-th link component in $\Sigma_f$, with boundary $\widetilde{S^1_i}$, then we have a map $$\Xi: H^1\left(\Sigma_f\setminus \bigcup_i \widetilde{D_i}^o, \mathfrak{k}_{\ad}\right)\rightarrow H^1\left(\bigcup_i \widetilde{S^1_i}, \mathfrak{k}_{\ad}\right).$$  Like for surfaces, it follows from \cite[Theorem 1]{prasad} that under the identification (\ref{identification}) we have $$H^1_{\operatorname{par}}(\pi_f, \Ad(\rho,g))\cong \ker (\Xi).$$  Finally remark that in the long exact sequence for pairs \begin{multline*}\ldots\longrightarrow H^0\left(\bigcup_i \widetilde{S^1_i}, \mathfrak{k}_{\ad}\right)\longrightarrow H^1\left(\Sigma_f \setminus \bigcup_i \widetilde{D_i}^o, \bigcup_i \widetilde{S^1_i}, \mathfrak{k}_{\ad}\right)\\ \longrightarrow H^1\left(\Sigma_f \setminus \bigcup_i \widetilde{D_i}^o, \mathfrak{k}_{\ad}\right)\stackrel{\Xi}{\longrightarrow} H^1\left(\bigcup_i \widetilde{S^1_i}, \mathfrak{k}_{\ad}\right)\longrightarrow\ldots\end{multline*} we can replace $H^1\left(\Sigma_f \setminus \bigcup_i \widetilde{D_i}^o, \bigcup_i \widetilde{S^1_i}, \mathfrak{k}_{\ad}\right)$ by $H^1_c\left(\Sigma_f \setminus L, \mathfrak{k}_{\ad}\right)$ and likewise $H^1\left(\Sigma_f \setminus \bigcup_i \widetilde{D_i}^o, \mathfrak{k}_{\ad}\right)$ by \\$H^1\left(\Sigma_f \setminus L, \mathfrak{k}_{\ad}\right)$.
\end{proof}
For any $f^*$-fixed point $[\rho] \in \calM_{\overline{\alpha}}$ with $g$ as before, one verifies that $f$ acts on $H^1_{\mathrm{par}}(\pi,\Ad \rho)$ by mapping $u$ to $\gamma \mapsto \ad(g) u(f_*\gamma)$. We denote this action also by $f^*$.
\begin{prop}\label{goodjobsoeren}
  Let $[\rho] \in \calM_{\overline{\alpha}}$ be a fixed point of $f^*$, and let $[(\rho,g)] \in \calM_{\Sigma_f^o, L,\overline{\alpha}}$ for a suitable $g \in K$. Then the $1$-eigenspace $E_1(f^*,\rho)$ of the action of $f^*$ on $T_{[\rho]}\calM_{\overline{\alpha}}\cong  H^1_{\mathrm{par}}(\pi,\Ad \rho)$ has dimension
  \beq\label{E1dim}
    \dim E_1(f^*,\rho) = \dim H^1_{\mathrm{par}}(\pi_f, \Ad (\rho,g) ) - \dim H^0(\pi_f,\Ad (\rho,g)).
  \eeq
\end{prop}
\begin{rem}  In the non-parabolic case this was proven using a Mayer--Vietoris sequence (cfr. \cite[\S 7]{andersen95}).  As we could not find any corresponding reference in the parabolic case, and as we in fact only need a small part of the exact sequence, we take  a slightly different approach here, based on an explicit and concrete description of an exact sequence inspired by the Wang exact sequence used in \cite[\S 5.1]{andersen-himpel2011}.  This is presented explicitly in terms of generators of $\pi_f$, which is the reason we work with (parabolic) group cohomology rather than twisted de Rham cohomology. 
  
\end{rem}
\begin{proof}[{Proof of Proposition \ref{goodjobsoeren}}]
We claim that there is an exact sequence
  \begin{align*}
  0 \longrightarrow H^0(\pi_f,& \Ad(\rho,g)) \stackrel{\phi^0}{\longrightarrow} H^0(\pi,\Ad\rho) \stackrel{\mu^0}{\longrightarrow} H^0(\pi,\Ad\rho)\\ \ & \stackrel{\delta}{\longrightarrow} H^1_{\mathrm{par}}(\pi_f,\Ad(\rho,g)) \stackrel{\phi^1}{\longrightarrow} H^1_{\mathrm{par}}(\pi,\Ad\rho) \stackrel{\mu^1}{\longrightarrow} H^1_{\mathrm{par}}(\pi,\Ad\rho) \longrightarrow \cdots,
  \end{align*}
  where the maps are the following:
  \begin{align*}
    \phi^0(v) &= v, \\
    \mu^0(v) &= v-\Ad(g)v,\\
    \delta(v)(\gamma) &= u_v(\gamma) = \begin{cases} 0, & \gamma \in \pi, \\ v, & \gamma = \eta, \end{cases} \\
    \phi^1(u) &= u|_{\pi},\\
    \mu^1(u) &= u-\Ad(g) \circ f^* u.
  \end{align*}
  Here,
  \[
    (\Ad(g) \circ f^* u)(\gamma) = \Ad(g) u(f_*\gamma).
  \]
  Assuming that we have this exact sequence, the Proposition then follows: first of all, the dimension of the $1$-eigenspace of $\Ad(g) \circ f^*$ on $H^1_{\mathrm{par}}(\pi, \Ad \rho)$ is exactly the dimension of $\ker \mu^1$, and one finds that this equals the right hand side of (\ref{E1dim}) by applying exactness for each of the first six maps. For this, the precise expressions for each of the maps are not needed.
  
  Let us ensure that all of the maps are well-defined and that indeed the sequence is exact.
  First of all, if $v \in \mathfrak{k}$ is $\pi_f$-invariant, then $v \in \mathfrak{k}$ is $\pi$-invariant, so $\phi^0(v) \in H^0(\pi, \Ad \rho)$. Likewise, if $v$ is $\pi$-invariant, then so is $\Ad(g)v$ (use the action of $f_*\gamma$ on $v$), and the first exactness claim follows as $v$ is $\pi_f$-invariant if and only if $v = \Ad(g)v$.
  
  That $\phi^1$ is well-defined is obvious. That $\im\, \phi^1 = \ker \mu^1$ boils down to showing that $u = \Ad(g) \circ f^*u$ for $u \in H^1(\pi_f,\Ad(\rho,g))$, which on the other hand follows by the cocycle condition, as
  \begin{align*}
    (\Ad(g)f^*u)(\gamma) &= \Ad (\rho(\eta)) u(f_*\gamma) = u(\eta f_*\gamma) - u(\eta) \\
      &= u(\gamma \eta) - u(\eta) = u(\gamma) + \Ad(\rho(\gamma))u(\eta) - u(\eta) \\
      &= u(\gamma) - \delta(u(\eta))(\gamma).
  \end{align*}
  At this point, we should note that the same calculation also shows that if $u$ is parabolic, then so is $\Ad(g) \circ f^* u$, and one thus finds that $\mu^1$ is in fact well-defined.
  
  To see that $\delta$ is well-defined, one readily checks that $u_v(\gamma \eta) = u_v(\eta f_*\gamma)$ for $\gamma \in \pi$. It is clear that $\im \,\delta \subseteq \ker \phi^1$ and that $\im \mu^0 \subseteq \ker \delta$. Assume now that $[u_v] = 0$. That is, that there exists a $\mu \in \frakg$ such that $u_v(\gamma) = \mu - \Ad(\rho(\gamma)) \mu$ for all $\gamma \in \pi_f$. Since $u_v(\gamma) = 0$ for $\gamma \in \pi$, we see that $\mu \in H^0(\pi,\Ad\rho)$. Moreover, $v = u_v(\eta) = \mu - \Ad(\rho(\eta))\mu$, and so $\ker \delta \subseteq \im \mu^0$. Finally, to see that $\ker \phi^1 \subseteq \im \delta$ suppose that $u(\gamma) = \mu - \Ad (\rho (\gamma))\mu$ for some $\mu \in \frakg$ and all $\gamma \in \pi$, and let $v = u(\eta) - (\mu - \Ad(\rho(\eta))\mu)$. One then finds by application of the cocycle condition that $v$ is $\pi$-invariant, and obviously $[u] = [u_v]$, so $u \in \im \delta$.
\end{proof}
We can link the group $H^1_{\operatorname{par}}(\pi_f, \Ad(\rho,g))$ with the coefficients occurring in the statement of the Growth Rate Conjecture \ref{conj:GR} and finally conclude with

\begin{proof}[Proof of Theorem \ref{GRT}]
It suffices to remark that it follows from (\ref{endresult}) that $d_{\gamma}$ equals the dimension of $\gamma$-component of $\mathcal{M}_{\overline{\alpha}}$, which we claim for generic $[\rho]$ of that component equals the dimension of the $1$-eigenspace of $T_{[\rho]}\mathcal{M}_{\overline{\alpha}}$. Hence the result follows from Proposition \ref{goodjobsoeren}.

To see the claim, let $M^\gamma$ be a fixed $f^*$-fixed component, let $[\rho]$ denote a smooth point in $M^\gamma$, and let $g$ be as before. We will show that $T_{[\rho]}M^\gamma \isom E_1(f^*,\rho)$.

Let $\alpha \mapsto [\rho_t(\alpha)] = [e^{tu(\alpha) + O(t^2)}\rho(\alpha)]$, $\alpha \in \pi$, denote a smooth path through $[\rho]$, completely contained in $M^\gamma$. The tangent vector at $t = 0$ is $u$, which in general is an element of $H^1_{\mathrm{par}}(\pi, \Ad \rho)$, and we claim that $u \in E_1(f^*,\rho)$. For each $t$, choose $g_t \in K$ so that $g_t\rho_t(f_*\alpha) = \rho_t(\alpha)g_t$, $g_0 = g$. Differentiating this equation and letting $t = 0$, we find that
\[
	\dot{g}_0\rho(f_*\alpha) + gu(f_*\alpha)\rho(f_*\alpha) = u(\alpha)\rho(\alpha)g + \rho(\alpha)\dot{g}_0.
\]
Letting $m = \dot{g}_0g^{-1} \in \Lie(K)$, this tells us that
\[
	u(\alpha) = \Ad(g) u(f_*\alpha) + \delta m(\alpha).
\]
On the other hand, we claim that mapping $u \in E_1(\rho,f^*)$ to $\left.\frac{d}{dt}\right|_{t=0}[e^{tu}\rho]$, we end up an element in $T_{[\rho]}M^\gamma$. Letting $\rho_t = e^{tu}\rho$, we find that to first order in $t$ and for all $\alpha \in \pi$,
\begin{align*}
	\rho_t(f_*\alpha) &= e^{t \Ad(g^{-1})u(\alpha)}\rho(f_*\alpha) = \Ad(g^{-1})(e^{t u (\alpha)}) \rho(f_*\alpha) \\
	 &= \Ad(g^{-1}) ( e^{tu(\alpha)}\rho(\alpha) ) = \Ad(g^{-1})\rho_t(\alpha),
\end{align*}
which shows that $\left.\frac{d}{dt}\right|_{t=0}[e^{tu}\rho] \in T_{[\rho]}M^\gamma$.

\end{proof}

\appendix

\section{The Chern--Simons functional of manifolds with links}
\label{chernsimonsboundaryappendix}
\subsection{Connections on $3$-manifolds with links}
\label{chernsimonsboundaryconnectionssection}
In this section, we discuss how to make sense of Chern--Simons values for flat connections on $3$-manifolds containing coloured framed links. The Chern--Simons functionals will be defined an the moduli space of all flat connections on the link complement, but will depend on the framing of the link.  Throughout, $\bar{\alpha} = (\alpha_1,\dots,\alpha_n)$ denotes a tuple of elements in the Weyl alcove (at level $1$), not necessarily in the weight lattice.

We begin by choosing a tubular neighbourhood around each link component $L_i$, with coordinates $(r,\theta_1, \theta_2)$, such that $(r, \theta_1)$ are polar coordinates in the normal direction at each point, and such that $\theta_2$ parametrises $L_i$.  We want this choice to be adapted to the framing, in the sense that the framing is determined by the radial direction $(1,0,\theta_2)$.
Consider the set ${\mathcal A}_{\operatorname{st}}$ of smooth connections  on $X \setminus L$, which are of exactly the following form in these neighbourhoods of the $L_i$:

$$\nabla= d+ \xi_{i,1} d\theta_1 + \xi_{i,2} d\theta_2.$$

Of course we also want that $ \exp(\xi_{i,1})$ lies in the conjucacy class of $\exp(\alpha_i)$ we have associated to $L$. We observe that the space of smooth gauge transformations ${\mathcal G}_{\operatorname{st}}$, which are constant in the neighbourhoods of each component $L_i$ acts on  ${\mathcal A}_{\operatorname{st}}$. We observe that any flat connection on $X\setminus L$ with the required holonomy around $L$ is gauge equivalent to one of these, e.g. by using the equivalence of the moduli space of flat connections on a tubular neighbourhood of $L_i$ minus $L_i$ is given by representations of $\pi_1$ of that neighboorhood and that one can get all representations from connections of the above form. We also observe that two flat such connections are gauge equivalent iff they are equivalent under ${\mathcal G}_{\operatorname{st}}$. Now it is clear that the Chern--Simons functional is well defined on ${\mathcal A}_{\operatorname{st}}$, since the support of the Chern--Simons $3$-form has compact support and the functional is invariant mod integers under ${\mathcal G}_{\operatorname{st}}$.

\subsection{The mapping torus case}
Now, we will see that the Chern--Simons values we used in Section~\ref{csbundle} are actually the Chern--Simons values introduced above.

More precisely, assume that we are in the setup of Lemma~\ref{tracelem}. That is, let $f : \Sigma \to \Sigma$ be an element of $\operatorname{Diff}_{+}(\Sigma, \overline{z},\overline{\alpha})$ (as in Section \ref{diffeos}).
In particular, $f$ preserves the disjoint union $D$ of the disks $D_i$ around the marked points in $\cP$. Assume that $[\nabla_A] \in \mathcal{M}_{\overline{\alpha}}$ is a fixed point of $f$ and let $g \in \mathcal{G}_{0,\delta}$ such that ${\nabla_A}^g = f^*\nabla_A$. Let $g_t$ be a path in $\widetilde{\mathcal{G}}_{0,\delta}$ with $g_0 = g$, $g_1 = e$. Now, choose a proper subset $D_{1/2} \subseteq \mathrm{int}(D)$, each component containing a marked point and assume furthermore, without loss of generality, that $g_t|_{D}$ is the identity for all $t$.

Let $\widetilde{\nabla_A}$ be the connection on $\Sigma_f = [0,1] \times \Sigma^o/\sim$ given by $\widetilde{\nabla_A}|_{\Sigma^o \times \{t\}} = \nabla_A^{g_t}$. Choose a smooth cut-off function $h : \Sigma^o \to [0,1]$ with $h|_{D_{1/2}} = 0$, $h|_{D^c} = 1$, and define $h : \Sigma_f \to [0,1]$ by $h(t,x) = h(x)$.

As in Section~\ref{csbundle}, it makes sense to talk about the Chern--Simons value of $\widetilde{\nabla_A}$ independently of the previous discussion.
\begin{lem}
  Write $\widetilde{A} = \pi^*A$ and $\widetilde{A_0} = \pi^* A_0$. Then
  \[
    \CS(\widetilde{A} + \widetilde{A_0}) = \CS(h\widetilde{A} + \widetilde{A_0}).
  \]
\end{lem}
\begin{proof}
  It suffices to notice that $\CS((\widetilde{A} + \widetilde{A_0})|_{[0,1] \times D}) = \CS((h\widetilde{A} + \widetilde{A_0})|_{[0,1] \times D}) = 0$. This is the case as for both connections, the $t$-derivative vanishes by choice of $g_t$.
\end{proof}
Now, everything has been set up for the following result to hold.
\begin{lem}
  \label{correspondencewitholdCS}
  By taking the natural tubular neighbourhood $N(L) = [0,1] \times D_{1/2}/\sim$ of $L$, adapted to the framing by definition, then $(h\widetilde{A} + \widetilde{A_0})|_{[0,1] \times D_{1/2}^c}$ extends to a connection in $\mathcal{A}_{\mathrm{st}}$, whose Chern--Simons functional, as defined in Section~\ref{chernsimonsboundaryconnectionssection}, agrees with the Chern--Simons functional of $\widetilde{\nabla_A}$  as defined in Section~\ref{csbundle}.
\end{lem}

\section{Dehn twist action}
\label{dehntwistappendix}
  
  As we noticed in Remark~\ref{dehntwistboundaryremark}, in general Dehn twists around marked points will act non-trivially on the Chern--Simons line bundle, even though they act trivially on the base moduli space. In this section, we evaluate the lifted action of such a Dehn twist explicitly, and as a corollary we show that the Asymptotic Expansion Conjecture also holds for mapping tori for these Dehn twists.

  Let $p_i \in \mathcal{P}$ be a given marked point with coordinate neighbourhood $(D_i,z_i)$ as in Section~\ref{two}, so that in particular $(\tau_i,\theta_i)$ denote coordinates of $D_i \setminus \{p_i\}$. Choose a smooth increasing function $f : \mathbb{R} \to [0,2\pi]$ with the property that $f(0) = 0$, $f(1) = 2\pi$, $f'(0)=f'(1) = 0$. Let $\alpha_i$ denote the Lie algebra element whose exponential is the fixed holonomy around $p_i$.  We assume the conditions of Theorem~\ref{liftunderconditions} are met, and in particular that $\lambda_i=k\alpha_i$ is a co-weight.
    
  By a Dehn twist around $p_i$ we mean the diffeomorphism (or its mapping class) $T_i : \Sigma^o \to \Sigma^o$ defined to be $(\tau_i,\theta_i) \mapsto (\tau_i,\theta_i+f(\tau_i))$ on $D_i \setminus \{p\}$ and the identity everywhere else. As a map of punctured surfaces, $T_i$ is isotopic to the identity on $\Sigma^o$ and thus acts trivially on $\mathcal{M}_{\overline{\alpha}}$. 
  \begin{prop}\label{dehntwist}
    For every point $[\nabla_A] \in \mathcal{M}_{\overline{\alpha}}$, the map induced by $T_i^*$ on $\mathcal{L}^k_{\mathrm{CS}}\big|_{[\nabla_A]}$ is given by multiplication by $\exp(-\pi ik \langle \alpha_i, \alpha_i \rangle)=\exp\left(\frac{-\pi i |\lambda_i |^2}{k}\right)$.
  \end{prop}  
\begin{proof}
In this case it suffices to calculate the Chern--Simons functional for a $g\in\mathcal{G}_{\delta}$ that matches a specific diffeomorphism  realizing the Dehn twist, and a suitable connection $\nabla_A$.  Indeed, in general the cocycle $\Theta^k$ may be ill-defined when the Chern--Simons functional does not converge for $g\in \mathcal{G}_{\delta} \setminus \mathcal{G}_{0,\delta}$, and then the full two-step approach to constructing the line bundle, as in \cite[\S 5]{daska-went1}, is required.  When it is well-defined however (as will be the case below), the reasoning as in \cite[Lemma 5.4]{daska-went1} goes through, and we explicitly get the lift of $g$ to the fibres of the trivial bundle over $\nabla_A$ and $\nabla_A^g\in\mathcal{A}_{\delta}$.  Since the lift (\ref{liftactionbis}) of the diffeomorphism is trivial on the fibres, we find that the Dehn twist acts by the inverse of the value the cocycle on the fibre of $\mathcal{L}^k_{\operatorname{CS}}$ over $[\nabla_A]\in\mathcal{M}_{\overline{\alpha}}$.

Let $(\tau,\theta) = (\tau_i,\theta_i)$. Without loss of generality, we can assume that $\nabla_A$ takes the form $d+\alpha_i\, d\theta$ on $D_i$.  Indeed, every $\mathcal{G}_{0,\delta}$-orbit contains a smooth connection \cite[Theorem 6.12]{poritz}, and every smooth flat connection can be put in this form \cite[Lemma 2.7]{daska-went1}.  We will now assume such a flat $\nabla_A$ chosen.
We introduce for each $t\in [0,1]$ the map $f_t : [0,\infty) \times \mathbb{R}/(2\pi \mathbb{Z}) \to [0,\infty) \times \mathbb{R}/(2\pi \mathbb{Z})$ given by $$f_t(\tau,\theta)=(\tau, \theta + (1-t)\,f(\tau)),$$ extending trivially to a map $f_t: \Sigma^o\rightarrow \Sigma^o.$  As in Section~\ref{isotopysection},  we get gauge transformations $g_t$  such that $$f_t^*\nabla_A = \nabla_A^{g_t},$$ the only difference being that the $g_t$ are now in $\mathcal{G}_{\delta}$ but not in $\mathcal{G}_{0,\delta}$.  Since $f_0=T$ we have $\nabla_A^{g_0}=T^*\nabla_A$.  Just as previously we now consider the gauge transform $\tilde{g}$ on the trivial bundle over $[0,1]\times \Sigma^o$ defined by the $g_t$, and we apply it to $\widetilde{\nabla_A}$.  It is straightforward to see that on $D_i$,
$$\widetilde{\nabla_A}^{\tilde{g}}=d+ \alpha_i\, d\theta + \alpha_i\, (1-t)f'(\tau)d\tau,$$ and hence we need to apply the Chern--Simons functional to $B=\alpha_i\, d\theta + \alpha_i\, (1-t)f'(\tau)\,d\tau$ on $D_i$.  Of course $B\wedge B\wedge B=0$, and by direct calculation we find $$B\wedge dB= - \langle \alpha_i,\alpha_i\rangle f'(\tau)\, dt \wedge d\tau \wedge d\theta.$$ Now since $g_t$ is trivial outside $D_i$ we can omit all but $[0,1]\times D_i$ from the integrand, and hence we have $$\operatorname{CS}(B)=\frac{-|\alpha_i |^2 }{8\pi^2} \int_{\tau=0}^{\infty}\int_{\theta=0}^{2\pi}  \int_{t=0}^1 f'(\tau)\, dt \wedge d \tau \wedge d\theta = \frac{-2\pi |\alpha_i |^2}{8\pi^2 }\int_{\tau=0}^{\infty} 
f'(\tau) \, d\tau = \frac{-|\alpha_i |^2}{2}.$$  This gives $\Theta^k(\nabla_A,g_0)=\exp\left(\frac{2\pi i k |\alpha_i|^2}{2}\right)$, from which we finally conclude that $T$ acts by  $$\exp\left(- \pi ik |\alpha_i|^2\right)=\exp\left(\frac{-\pi i |\lambda_i |^2}{k}\right)$$ on the fibres of $\mathcal{L}^k_{\operatorname{CS}}$ over all of $\mathcal{M}_{\overline{\alpha}}$.
\end{proof}

\begin{rem}
\label{cftdehntwist}
When $\lambda^{(k)} = k\alpha \in \Lambda^{(k)}_K$, this agrees with the action of $T_i$ in conformal field theory; here, $T_i$ acts on conformal blocks by multiplication by a $\lambda$-dependent root of unity $T^{(k)}_{\lambda \lambda}$, an entry of the so-called $T$-matrix, cf. e.g. \cite{gepner-witten}, \cite{kac}.
\end{rem}

\begin{prop}
	The Asymptotic Expansion and Growth Rate Conjectures hold for $(\Sigma_f,L)$ obtained from mapping tori of $f \in \langle T_i \mid i = 1, \dots, n \rangle$ with $\bar{\lambda}^{(k)}$ as in Corollary~\ref{aecthm}.
\end{prop}
\begin{proof}
	That the quantum invariants in this case have asymptotic expansions of the desired form follows from Remark~\ref{cftdehntwist}. That the phases occuring are the relevant Chern--Simons values follows from the proof of Proposition~\ref{dehntwist}; see in particular \cite[App. A]{gepner-witten}.
	
	As follows also by Remark~\ref{cftdehntwist}, the growth rate of the quantum invariants is the growth rate of the spaces of conformal blocks, which on the other hand is the dimension of the moduli space, i.e. the dimension of the fixed point set of $T$, so the Growth Rate Conjecture follows from Proposition~\ref{goodjobsoeren} as in the proof of Theorem~\ref{GRT}.
\end{proof}

\providecommand{\MR}[1]{}\def\Dbar{\leavevmode\lower.6ex\hbox to
  0pt{\hskip-.23ex \accent"16\hss}D} \def\cprime{$'$}
  \def\dbar{\leavevmode\hbox to 0pt{\hskip.2ex \accent"16\hss}d}
  \def\cprime{$'$}
\providecommand{\bysame}{\leavevmode\hbox to3em{\hrulefill}\thinspace}
\providecommand{\MR}{\relax\ifhmode\unskip\space\fi MR }
\providecommand{\MRhref}[2]{
  \href{http://www.ams.org/mathscinet-getitem?mr=#1}{#2}
}
\providecommand{\href}[2]{#2}


\begin{thebibliography}{10}

\bibitem{Alper}
Jarod Alper, \emph{Good moduli spaces for {A}rtin stacks}, Ann. Inst. Fourier
  (Grenoble) \textbf{63} (2013), no.~6, 2349--2402.

\bibitem{andersen1998_NewPolarizations}
J{\o}rgen~Ellegaard Andersen, \emph{New polarizations on the moduli spaces and
  the {T}hurston compactification of {T}eichm\"uller space}, Internat. J. Math.
  \textbf{9} (1998), no.~1, 1--45. \MR{1612326 (99e:58029)}

\bibitem{andersen2006_Asymptoticfaithfulness}
\bysame, \emph{Asymptotic faithfulness of the quantum {${\rm SU}(n)$}
  representations of the mapping class groups}, Ann. of Math. (2) \textbf{163}
  (2006), no.~1, 347--368. \MR{2195137 (2007c:53131)}

\bibitem{andersen95}
\bysame, \emph{The {W}itten-{R}eshetikhin-{T}uraev invariants of finite order
  mapping tori {I}}, J. Reine Angew. Math. \textbf{681} (2013), 1--38.
  \MR{3181488}

\bibitem{andersen-himpel2011}
J{\o}rgen~Ellegaard Andersen and Benjamin Himpel, \emph{The
  {W}itten-{R}eshetikhin-{T}uraev invariants of finite order mapping tori
  {II}}, Quantum Topol. \textbf{3} (2012), no.~3-4, 377--421. \MR{2928090}

\bibitem{andersen-ueno2007}
J{\o}rgen~Ellegaard Andersen and Kenji Ueno, \emph{Abelian conformal field
  theory and determinant bundles}, Internat. J. Math. \textbf{18} (2007),
  no.~8, 919--993. \MR{MR2339577}

\bibitem{AU2}
\bysame, \emph{Geometric construction of modular functors from conformal field
  theory}, J. Knot Theory Ramifications \textbf{16} (2007), no.~2, 127--202.
  \MR{MR2306213}

\bibitem{AU4}
\bysame, \emph{Modular functors are determined by their genus zero data},
  Quantum Topol. \textbf{3} (2012), no.~3-4, 255--291. \MR{2928086}

\bibitem{AU5}
\bysame, \emph{Construction of the {W}itten-{R}eshetikhin-{T}uraev {TQFT} from
  conformal field theory}, Invent. Math. \textbf{201} (2015), no.~2, 519--559.
  \MR{3370620}

\bibitem{atiyah}
M.~F. Atiyah, \emph{Vector bundles over an elliptic curve}, Proc. London Math.
  Soc. (3) \textbf{7} (1957), 414--452. \MR{0131423 (24 \#A1274)}

\bibitem{ab}
M.~F. Atiyah and R.~Bott, \emph{The {Y}ang-{M}ills equations over {R}iemann
  surfaces}, Philos. Trans. Roy. Soc. London Ser. A \textbf{308} (1983),
  523--615.

\bibitem{ahs}
M.~F. Atiyah, N.~J. Hitchin, and I.~M. Singer, \emph{Self-duality in
  four-dimensional {R}iemannian geometry}, Proc. Roy. Soc. London Ser. A
  \textbf{362} (1978), no.~1711, 425--461. \MR{506229 (80d:53023)}

\bibitem{APS}
M.~F. Atiyah, V.~K. Patodi, and I.~M. Singer, \emph{Spectral asymmetry and
  {R}iemannian geometry. {I}}, Math. Proc. Cambridge Philos. Soc. \textbf{77}
  (1975), 43--69. \MR{0397797 (53 \#1655a)}

\bibitem{axelrod-dellapietra-witten91}
Scott Axelrod, Steve Della~Pietra, and Edward Witten, \emph{Geometric
  quantization of {C}hern-{S}imons gauge theory}, J. Differential Geom.
  \textbf{33} (1991), no.~3, 787--902. \MR{MR1100212 (92i:58064)}

\bibitem{baum-fulton-macpherson1979_Riemann-Roch}
Paul Baum, William Fulton, and Robert MacPherson, \emph{Riemann-{R}och and
  topological {$K$} theory for singular varieties}, Acta Math. \textbf{143}
  (1979), no.~3-4, 155--192. \MR{549773 (82c:14021)}

\bibitem{baum-fulton-quart_LefschetzRiemann-Roch}
Paul Baum, William Fulton, and George Quart, \emph{Lefschetz-{R}iemann-{R}och
  for singular varieties}, Acta Math. \textbf{143} (1979), no.~3-4, 193--211.
  \MR{549774 (82c:14022a)}

\bibitem{beauville}
Arnaud Beauville, \emph{Conformal blocks, fusion rules and the {V}erlinde
  formula}, Proceedings of the {H}irzebruch 65 {C}onference on {A}lgebraic
  {G}eometry ({R}amat {G}an, 1993), Israel Math. Conf. Proc., vol.~9, Bar-Ilan
  Univ., Ramat Gan, 1996, pp.~75--96. \MR{1360497 (97f:17025)}

\bibitem{beaulas}
Arnaud Beauville and Yves Laszlo, \emph{Conformal blocks and generalized theta
  functions}, Comm. Math. Phys. \textbf{164} (1994), no.~2, 385--419.
  \MR{1289330 (95k:14011)}

\bibitem{biquard}
Olivier Biquard, \emph{Fibr\'es paraboliques stables et connexions
  singuli\`eres plates}, Bull. Soc. Math. France \textbf{119} (1991), no.~2,
  231--257. \MR{1116847 (93a:58039)}

\bibitem{biswas}
I.~Biswas and K.~Guruprasad, \emph{Principal bundles on open surfaces and
  invariant functions on {L}ie groups}, Internat. J. Math. \textbf{4} (1993),
  no.~4, 535--544. \MR{1232983 (94i:58033)}

\bibitem{biswas-ragha}
I.~Biswas and N.~Raghavendra, \emph{Determinants of parabolic bundles on
  {R}iemann surfaces}, Proc. Indian Acad. Sci. Math. Sci. \textbf{103} (1993),
  no.~1, 41--71. \MR{1234199 (94i:58206)}

\bibitem{bisw3}
\bysame, \emph{Curvature of the determinant bundle and the {K}\"ahler form over
  the moduli of parabolic bundles for a family of pointed curves}, Asian J.
  Math. \textbf{2} (1998), no.~2, 303--324. \MR{1639556 (99k:32032)}

\bibitem{bisw2}
Indranil Biswas, \emph{Determinant line bundle on moduli space of parabolic
  bundles}, Ann. Global Anal. Geom. \textbf{40} (2011), no.~1, 85--94.
  \MR{2795451 (2012g:14010)}

\bibitem{BHMV1}
C.~Blanchet, N.~Habegger, G.~Masbaum, and P.~Vogel, \emph{Three-manifold
  invariants derived from the {K}auffman bracket}, Topology \textbf{31} (1992),
  no.~4, 685--699. \MR{MR1191373 (94a:57010)}

\bibitem{BHMV2}
\bysame, \emph{Topological quantum field theories derived from the {K}auffman
  bracket}, Topology \textbf{34} (1995), no.~4, 883--927. \MR{MR1362791
  (96i:57015)}

\bibitem{Blanchet}
Christian Blanchet, \emph{Hecke algebras, modular categories and
  {$3$}-manifolds quantum invariants}, Topology \textbf{39} (2000), no.~1,
  193--223. \MR{MR1710999 (2000i:57020)}

\bibitem{boden-hu}
Hans~U. Boden and Yi~Hu, \emph{Variations of moduli of parabolic bundles},
  Math. Ann. \textbf{301} (1995), no.~3, 539--559. \MR{1324526 (96f:14012)}

\bibitem{bodyok}
Hans~U. Boden and K{\^o}ji Yokogawa, \emph{Rationality of moduli spaces of
  parabolic bundles}, J. London Math. Soc. (2) \textbf{59} (1999), no.~2,
  461--478. \MR{1709179 (2000g:14043)}

\bibitem{chang}
Sheldon Xu-{D}ong Chang, \emph{Determinant line bundles and {R}iemann surfaces
  with boundaries}, Preprint.

\bibitem{charles}
Laurent Charles, \emph{A note on {C}hern-{S}imons bundles and the mapping class
  group}, 2010, Preprint.

\bibitem{toric}
David~A. Cox, John~B. Little, and Henry~K. Schenck, \emph{Toric varieties},
  Graduate Studies in Mathematics, vol. 124, American Mathematical Society,
  Providence, RI, 2011. \MR{2810322 (2012g:14094)}

\bibitem{daska-went2}
Georgios Daskalopoulos and Richard Wentworth, \emph{Factorization of rank two
  theta functions. {II}. {P}roof of the {V}erlinde formula}, Math. Ann.
  \textbf{304} (1996), no.~1, 21--51. \MR{1367881 (97f:32018)}

\bibitem{daska-went1}
\bysame, \emph{Geometric quantization for the moduli space of vector bundles
  with parabolic structure}, Geometry, topology and physics ({C}ampinas, 1996),
  de Gruyter, Berlin, 1997, pp.~119--155. \MR{1605216 (99k:32033)}

\bibitem{daska-uhl}
Georgios~D. Daskalopoulos and Karen~K. Uhlenbeck, \emph{An application of
  transversality to the topology of the moduli space of stable bundles},
  Topology \textbf{34} (1995), no.~1, 203--215. \MR{1308496 (96e:32017)}

\bibitem{debar}
Olivier Debarre, \emph{Higher-dimensional algebraic geometry}, Universitext,
  Springer-Verlag, New York, 2001. \MR{1841091 (2002g:14001)}

\bibitem{dolg-hu}
Igor~V. Dolgachev and Yi~Hu, \emph{Variation of geometric invariant theory
  quotients}, Inst. Hautes \'Etudes Sci. Publ. Math. (1998), no.~87, 5--56,
  With an appendix by Nicolas Ressayre. \MR{1659282 (2000b:14060)}

\bibitem{donaldson}
S.~K. Donaldson, \emph{A new proof of a theorem of {N}arasimhan and
  {S}eshadri}, J. Differential Geom. \textbf{18} (1983), no.~2, 269--277.
  \MR{710055 (85a:32036)}

\bibitem{DZ}
J.-M. Drezet and M.~S. Narasimhan, \emph{Groupe de {P}icard des vari\'et\'es de
  modules de fibr\'es semi-stables sur les courbes alg\'ebriques}, Invent.
  Math. \textbf{97} (1989), no.~1, 53--94. \MR{999313 (90d:14008)}

\bibitem{freed95}
Daniel~S. Freed, \emph{Classical {C}hern-{S}imons theory. {I}}, Adv. Math.
  \textbf{113} (1995), no.~2, 237--303. \MR{MR1337109 (96h:58019)}

\bibitem{fujita}
Hajime Fujita, \emph{On the functoriality of the {C}hern-{S}imons line bundle
  and the determinant line bundle}, Commun. Contemp. Math. \textbf{8} (2006),
  no.~6, 715--735. \MR{2274939 (2007k:58037)}

\bibitem{fult-laz}
William Fulton and Robert Lazarsfeld, \emph{Connectivity and its applications
  in algebraic geometry}, Algebraic geometry ({C}hicago, {I}ll., 1980), Lecture
  Notes in Math., vol. 862, Springer, Berlin, 1981, pp.~26--92. \MR{644817
  (83i:14002)}

\bibitem{gepner-witten}
Doron Gepner and Edward Witten, \emph{String theory on group manifolds},
  Nuclear Phys. B \textbf{278} (1986), no.~3, 493--549. \MR{862896 (88h:81169)}

\bibitem{gold}
William~M. Goldman, \emph{The symplectic nature of fundamental groups of
  surfaces}, Adv. in Math. \textbf{54} (1984), no.~2, 200--225. \MR{MR762512
  (86i:32042)}

\bibitem{hitchin90}
N.~J. Hitchin, \emph{Flat connections and geometric quantization}, Comm. Math.
  Phys. \textbf{131} (1990), no.~2, 347--380. \MR{MR1065677 (91g:32022)}

\bibitem{jantzen}
Jens~Carsten Jantzen, \emph{Representations of algebraic groups}, second ed.,
  Mathematical Surveys and Monographs, vol. 107, American Mathematical Society,
  Providence, RI, 2003. \MR{2015057 (2004h:20061)}

\bibitem{kac}
Victor~G. Kac, \emph{Infinite-dimensional {L}ie algebras}, third ed., Cambridge
  University Press, Cambridge, 1990. \MR{1104219 (92k:17038)}

\bibitem{kob-nom}
Shoshichi Kobayashi and Katsumi Nomizu, \emph{Foundations of differential
  geometry. {V}ol. {I}}, Wiley Classics Library, John Wiley \& Sons Inc., New
  York, 1996, Reprint of the 1963 original, A Wiley-Interscience Publication.
  \MR{1393940 (97c:53001a)}

\bibitem{konno1}
Hiroshi Konno, \emph{On the natural line bundle on the moduli space of stable
  parabolic bundles}, Comm. Math. Phys. \textbf{155} (1993), no.~2, 311--324.
  \MR{1230031 (95c:32018)}

\bibitem{konno2}
\bysame, \emph{Geometric quantization of the moduli space of parabolic
  bundles}, Topology, geometry and field theory, World Sci. Publ., River Edge,
  NJ, 1994, pp.~79--87. \MR{1312174 (96a:58036)}

\bibitem{kumar}
Shrawan Kumar, \emph{Demazure character formula in arbitrary {K}ac-{M}oody
  setting}, Invent. Math. \textbf{89} (1987), no.~2, 395--423. \MR{894387
  (88i:17018)}

\bibitem{KNR}
Shrawan Kumar, M.~S. Narasimhan, and A.~Ramanathan, \emph{Infinite
  {G}rassmannians and moduli spaces of {$G$}-bundles}, Math. Ann. \textbf{300}
  (1994), no.~1, 41--75. \MR{1289830 (96e:14011)}

\bibitem{laszlo-sorger}
Yves Laszlo and Christoph Sorger, \emph{The line bundles on the moduli of
  parabolic {$G$}-bundles over curves and their sections}, Ann. Sci. \'Ecole
  Norm. Sup. (4) \textbf{30} (1997), no.~4, 499--525. \MR{1456243 (98f:14007)}

\bibitem{looij}
Eduard Looijenga, \emph{From {WZW} models to modular functors}, Handbook of
  moduli. {V}ol. {II}, Adv. Lect. Math. (ALM), vol.~25, Int. Press, Somerville,
  MA, 2013, pp.~427--466. \MR{3184182}

\bibitem{mathieu1}
Olivier Mathieu, \emph{Formules de {D}emazure-{W}eyl, et g\'en\'eralisation du
  th\'eor\`eme de {B}orel-{W}eil-{B}ott}, C. R. Acad. Sci. Paris S\'er. I Math.
  \textbf{303} (1986), no.~9, 391--394. \MR{862200 (87m:17036)}

\bibitem{mathieu2}
\bysame, \emph{Formules de caract\`eres pour les alg\`ebres de {K}ac-{M}oody
  g\'en\'erales}, Ast\'erisque (1988), no.~159-160, 267. \MR{980506
  (90d:17024)}

\bibitem{matic}
Gordana Mati{\'c}, \emph{{${\rm SO}(3)$}-connections and rational homology
  cobordisms}, J. Differential Geom. \textbf{28} (1988), no.~2, 277--307.
  \MR{961516 (89m:57035)}

\bibitem{mehtaseshadri}
V.~B. Mehta and C.~S. Seshadri, \emph{Moduli of vector bundles on curves with
  parabolic structures}, Math. Ann. \textbf{248} (1980), no.~3, 205--239.
  \MR{575939 (81i:14010)}

\bibitem{meinwood}
E.~Meinrenken and C.~Woodward, \emph{Hamiltonian loop group actions and
  {V}erlinde factorization}, J. Differential Geom. \textbf{50} (1998), no.~3,
  417--469. \MR{1690736 (2000g:53103)}

\bibitem{pauly}
Christian Pauly, \emph{Espaces de modules de fibr\'es paraboliques et blocs
  conformes}, Duke Math. J. \textbf{84} (1996), no.~1, 217--235. \MR{1394754
  (97h:14022)}

\bibitem{poritz}
Jonathan~A. Poritz, \emph{Parabolic vector bundles and
  {H}ermitian-{Y}ang-{M}ills connections over a {R}iemann surface}, Internat.
  J. Math. \textbf{4} (1993), no.~3, 467--501. \MR{1228585 (94i:32046)}

\bibitem{prasad}
P.~K. Prasad, \emph{Cohomology of {F}uchsian groups}, J. Indian Math. Soc.
  (N.S.) \textbf{38} (1974), no.~1-4, 51--63 (1975). \MR{0390076 (52 \#10902)}

\bibitem{quillen}
Daniel Quillen, \emph{Determinants of {C}auchy-{R}iemann operators over a
  {R}iemann surface}, Functional Analysis and Its Applications \textbf{19}
  (1985), no.~1, 31--34 (English).

\bibitem{rsw}
T.~R. Ramadas, I.~M. Singer, and J.~Weitsman, \emph{Some comments on
  {C}hern-{S}imons gauge theory}, Comm. Math. Phys. \textbf{126} (1989), no.~2,
  409--420. \MR{1027504 (90m:58218)}

\bibitem{RT1}
N.~Yu. Reshetikhin and V.~G. Turaev, \emph{Ribbon graphs and their invariants
  derived from quantum groups}, Comm. Math. Phys. \textbf{127} (1990), no.~1,
  1--26. \MR{MR1036112 (91c:57016)}

\bibitem{RT2}
\bysame, \emph{Invariants of {$3$}-manifolds via link polynomials and quantum
  groups}, Invent. Math. \textbf{103} (1991), no.~3, 547--597. \MR{MR1091619
  (92b:57024)}

\bibitem{scheinost}
Peter Scheinost and Martin Schottenloher, \emph{Metaplectic quantization of the
  moduli spaces of flat and parabolic bundles}, J. Reine Angew. Math.
  \textbf{466} (1995), 145--219. \MR{1353317 (97m:14012)}

\bibitem{sorger-lectures}
Christoph Sorger, \emph{Lectures on moduli of principal {$G$}-bundles over
  algebraic curves}, School on {A}lgebraic {G}eometry ({T}rieste, 1999), ICTP
  Lect. Notes, vol.~1, Abdus Salam Int. Cent. Theoret. Phys., Trieste, 2000,
  pp.~1--57. \MR{1795860 (2002h:14017)}

\bibitem{tak-zog}
Leon~A. Takhtajan and Peter Zograf, \emph{The first {C}hern form on moduli of
  parabolic bundles}, Math. Ann. \textbf{341} (2008), no.~1, 113--135.
  \MR{2377472 (2008k:58073)}

\bibitem{taubes}
Clifford~Henry Taubes, \emph{Gauge theory on asymptotically periodic
  {$4$}-manifolds}, J. Differential Geom. \textbf{25} (1987), no.~3, 363--430.
  \MR{882829 (88g:58176)}

\bibitem{teleman}
Constantin Teleman, \emph{The quantization conjecture revisited}, Ann. of Math.
  (2) \textbf{152} (2000), no.~1, 1--43. \MR{1792291 (2002d:14073)}

\bibitem{thaddeus}
Michael Thaddeus, \emph{Geometric invariant theory and flips}, J. Amer. Math.
  Soc. \textbf{9} (1996), no.~3, 691--723. \MR{1333296 (96m:14017)}

\bibitem{coordfree}
Yoshifumi Tsuchimoto, \emph{On the coordinate-free description of the conformal
  blocks}, J. Math. Kyoto Univ. \textbf{33} (1993), no.~1, 29--49. \MR{1203889
  (95c:14023)}

\bibitem{TUY}
Akihiro Tsuchiya, Kenji Ueno, and Yasuhiko Yamada, \emph{Conformal field theory
  on universal family of stable curves with gauge symmetries}, Integrable
  systems in quantum field theory and statistical mechanics, Adv. Stud. Pure
  Math., vol.~19, Academic Press, Boston, MA, 1989, pp.~459--566. \MR{1048605
  (92a:81191)}

\bibitem{tu}
Loring~W. Tu, \emph{Semistable bundles over an elliptic curve}, Adv. Math.
  \textbf{98} (1993), no.~1, 1--26. \MR{1212625 (94a:14008)}

\bibitem{WenzlTuraev2}
V.~Turaev and H.~Wenzl, \emph{Quantum invariants of {$3$}-manifolds associated
  with classical simple {L}ie algebras}, Internat. J. Math. \textbf{4} (1993),
  no.~2, 323--358. \MR{1217386 (94i:57019)}

\bibitem{WenzlTuraev1}
\bysame, \emph{Semisimple and modular categories from link invariants}, Math.
  Ann. \textbf{309} (1997), no.~3, 411--461. \MR{1474200 (98j:18012)}

\bibitem{Turaev}
Vladimir~G. Turaev, \emph{Quantum invariants of knots and 3-manifolds}, revised
  ed., de Gruyter Studies in Mathematics, vol.~18, Walter de Gruyter \& Co.,
  Berlin, 2010. \MR{2654259 (2011f:57023)}

\bibitem{Walker}
Kevin Walker, \emph{On {W}itten's 3-manifold {I}nvariants}, {P}reprint,
  {P}reliminary {V}ersion \#2 (1991).

\bibitem{witten-jones}
Edward Witten, \emph{Quantum field theory and the {J}ones polynomial}, Comm.
  Math. Phys. \textbf{121} (1989), no.~3, 351--399. \MR{990772 (90h:57009)}

\end{thebibliography}
\end{document}